\documentclass[11pt,reqno,a4paper]{amsart}

\usepackage{a4,amssymb}

\usepackage{bbm}
\usepackage{hyperref}

\allowdisplaybreaks   

\theoremstyle{plain}
\newtheorem{de}{Definition}[section]
\newtheorem{lem}[de]{Lemma}
\newtheorem{prop}[de]{Proposition}
\newtheorem{cor}[de]{Corollary}
\newtheorem{thm}[de]{Theorem}

\theoremstyle{definition}
\newtheorem{rem}[de]{Remark}

\newcommand{\supp}{\operatorname{supp}} 
\newcommand{\dist}{\operatorname{dist}}

\DeclareMathOperator{\curl}{curl}

\DeclareMathOperator{\tr}{tr}
\DeclareMathOperator{\DDiv}{Div}
\DeclareMathOperator{\ran}{ran}

\newcommand{\RR}{{\mathbb{R}}}
\newcommand{\NN}{{\mathbb{N}}}

\newcommand{\ZZ}{{\mathbb{Z}}}

\newcommand{\cF}{{\mathcal{F}}}
\newcommand{\cA}{{\mathcal{A}}}
\newcommand{\cB}{{\mathcal{B}}}

\newcommand{\cH}{{\mathcal{H}}}
\newcommand{\cU}{{\mathcal{U}}}

\newcommand{\ph}{\varphi}
\newcommand{\ep}{\varepsilon}

\newcommand{\sub}{\subseteq}
\newcommand{\hra}{\hookrightarrow}
\newcommand{\D}{\mathrm{d}}
\newcommand{\dd}{\,\mathrm{d}}
\newcommand{\e}{\mathrm{e}}
\newcommand{\ii}{\mathrm{i}}
\newcommand{\loc}{\mathrm{loc}}
\newcommand{\sym}{\mathrm{sym}}

\newcommand{\tdiv}{\mathrm{div}}
\newcommand{\tcurl}{\mathrm{curl}}
\newcommand{\co}{\mathrm{co}}

\newcommand{\cf}{\mathrm{cf}}
\newcommand{\cp}{\mathrm{cp}}
\newcommand{\lin}{\mathrm{lin}}
\newcommand{\nl}{\mathrm{nl}}
\newcommand{\ta}{\mathrm{ta}}

\newcommand{\ol}{\overline}

\newcommand{\wt}{\widetilde}

\newcommand{\E}{\RR^3_+}

\newcommand{\Fnorm}[2]{\|#2\|_{F^{#1}(\Omega)}}
\newcommand{\Fvarnorm}[2]{\|#2\|_{F^{{#1},0}(\RR^3_+)}}
\newcommand{\Hh}[1]{\cH^{#1}(\RR^3_+)}
\newcommand{\Hhn}[2]{\|#2\|_{\Hh{#1}}}

\numberwithin{equation}{section}

\newcommand{\beq}{\begin{equation}}
\newcommand{\eeq}{\end{equation}}

\begin{document}

\title{Local wellposedness of quasilinear Maxwell equations with absorbing boundary conditions}

\author{Roland Schnaubelt}
\author{Martin Spitz}

\address{Department of Mathematics,
Karlsruhe Institute of Technology, 76128 Karlsruhe, Germany.}
\email{schnaubelt@kit.edu}
\email{martin.spitz@kit.edu}
\thanks{We gratefully acknowledge financial support by the Deutsche Forschungsgemeinschaft (DFG)
through CRC 1173.}




\begin{abstract}
In this article we provide a local wellposedness theory for quasilinear Maxwell equations with absorbing boundary
conditions in $\cH^m$ for $m \geq 3$. The Maxwell equations are equipped with instantaneous  nonlinear material laws
leading to a quasilinear symmetric hyperbolic first order system. We consider both linear and nonlinear absorbing boundary conditions.
We show existence and uniqueness of a local solution, provide a blow-up criterion in the  Lipschitz norm,
and prove the continuous dependence on the data. In the case of nonlinear boundary conditions we need a 
smallness assumption on the tangential trace of the solution.
 The proof is based on detailed apriori estimates and the regularity theory for the corresponding linear 
 problem which we also develop here.
\end{abstract}

\maketitle


\section{Introduction}\label{sec:intro}
 The Maxwell system is the foundation of electromagnetism and thus one of the core partial differential equations
 in physics. For nonlinear instantaneous material laws, it can be written as a symmetric hyperbolic system under natural
assumptions. On the full space $\RR^d$, for such systems  a  satisfactory local wellposedness theory in 
$\cH^s(\RR^d)$ for $s > 1+\frac{d}{2}$ is provided by Kato's work \cite{Ka}. On domains $G\subsetneq\RR^3$ the Maxwell 
system is characteristic, and with its standard boundary conditions  it does not fit into 
the classes of hyperbolic problems for which one has a local wellposedness theory in $\cH^3$. The available results are stated 
in Sobolev spaces of much higher order and with weights encoding a loss of derivatives in normal direction,  see~\cite{Gu} or \cite{Se}.
In the very recent papers \cite{Sp1} and \cite{Sp2} by one of the authors, an encompassing local wellposedness theory in $\cH^m$ with $m\ge 3$ 
was derived for the Maxwell system endowed with  perfectly conducting boundaries, and it has been extended to interface problems in \cite{SS}.

In this paper we treat the quasilinear Maxwell system \eqref{eq:maxwell0} with absorbing boundary conditions which occur if one has a finite, 
strictly positive conductivity at the boundary, see \cite{FM}.
We establish a comprehensive local wellposedness theory in $\cH^m$ with $m\ge 3$ for linear boundary conditions and also treat nonlinear ones 
under a smallness condition (which automatically holds in the linear case). Our result provides the 
framework to show global existence 
and exponential decay of the solutions if the initial data are small, see the companion paper \cite{PS} co-authored by one of us.

For such boundary conditions, local solutions in $\cH^3$ were already constructed in \cite{PZ} under a similar smallness assumption. However, 
neither uniqueness, nor blow-up criteria, nor the continuous dependence on data were  addressed in this paper. These results (and ours below) 
rely on a regularity  theory for the linearized non-autonomous problem. It seems to us that the corresponding estimates
in \cite{PZ} were not precise enough to show uniqueness of the nonlinear  problem and to treat its wellposedness theory. The crucial problem 
in this respect is to derive differentiability in normal direction to the boundary, whereas tangential regularity can be shown in much greater 
generality, see \cite{Ra2}. In \cite{CE} and \cite{MO} this difficulty was solved on the linear level by transforming the system in a 
non-characteristic one, but the resulting estimates do not fit to the fixed point argument for the nonlinear system, as already 
observed in \cite{PZ} concerning \cite{MO}. We note that \cite{MO}  deals with a far more general situation. 

In this work we study the Maxwell system
\begin{align} \label{eq:maxwell0}
  \begin{aligned}
\partial_t \theta_e(E,H)&= \curl H -\sigma_e(E,H)E -J_0, \qquad &&x\in G, \quad &t\ge t_0,\\
 \partial_t \theta_m(E, H)&= -\curl E, &&x\in G, &t\ge t_0, \\
  H\times \nu &= \nu\times (\zeta(E\times \nu) E\times \nu)+g, &&x\in \Sigma, &t\ge t_0, \\
 E(t_0)&=E_0,  \qquad H(t_0)=H_0, &&x\in G,
 \end{aligned}
\end{align}
for an initial time $t_0 \in \RR$, an open subset $G$ of $\RR^3$ with a smooth compact boundary $\Sigma$, and the unit outward normal $\nu$.
We look for the electric and magnetic fields $(E(t,x), H(t,x))\in\cU$, where $\cU\sub \RR^6$ is a fixed open convex set.

The material laws $\theta=(\theta_e,\theta_m):G\times \cU\to  \RR^6$ are differentiable and their derivative $\partial_{(E,H)} \theta$ is $C^m$ for some
$m\in \NN$ with $m\ge3$, and it is assumed to be symmetric and positive definite. The latter is a standard assumption already in the linear case and 
was also imposed in \cite{PZ},  for instance. It is true for isotropic nonlinearities and large classes of constitutive relations
arising in optics, see e.g.\ Example~2.1 in \cite{LPS}. We refer to \cite{Busch} and \cite{FM} for further background.
 The conductivities $\sigma_e$ and $\zeta$ are also of class $C^m$, and $\zeta$ is symmetric and 
positive definite. The given current densities $J_0$ and $g$ and the initial fields $E_0$ and $H_0$ are supposed to belong to $\cH^m$. 

Guided by the basic energy estimate \eqref{est:apriori-0}, we look for solutions 
\[ u=(E,H)\in \bigcap\nolimits_{j=0}^m C^j(\ol{J},\cH^{m-j}(G)^ 6)\]
having tangential traces in $\cH^m(J\times \Sigma)$, where $J=(t_0,T)$ for some $T>t_0$. The space of these functions is called $G^m_\Sigma$.
For such a solution the data and coefficients 
have to satisfy the compatibility conditions \eqref{eq:cc-nl}. Assuming them, in our main Theorem~\ref{thm:lwp}
we show that
\begin{enumerate}
	\item[(1)] the system~\eqref{eq:maxwell0} has a unique maximal solution in $G^m_\Sigma$ with $m \geq 3$,
	\item[(2)] blow-up can be characterized in the Lipschitz-norm, 
	\item[(3)] the solution depends continuously on the data in $\cH^m$.
\end{enumerate}

In the case of nonlinear boundary conditions we have to add a smallness assumption on the product 
$\ol{\kappa}\, |\partial_E \zeta(\cdot, E\times \nu)|$ for fields with $|E\times\nu |\le \ol{\kappa}$. So, either the boundary 
condition is close to be linear or the solution has to be uniformly small (as in \cite{PZ}). We also deal with
non-autonomous linear boundary conditions,  see \eqref{ass:main1}.
The smallness condition is enforced by the basic energy estimate \eqref{est:apriori-0}
which allows us to bound the  tangential traces of the solution in $L^2(J\times \Sigma)$ by the boundary data in the same norm, 
but with a constant which cannot be made small. This behavior reappears on higher regularity levels and spoils the fixed point 
argument if the boundary condition is nonlinear.  Still the situation is much better than for perfectly conducting boundaries 
where one may lose a derivative at the  boundary, cf.\ \cite{El}.

For linear material laws one can treat  nonlinear boundary conditions with bounded $\zeta$ even on an $L^2$--level
without a smallness condition, see e.g.\ \cite{ELN}. These results are based on the theory of  monotone operators and semigroups. 
In our setting  this seems to be impossible, also in view of blow-up examples in $\cH(\mathrm{curl})$, see~\cite{DNS}.

We follow the strategy of \cite{Sp1} and \cite{Sp2}. One freezes the solutions in the nonlinearities of the system
and solves the resulting non-autonomous linear problems via localization, duality and regularization 
with precise apriori estimates. Then local solutions are constructed via a contraction argument. 
Similar ideas had been used in \cite{PZ}, though there core parts (like the regularization procedure) were not worked out.
The improved blow-up condition and the continuous dependence on data require additional significant efforts.

We first rewrite the system \eqref{eq:maxwell0}  in 
the equivalent form of a standard hyperbolic system in Section~\ref{sec:aux}, where we also collect our notation. Moreover,
we describe the localization procedure. It is crucial for our arguments that in various steps we can partly decouple the normal direction
from the time and tangential ones. To achieve this, in the localization one has to keep the form of the boundary condition 
and the constant coefficients  in front of $\partial_3$. For perfectly conducting boundaries and interfaces
this has been discussed in \cite{Sp0} respectively \cite{SS}, so that we can focus below on the new boundary conditions.
The compatibility conditions for the   linear and nonlinear problems are  derived in \eqref{eq:cc-lin} and \eqref{eq:cc-nl}.
They differ from those in \cite{Sp1}  and \cite{Sp2} in several respects.

In Section~\ref{sec:est} we first solve the non-autonomous linear problem in $G_\Sigma^0$ with $L^2$--data and Lipschitz coefficients
and derive the basic $L^2$--estimate in Proposition~\ref{prop:L2}.
This result is known, see e.g.\ \cite{CE}, but it is hard to find complete proofs and
we need more precise information about the constants than given in e.g.\ \cite{CE}. 
So we give a sketch and also obtain a rather general uniqueness statement. The apriori estimates in $G_\Sigma^m$
are then proven inductively by combining bounds for normal derivatives and for those in  tangential and time directions.
Here and later on we can use the results in normal direction from \cite{Sp1} and \cite{Sp2} since they do not involve 
boundary conditions. (And so we can omit a few very lenghty and intricate proofs.) However, the tangential bounds lead to 
new terms which have to be estimated carefully, since the nonlinear boundary conditions lead to  coefficients at the boundary
with less integrability  than those in the interior. This fact causes the smallness condition mentioned above.

In Section~\ref{sec:lin} we then derive the main linear regularity Theorem~\ref{thm:reg} needed for the nonlinear theory. 
We again follow the procedure from \cite{Sp1} and employ different regularization procedures in normal, tangential
and time directions which have to be intertwined in a careful induction. 
We again have to deal with  new terms at the boundaries, but also various additional  difficulties arise 
because of the more complicated compatibility conditions due to the time-dependent boundary coefficient on the linear level.

Based on the linear theory, we can then perform the core fixed point argument in Theorem~\ref{thm:local}.
Relying on the reasoning from \cite{Sp2}, we can focus on the smallness conditions needed to deal with the semilinear
boundary conditions.  They allow us to absorb error terms in the crucial estimates.
In the last section we  derive our main local wellposedness Theorem~\ref{thm:lwp}.
It is based on auxiliary results preparing the improved blow-up condition and the continuous dependence on data.
Compared to \cite{Sp2}, we have to deal again with additional boundary terms, the new compatibility conditions and
the needed smallness assumptions.

\section{Auxiliary results and notation}\label{sec:aux}
Let $m\in\NN$ and $G\subseteq \RR^3$ be either an open set with a compact  boundary $\Sigma:=\partial G$ of 
class $C^{m+2}$ or $G=\RR^ 3_+=\{x\in\RR^ 3\,|\, x_3>0\}$. Its  outer unit normal is $\nu$.
We write $c$ and $C$ for generic positive 
constants, as well as $c(a,b,\dots)$ or $C(a,b,\dots)$ if they depend on $a,b,\dots$. The range of a map $v$
is denoted by $\ran(v)$, and $B(x,r)$ or $B_M(x,r)$ is the closed  ball in a metric space $M$ with center $x$ and radius $r$.
Let $t_0\in\RR$ be the initial time, where  we often take $t_0=0$ in view  of time invariance.
To control constants, we partly fix a time $T'>t_0$ and let $T\in(t_0,T')$.
For $J=(t_0,T)$, we set $\Omega=J\times \RR^ 3_+$, and $\Gamma= J\times (\RR^2\times \{0\}) \cong J\times \RR^2$.
Sometimes $J$ also denotes other open intervals. 

We often use the same symbols for spaces of scalar, vector 
or matrix valued  functions. Sobolev spaces are designated by $W^ {s,p}$ with $W^{s,2}=\cH^s$.
Spaces on $\Sigma$ are equipped with the surface measure,  written as $\D x$.
For $\gamma\ge0$ and $t\in \RR$ we set $e_{-\gamma}(t)=\e^{-\gamma t}$. We employ time-weighted norms such as
\[\|f\|_{\cH^m_\gamma(\Omega)}^2 = \sum_{|\alpha|\le m} \|e_{-\gamma} \partial ^\alpha f\|_{L^2(\Omega)}^2,\]
and denote the respective spaces by the subscript $\gamma$.

Let $v: G\to \RR^3 $ be sufficiently regular.
We write $\tr_n v$ for the trace of the normal component $v\cdot \nu$ on $\Gamma$, and
$\tr_t v$ for the tangential trace $v\times \nu$ on $\Gamma$. We further employ its rotated variant 
$\tr_\tau v = \nu \times (\tr_t v)$ which is the tangential component $\tr v-(\tr_n v) \nu$ of
the full trace $\tr v$. Note that the Euclidean norms of $\tr_t v$ and $\tr_\tau v$ coincide. 
For vector fields $u=(u^1, u^2):G\to\RR^6$ and $k\in\{t,\tau,n\}$ we further set $\tr_k u = (\tr_k u^1, \tr_k u^2)$.  
It is well known that the mappings 
\[ \tr_n : \cH(\tdiv)\to \cH^ {-1/2}(\Sigma) \qquad \text{and} \qquad
\tr_t : \cH(\tcurl)\to \cH^ {-1/2}(\Sigma)^ 3\]
are continuous, where their maximal domains $\cH(\tcurl)$ and $\cH(\tdiv)$ in $L^2(G)^3$
are endowed with the respective graph norm, see Theorems~IX.1.1+2 in \cite{DL}. 
  For sufficiently regular functions $w: J\times G \to \RR^l$  with $l\in\{3,6\}$ we set
  $(\tr_k w)(t)= \tr_k (w(t))$, see also the remarks at the beginning of Section~\ref{sec:est}.
We introduce
\begin{align*}
G^m &=G^m(J\times G) = \bigcap\nolimits_{j=0}^m C^j(\ol{J},\cH^{m-j}(G)^ 6) \qquad  \text{and}\\
\tilde G^m &=\tilde G^m(J\times G)= \bigcap\nolimits_{j=0}^m W^ {j,\infty}(J,\cH^{m-j}(G)^ 6).
\end{align*}
 Our solutions have extra trace regularity expressed by the space
\[ G^m_\Sigma =G^m_\Sigma(J\times G)
   = \{ u\in G^m(J\times G)\,|\, \tr_\tau u \in \cH^m (J\times \Sigma)^6\},\]
which is equipped with its canonical norm. In the fixed point argument we also need
the slightly larger one
\[ \tilde G^m_\Sigma  =\tilde G^m_\Sigma (J\times G) 
 = \{ u\in \tilde G^m(J\times G) \,|\, \tr_\tau u \in \cH^m (J\times \Sigma)^6\}.\]
 We also use the subscript $\Sigma$ if $G=\RR_+^3$.

To reformulate  \eqref{eq:maxwell0} as a standard first  order system, we 
write  $u = (E, H)= (u^1,u^2)$ for the unknowns and introduce the matrices and maps
\begin{align}\label{eq:Aj}
 J_1 &= \begin{pmatrix}
	  0 &0 &0 \\
	  0 &0 &-1 \\
	  0 &1 &0
       \end{pmatrix},
        \ \
 J_2 =  \begin{pmatrix}
         0 &0 &1 \\
         0 &0 &0 \\
         -1 &0 &0
        \end{pmatrix},
        \ \
 J_3 = \begin{pmatrix}
        0 &-1 &0 \\
        1 &0 &0 \\
        0 &0 &0
       \end{pmatrix},  \quad
    \!    f =\! \begin{pmatrix}
	- J_0 \\ 0  
     \end{pmatrix}\!,   \notag\\
A_j^{\co} &= \begin{pmatrix}
        0 & -J_j \\
        J_j &0 
       \end{pmatrix},
    \quad \theta=(\theta_e,\theta_m), \quad\chi = \partial_{u} \theta , \quad 
        \sigma = \begin{pmatrix}
      \sigma_e &0 \\
      0 & 0
     \end{pmatrix},  \\
    B_1&=  \begin{pmatrix}\tr_t &  0  \end{pmatrix}, \quad
    B_2=  \begin{pmatrix} 0 & \tr_t   \end{pmatrix}, \quad
    B(u)= B_2 -\nu\times  \zeta(B_1 u) B_1 \notag
\end{align}
for $j\in\{1,2,3\}$. Observe that $\curl = J_1 \partial_1 +  J_2 \partial_2  +  J_3 \partial_3$. 
With this notation the Maxwell system  \eqref{eq:maxwell0} becomes
\begin{align}\label{eq:maxwell}
\begin{aligned}
\chi(u(t))\partial_t u(t) +\sum_{j=1}^3 A_j^{\co}\partial_j u(t) + \sigma(u(t))u(t) &=f(t), 
   \quad &&x \in G, \hspace{0.5em} &t\ge t_0, \\
B(u(t))u(t)   &= g(t), &&x \in \Sigma, &t\ge t_0,\\
u(t_0)&= u_0 &&x \in G,
\end{aligned}
\end{align}
where $u_0:=(E_0, H_0)$.
A \emph{solution} $u$ on an interval $J$ (with $t_0 \in J$) to this system belongs to $G^m_\Sigma(J' \times G)$ for every 
compact interval $J' \subseteq J$,
$u(t)$ takes values in a closed subset of $\cU$  for each $t\in J$, and $u$ satisfies \eqref{eq:maxwell} for $t\in J$.

Let $\tilde{m}=\max \{m,3\}$ and $\cU\sub \RR^ 6$ be an open convex set. We write
 $V\Subset\cU$ for open subsets $V$ of $\cU$ with a compact closure $\ol{V}\sub \cU$.
We assume that the coefficient $\theta$  belongs to $C^{\tilde{m}}(G\times \cU, \RR^{6\times6})$  and that
\begin{align} \label{ass:main}
&\chi = \partial_{u} \theta\in C^{\tilde{m}}(G\times \cU, \RR^{6\times6}_{\sym}), \quad
  \sigma \in C^{\tilde{m}}(G\times \cU, \RR^{6\times6}), \quad 
  \zeta \in  C^{\tilde{m}}_\tau(T_{\cU}\Sigma, \RR^{3\times3}_{\sym}),\notag \\
  &\forall\; V\Subset\cU, \ \alpha\in \NN_0^9, \ |\alpha|\le \tilde{m}: 
   \sup_{G\times V} (|\partial^\alpha \chi| +|\partial^\alpha\sigma|)<\infty,
     \  \sup_{\Sigma\times V} |\partial^\alpha \zeta|<\infty, \notag \\
  &\chi,\zeta \ge \eta I>0,
 \end{align}
 for some number $\eta>0$, where $T_{\cU}\Sigma = \{(x,v) \,|\, x \in \Sigma , v \in B_1(x) \cU\}$. 
 Mainly to unify notation, here we allow for 
a larger class of `conductivities' than in \eqref{eq:maxwell0}.
The subscript $\tau$ means that $\zeta$ is tangential in the sense that  $\zeta\nu^\perp \sub \nu^\perp$.
 Actually we only need uniform positive definiteness on each set
 $G\times \ol{V}$, respectively $\Sigma\times \ol{V}$, but we impose the above condition to
 simplify the presentation a bit. 
 Below we also state a variant of these assumptions for linear boundary conditions.

Let $\hat {u}\in \tilde G^m_\Sigma$ with $m\ge 3$. We freeze the coefficients in the nonlinearities
setting  $b= \zeta(B_1 \hat u)$, $A_0= \chi(\hat u)$, $D=\sigma(\hat u)$, and
\[B(t) = B_2 -  \nu\times  b(t) B_1\]
for $t\in J$. 
Such coefficients belong to  the function spaces 
\begin{align*}
F^m&= F^m(J\times G) =\{A\in W^{1,\infty}(J\times G)^{k\times k}\,|\,\forall\;\alpha\in\NN_0^4
   \text{ with } 1\le |\alpha| \le m:\\
       &\qquad \qquad \qquad \qquad \qquad \partial^ \alpha A \in L^ \infty(J, L^2(G))\},\\
F^{m,0} &= F^{m,0}(G) =\{A \!\in \! L^{\infty}(G)^ {k\times k} \, | \, \forall \alpha \!\in \! \NN_0^3 \text{ with } 
  1 \le |\alpha| \le m:   \partial^ \alpha A\in L^2(G)\},
\end{align*}
which are  endowed with their natural norms.
The corresponding spaces for the domains $\Omega$, $\RR^3_+$, $J\times \Sigma$, 
$\Sigma$, $\partial \RR^3_+$,  and  $\Gamma$ are denoted analogously. For $b$ we need the space
\[F^m_\cH= F^m_\cH(J\times \Sigma)=  F^{m-1}_\tau(J\times \Sigma)\cap  \cH^m(J\times \Sigma)^ k\]
whose norm is given by 
\begin{align*} 
\|b\|_{ F^m_\cH(J\times \Sigma)}^2 &= \|b\|_{ F^{m-1}(J\times \Sigma)}^2 + [b]_{\cH^m(J\times \Sigma)}^2
\qquad \text{with}\\
 [b]_{\cH^m(J\times \Sigma)}^2 &=\sum_{\alpha\in\NN_0^4, |\alpha|=m} \|\partial^\alpha b\|_{L^2(J\times \Sigma)}^2,
\end{align*}
Here, $k\in\NN$ usually is $1$, $3$, or $6$.
If we drop $J$ in $F^m$, we refer to the subspace of maps being constant in time.
The subscript $\eta>0$ in any of these or other spaces means that the functions take values in symmetric matrices 
with lower bound $\eta>0$; whereas $\cp$ indicates that the maps are constant outside a compact subset of $\ol{J\times G}$, 
respectively  $\ol{G}$. 
(For bounded $G$ the latter subspace is equal to $F^m$ or $F^{m,0}$.) In Lemma~2.1 and Corollary~2.2 of 
\cite{Sp2} one can find detailed results concerning the mapping properties of $\chi$ 
and $\sigma$ in these spaces. Several variants of the product rules in $\tilde{G}^m$, $F^m$ and $\cH^m$
are shown in Lemma~2.1 of  \cite{Sp1} for $G$, which easily extend to the boundary.
(See also \cite{Sp0} for a more detailed presentation.)  In the context of  $F^{m}_\cH(J\times \Sigma)$ 
new issues arise for the terms of highest order, which are discussed in the relevant parts of the proofs.

We also look at the case of linear non-autonomous boundary conditions assuming that
\begin{align} \label{ass:main1}
&\chi \in C^{\tilde{m}}_\eta(G\times \cU, \RR^{6\times6}_{\sym}), \quad
  \sigma\in C^{\tilde{m}}(G\times \cU, \RR^{6\times6}), \quad 
 \zeta=b \in  F^{\tilde{m}}_{\eta,\tau}(\RR_+\times \Sigma),\notag \\
  &\forall\; V\Subset\cU, \ \alpha\in \NN_0^9, \ |\alpha|\le \tilde{m}: 
  \sup\nolimits_{G\times V} (|\partial^\alpha \chi| +|\partial^\alpha\sigma|)<\infty,
 \end{align}
 for some number $\eta>0$.

Let $A_0\in F_{\eta}^ m(J\times G)$, $D\in F^m(J\times G)$, 
$b\in F_{\cH,\eta}^m(J\times \Sigma)$, $f\in \cH^m(J\times G)$, $g\in \cH^m(J\times \Sigma)$ 
with $g\cdot \nu=0$, and $v_0\in\cH^m(G)$. We then look for a solution $v\in G^m_\Sigma(J\times G)$ of the linearized 
problem 
\begin{align}\label{eq:maxwell-lin}
\begin{aligned}
A_0(t)\partial_t v(t) +\sum_{j=1}^3 A_j^{\co}\partial_j v(t) + D(t)v(t) &=f(t), \qquad &&x \in G, \quad &t\ge t_0,\\
B(t)v(t)    &= g(t), &&x \in \Sigma, &t\ge t_0, \\
v(t_0)&= v_0, &&x \in G. 
\end{aligned}
\end{align}

We solve the system \eqref{eq:maxwell-lin} via localization, proceeding as in \cite{SS}, \cite{Sp0} and \cite{Sp1}.
To this aim, we cover $\ol{G}$ with connected open 
sets $U_0, U_1, \dots, U_N$ where ${U_0}\sub G$ and $\Sigma\sub U_1\cup \dots \cup U_N$. 
For each $i\in\{1,\dots, N\}$ we fix a $C^{m+2}$--diffeomorphism 
$\ph_i$ from $U_i$ onto an open subset $V_i$ of $B(0,1)$ which maps  $\Sigma\cap U_i$ onto $\{y\in V_i\,|\, y_3=0\}$ 
and $G\cap U_i$ onto $\{y\in V_i\,|\, y_3>0\}\sub \RR^3_+$. The resulting composition operators are denoted by
\[ \Phi_i: L^2(U_i) \to L^2(V_i); \ \ v\mapsto v\circ \ph^{-1}_i, \qquad 
  \Phi_i^ {-1}: L^2(V_i) \to L^2(U_i); \ \ v\mapsto v\circ \ph_i.\]
We use the same notation for the induced maps  on Sobolev spaces, and also for the spaces on the domains
$J\times U_i$ or $J\times V_i$. The extension by 0 or restrictions of a function $v$ are also denoted by $v$.

Let $\{\theta_i \,|\,i\in \{0,1,\dots, N\}\}$ be a smooth partition of unity subordinate to the sets $U_i$, and 
$\sigma_i$ (resp.\ $\omega_i$) be test functions in $U_i$ (resp.\ $V_i$) which are equal to 1 on $\supp \theta_i$
(resp.\ on $K_i:= \ph_i (\supp \sigma_i$)). We can find a constant $\tau\in(0,1)$ and an index $z(i)\in\{1,2,3\}$
with $|\partial_{z(i)}\ph_{i,3}|\ge \tau$ for each $i$, see e.g.\ Lemma~5.1 of \cite{Sp0}, where $\ph_{i,3}$ is the third 
component of $\ph_i$. We assume that  $z(i)=3$. The other cases are treated analogously, cf.\ Section~5 of
\cite{Sp0} and Section~3 of~\cite{SS}. Let $i\in\{1,\dots,N\}$.

The usual localization procedure leads to a first order system on $\RR_+^3$ with variable coefficient matrices,
see e.g.\ (2.5) in \cite{Sp1}.  For our analysis it is important to keep the constant matrix  $A_3^\co$ and the form
of the boundary condition. To this aim, we set
\[\beta_i= \omega_i \Phi_i (\partial_3\ph_{i,3})
  + (1-\omega_i)\frac{\partial_3 \ph_{i,3}}{|\partial_3 \ph_{i,3}|}(\ph^ {-1}_i(y_i))\]
for a fixed point $y_i\in V_i$. The second summand is not important since we mostly work on  $K_i$ where we have
$\beta_i= \Phi_i \partial_3\ph_{i,3}$. It is easy to check the lower bound $|\beta_i| \ge \tau$. 
We assume that 
$\beta_i\ge \tau$ as the other sign is handled in the same way, cf.\ Section~5 of \cite{Sp0} and Section~3 of~\cite{SS}. We then set 
\[ \hat R_i = \beta_i^ {-1/2} \begin{pmatrix} 1 & 0 & 0\\ 0 & 1& 0\\ 
    \omega_i \Phi_i(\partial_1 \ph_{i,3}) & \omega_i \Phi_i(\partial_2 \ph_{i,3}) & -\beta_i\end{pmatrix}, \qquad
     R_i = \begin{pmatrix} \hat R_i & 0 \\ 0& \hat R_i \end{pmatrix}, \]
and define the `localized' coefficients
\begin{align}\label{def:A-i}
A_0^i &= R_i\big(\omega_i \Phi_i A_0  + (1-\omega_i)\eta I\big)R_i^T,\notag \\
A_j^i &= R_i\Big(\omega_i \Phi_i \Big(\sum\nolimits_{k=1}^3 A_k^\co \partial_k\ph_{i,j} \Big) +(1-\omega_i)
     \frac{\partial_3 \ph_{i,3}}{|\partial_3\ph_{i,3}|}(\ph^{-1}_i(y_i))A_3^\co\Big)R_i^T,\notag\\
A_3^i &= A_3^ \co,\\
D^i &= \omega_i R_i\Phi_i D R_i^T - \sum\nolimits_{j=1}^3 A_j^i
\partial_j (R_i^T)^ {-1} R_i^T\notag
\end{align}
for $j\in\{1,2\}$ on $\RR^3_+$ as in Section~3 of \cite{SS} or Section~2 of \cite{Sp1}, where partly a different notation was used. 
We note that $A_0^i$ belongs to $F^m_{\eta,\cp} (J\times G)$ and $D^i$ to $F^m_{\cp} (J\times G)$, and that
their norms in these spaces are bounded by a constant  times the analogous norms of $A_0$ and $D$.
The maps $A_1^i$ and $A_2^i$ are contained in  the space
\begin{align*}
F^m_{\cf}(\RR^3_+)&=\Big\{ A\in F^m_{\cp}(\RR^3_+) \,\Big|\, \exists\; \mu_j \in F^m_{\cp}(\RR^3_+,\RR):
  A = \sum\nolimits_{j = 1}^3 A_j^\co  \mu_j\Big\},
\end{align*}
and they are dominated in the norm of $F^m$ by a constant only depending on $G$, the charts, and $\omega_i$.
We stress that the functions $\mu_j$ are scalar.

To deal with the boundary condition, we set $\kappa_i= - \nabla\ph_{i,3}\cdot \nu\in C^ {m+1}(\Sigma\cap U_i)$.
Observe that  $\nabla\ph_{i,3}=-\kappa_i \nu$ since $\nabla\ph_{i,3}$ is normal to $\Sigma$ 
and that $\kappa_i > 0$ on $\Sigma \cap U_i$ by the properties of $\ph_{i,3}$. We further define
\begin{align*}
B_0 &=\begin{pmatrix} 0&\nu_3 &-\nu_2\\ -\nu_3 & 0& \nu_1\\ \nu_2 & -\nu_1& 0  \end{pmatrix}, \qquad
B_0^ \co=\begin{pmatrix} 0 & -1 & 0\\ 1 & 0 & 0\\ 0 & 0 & 0 \end{pmatrix} =J_3, \\ 
 B_1^\co&=\begin{pmatrix} B_0^\co & 0  \end{pmatrix}, \qquad B_2^\co=\begin{pmatrix} 0 &B_0^\co  \end{pmatrix},
\end{align*} 
cf.\ \eqref{eq:Aj}.
Abusing notation, we identify $B_0$ with $\tr_t$ at $\Sigma$ and  $B_0^ \co$ with $\tr_t$ at $\partial \RR^3_+$
(where  $\nu= -e_3$). As in (3.12) of \cite{SS} or Section~2 of \cite{Sp1}, we then compute
\[ B_0^ \co =\hat R_i \Big(\omega_i \Phi_i(\kappa_i B_0) 
     + (1-\omega_i) \frac{\partial_3 \ph_{i,3}}{|\partial_3 \ph_{i,3}|}(\ph^ {-1}_i(y_i)) B_0^\co   \Big) \hat R_i^T\]
on $\RR^2\times\{0\}$.  For the transformed coefficients, we take further cut-offs $\tilde{\omega}_i\in C^\infty_c (V_i)$
which are equal to 1 on $\supp \omega_i$ and define the auxiliary maps
\begin{align*} 
\tilde{b}_i(t) &= \tilde\omega_i \Phi_i(\kappa_i^{-1} b(t)) +(1-\tilde \omega_i) \eta I,\\
C^i(t) &= B_0^\co \tilde{b}_i(t) \Phi_i(\kappa_i B_1) + \Phi_i(\kappa_i B_0)\tilde{b}_i(t)B_1^\co,\\
\Omega_i &=\begin{pmatrix} \omega_i I & 0\\ 0 & I\end{pmatrix}.
\end{align*}
We now introduce the localized boundary operators and coefficients
\begin{align}\label{def:B-i}
B^i(t)&= \hat{R}_i \Big[\omega_i \Phi_i(\kappa_i B(t)) \Omega_i  
         + (1-\omega_i) \frac{\partial_3 \ph_{i,3}}{|\partial_3 \ph_{i,3}|}(\ph^ {-1}_i(y_i))B_2^\co\\
       &\qquad  +\omega_i (1-\omega_i) \frac{\partial_3 \ph_{i,3}}{|\partial_3 \ph_{i,3}|}(\ph^ {-1}_i(y_i))C_i(t)
		+ (1-\omega_i)^2B_0^\co \tilde{b}_i(t)B_1^\co\Big] R_i^T,\notag \\
b_i(t)&= (\hat{R}_i^T)^{-1} \tilde{b}_i(t)\hat{R}_i^{-1}, \notag
\end{align}
on $\RR^2\times\{0\}$. One can then derive the  identity
\beq\label{eq:B-i} 
B^i(t) =B_2^\co +B_0^\co b_i(t) B_1^\co.
\eeq
For the reconstruction of the original boundary condition on $\Sigma$ from the localized ones it is crucial 
to note that $B_i(t) = \hat{R}_i \Phi_i(\kappa_i B(t)) R_i^T$ on $\omega_i^{-1}(\{1\})$.
 Redefining $\eta$ if necessary, we obtain that $b_i$ is contained in $F_{\cH,\eta,\cp}^m(\Gamma)$
and bounded in this norm by a constant times the norm of $b$ in $F_{\cH,\eta}^m(J\times \Sigma)$.

Let $h\in \cH^m(J\times G)$ and $v\in G^{m}(J\times G)$.  We introduce the transformed data
\begin{align}\label{def:data-i} 
v_0^i&=  (R_i^ T)^{-1} \Phi_i(\theta_i v_0), \qquad g^i = \hat{R}_i \Phi_i(\theta_i \kappa_i g),\\
f^i&= f^i(h,v) = R_i\Big(\Phi_i(\theta_i h) 
  + \Phi_i\Big(\sum\nolimits_{j=1}^3A_j^\co\partial_j\theta_i v\Big) \Big).\notag
\end{align}
These functions belong to $\cH^m(\RR^3_+)$, $\cH^m(\Gamma)$,
and $\cH^m(\Omega)$, with norms bounded by a constant times the corresponding norms of $v_0$, $g$, $h$, and $v$, 
respectively. (In the existence proof one has to construct a suitable map $h$ for a given $f$.)
Instead of \eqref{eq:maxwell-lin}, we are now looking at the linear system
\begin{align}\label{eq:maxwell-lin-i}
\begin{aligned}
A_0^ i(t)\partial_t v(t) +\sum_{j=1}^3 A_j^i\partial_j v(t) + D^i(t)v(t) &=f^i(t), \quad &&x \in \RR^3_+ \hspace{0.5em} &t\ge t_0, \\
B^i(t)v(t)    &= g^i(t), &&x \in \partial \RR^3_+ &t\ge t_0,\\
v(t_0)&= v_0^i &&x \in \RR^3_+,
\end{aligned}
\end{align}
for $i\in\{1,\dots,N\}$ with the operators and maps from \eqref{def:A-i}, \eqref{eq:B-i}, and 
\eqref{def:data-i}. Once we have established apriori estimates and the regularity theory for~\eqref{eq:maxwell-lin-i}, 
we obtain the corresponding assertions on $G$ by proceeding as in Section~5 of~\cite{Sp0} 
respectively Section~3 of~\cite{SS}. 
To that purpose, we also need the case $i=0$ which leads to a much simpler full space problem already treated
in \cite{Sp1}, for instance.
We put $\partial_0=\partial_t$ and define the hyperbolic operators
\begin{align} \label{def:L}
L(w)&= \chi(w)\partial_t  +\sum\nolimits_{j=1}^3 A_j^{\co}\partial_j + \sigma(w) \qquad \text{on } J\times G,\notag\\
L^\co(A_0, D)=L^\co &=A_0\partial_t  +\sum\nolimits_{j=1}^3 A_j^{\co}\partial_j  + D \qquad \text{on } J\times G,\notag\\
L^i(A_0,A_1, A_2,A_3^\co,D)= L^i&=\sum\nolimits_{j=0}^3 A_j^ i\partial_j  + D^ i \qquad \text{on } \Omega.
\end{align}
In the last operator and in \eqref{eq:maxwell-lin-i} we often omit the superscript $i$.

If $u$ or $v$ in $G^m$ solves one of the above linear or nonlinear Maxwell systems, we can differentiate 
the evolution equation and
the boundary condition $m-1$ times and then take the time trace at $t=t_0\in\ol{J}$. The resulting 
compatibility conditions on $\{t_0\} \times \Sigma$ are thus a necessary property for any sufficiently regular solution. From (2.1) of 
\cite{Sp1} and (2.9) of \cite{Sp2} we first recall several
important formulas relating time and space derivatives of solutions, where we assume conditions  \eqref{ass:main1},
respectively  \eqref{ass:main} for the  nonlinear boundary condition.

Take a time $t_0\in\ol{J}$, 
an inhomogeneity $f\in \cH^m(\Omega)$, and initial values $u_0,v_0\in\cH^m(\E)$. Let $p\in\{0,\dots,m\}$.
Assume that $v\in G^m(\Omega)$ solves \eqref{eq:maxwell-lin-i} without the boundary condition. 
Differentiating the evolution equation in time and dropping the superscript $i$, we deduce that 
this function satisfies
\begin{equation} \label{eq:S-lin}
 \partial_t^p v(t_0) = S_{m,p,A_j,D}(t_0,v(t_0),f),
\end{equation}
for all  $p \in \{0, \dots, m\}$, where the term
$S_{m,p}^\lin = S_{m,p, A_j,D}(t_0,u_0,f)$ is defined by
\begin{align} \label{def:S-lin}
 S_{m,0}^\lin &= u_0, \nonumber \\
 S_{m,p}^\lin&= A_0(t_0)^{-1} \Big[\partial_t^{p-1} f(t_0) 
  - \sum_{j = 1}^3 A_j \partial_j S_{m, p-1}^\lin 
  - \sum_{l=1}^{p-1}\binom{p-1}{l} \partial_t^l A_0(t_0) S_{m, p-l}^\lin \nonumber \\
  &\qquad \quad- \sum_{l=0}^{p-1} \binom{p-1}{l} \partial_t^l D(t_0) S_{m,p-1-l}^\lin\Big], \quad p\ge1.
\end{align}
An analogous formula is true on $G$ if $v$ fulfills \eqref{eq:maxwell-lin}
and we replace $A_j$ by $A_j^\co$ for $j\in\{1,2,3\}$.

Next, let $u\in G^m(J\times G)$ satisfy \eqref{eq:maxwell}. We then obtain
\begin{equation} \label{eq:S-nl}
	\partial_t^p u(t_0) = S_{m,p,\chi, \sigma}(t_0, u(t_0), f) 
\end{equation}
for all $p \in \{0, \ldots, m\}$. Here we inductively define the maps
$ S_{m,p}^\nl =S_{m,p,\chi, \sigma}(t_0,u_0,f)$ by 
\begin{align} \label{def:S-nl}
 S_{m, 0}^\nl  &= u_0,\notag\\
 S_{m,p}^\nl &= \chi(u_0)^{-1} \Big[\partial_t^{p-1}f(t_0) - \sum_{j=1}^3 A_j^{\co} \partial_j 
  S_{m,p-1}^\nl 
 - \sum_{l=1}^{p-1} \binom{p-1}{l} M_1^l S_{m,p-l}^\nl  \notag\\
   &\qquad\qquad\qquad - \sum_{l=0}^{p-1} \binom{p-1}{l} M_2^l S_{m,p-1-l}^\nl\Big], \\
M_k^p &:= \sum_{1 \leq j \leq p} \sum_{\substack{\gamma_1,\dots,\gamma_j\in  \NN \\ 
   \sum \gamma_i = p}} \sum_{l_1, \ldots, l_j = 1}^6 C(p,\gamma)
	  (\partial_{y_{l_j}} \dots \partial_{y_{l_1}} \theta_k)(u_0)
	      \prod_{i=1}^j (S_{m,\gamma_i}^\nl)_{l_i}, \notag
\end{align}
where $p\ge1$,  $k \in \{1,2\}$, $\theta_1 = \chi$, $\theta_2 = \sigma$,  $M_2^0 = \sigma(u_0)$, 
and $C(p,\gamma)=C((p,0,0,0), \gamma_1, \dots, \gamma_j) $ is the constant from Lemma~2.1 of \cite{Sp2}. 

We have to estimate these maps. Lemma~2.3 of \cite{Sp1} shows the inequality
\begin{align}\label{est:S-lin}
\|S_{m,p,A_j,D}(&t_0, u_0,f)\|_{\cH^{m-p}(G)} \\
&\le c(r_0, \eta, m,p) \Big(\sum_{j = 0}^{p-1} \|\partial_t^j f(t_0)\|_{\cH^{m-1-j}(G)} 
  + \|u_0\|_{\cH^m(G)}\Big),\notag
\end{align}
 provided that $A_0\in F^{\tilde{m}}_\eta(J\times G)$, $A_1,A_2,A_3\in  F^{\tilde{m}}(G)$ 
and $D \in F^{\tilde{m}}(J\times G)$ and that the quantities 
\[ \|A_k(t_0)\|_{F^{\tilde{m}-1,0}(G)},  \|D(t_0)\|_{F^{\tilde{m}-1,0}(G)}, 
\|\partial_t^j A_0(t_0)\|_{\cH^{\tilde{m}-1-j}(G)},  \|\partial_t^j D(t_0)\|_{\cH^{\tilde{m}-1-j}(G)}\]
are bounded by  $r_0$ for $k\in\{0,1,2,3\}$ and $j\in \{1,\dots, m-1\}$. Here one can replace $G$ by $\RR^3_+$.
Similarly, Lemma~2.4 of \cite{Sp2}  says that
\begin{align}\label{est:S-nl}
\|S_{m,p,\chi, \sigma}(&t_0, u_0,f)\|_{\cH^{m-p}(G)} \\
&\le c(r_0, \eta,V, m,p) \Big(\sum_{j = 0}^{p-1} \|\partial_t^j f(t_0)\|_{\cH^{m-1-j}(G)} 
   + \|u_0\|_{\cH^m(G)}\Big),\notag
\end{align}
if $\chi$ and $\sigma$ fulfill \eqref{ass:main}, the range of $u_0$ is contained in $V\Subset \cU$  and 
the number in parentheses is less or equal $r_0$.  Lemma~2.4 of \cite{Sp2} also provides an analogous
Lipschitz estimate for arguments $(u_0, f)$ and  $(\tilde{u}_0, \tilde{f})$.

On the other hand, for $g\in \cH^m(J\times \Sigma)$ and $v\in G^m_\Sigma(J\times G)$ solving~\eqref{eq:maxwell-lin}  
we can differentiate the boundary condition in  \eqref{eq:maxwell-lin} up to $m-1$ times in time arriving at 
\begin{equation} \label{eq:cc-lin0}
 B(t_0) \partial_t^p v(t_0)  = \partial_t^p g(t_0)
     + \nu\times \sum_{k=1}^p \binom{p}{k} \partial^k_t b(t_0)  B_1 \partial_t^{p-k} v(t_0)
\end{equation}
on $\Sigma $ for all $ p\in\{1, \dots,m-1\}$. Replacing $B_1$ by $B_1^\co$ the same equation
is true on $\partial\RR^3_+$ and  a function $v\in G^m_\Sigma(\Omega)$ fulfilling \eqref{eq:maxwell-lin-i}.
Analogously, each solution $u\in G^m_\Sigma(J\times G)$ of \eqref{eq:maxwell} satisfies
\begin{equation} \label{eq:cc-nl0}
 B(u(t_0))  \partial_t^p u(t_0)= \partial_t^p g(t_0)
     +  \nu\times  \sum_{k=1}^p \binom{p}{k} \partial^k_t \zeta(B_1 u)(t_0)  B_1  \partial_t^{p-k} u(t_0)
\end{equation}
on $\Sigma$ for all $ p\in\{1, \dots,m-1\}$. To express the factors $\partial^k_t \zeta(B_1 u)(t_0)$,
as in \eqref{def:S-nl} we set
\[M_3^p = \sum_{1 \leq j \leq p} \sum_{\substack{\gamma_1,\ldots,\gamma_j \in \NN \\ 
   \sum \gamma_i = p}} \sum_{l_1, \ldots, l_j = 1}^2 C(p,\gamma)
	  (\partial_{y_{l_j}} \dots \partial_{y_{l_1}} \zeta)(B_1u_0)
	      \prod_{i=1}^j (B_1 S_{m,\gamma_i}^\nl)_{l_i} \]
	for $p\ge1$.      
Taking into account \eqref{eq:S-lin} and \eqref{eq:S-nl}, formulas \eqref{eq:cc-lin0} and \eqref{eq:cc-nl0}
lead to the equations 
\begin{align} \label{eq:cc-lin}
 B(t_0)  S_{m, p}^\lin &= \partial_t^p g(t_0)
     + \nu\times \sum_{k=1}^p \binom{p}{k} \partial^k_t b(t_0)  B_1 S_{m, p-k}^\lin,\\
B(u(t_0)) S_{m, p}^\nl &= \partial_t^p g(t_0)
     +  \nu\times  \sum_{k=1}^p \binom{p}{k} M^k_3 B_1 S_{m, p-k}^\nl,   \label{eq:cc-nl}
\end{align}
on $\Sigma$ for all $p \in \{0, \ldots, m-1\}$ respectively, which are called 
the \emph{compatibility conditions} of order $m$ for the systems \eqref{eq:maxwell-lin},  \eqref{eq:maxwell-lin-i}, 
respectively \eqref{eq:maxwell}. (For $p=0$ the sums are omitted.)

\section{Linear apriori estimates}\label{sec:est}
We first state the basic well-posedness result of the localized linear problem \eqref{eq:maxwell-lin-i} 
on the regularity level $m=0$. In particular, the data $f$, $g$ and $u_0$ belong to $L^2$.  If we have 
a solution $u\in L^2(\Omega)$, the evolution equation implies that $u$ is contained
in $\cH^1(J,\cH^{-1}(\RR_+^3))$. The initial condition thus makes sense in $\cH^{-1}(\RR_+^3)$. 
Moreover, the tangential trace can be extended from regular functions on $\Omega$ to 
those $u\in L^2(\Omega)$ with $Lu\in  L^2(\Omega)$ yielding a distribution $\tr_\tau u$ in $\cH^ {-1/2}(\Gamma)$,
see e.g.\ Remark~2.14 of \cite{Sp0}. The boundary condition can thus be understood as an equation 
in $\cH^ {-1/2}(\Gamma)$. We put $\DDiv A= \sum_{j=0}^3 \partial_j A_j$.

\begin{prop}\label{prop:L2}
Let $t_0=0$, $A_j=A_j^T\in W^{1,\infty}(\Omega)^ {6\times 6} $, $A_0\ge \eta I$, $A_3=A_3^\co$, 
$D\in L^\infty(\Omega)^ {6\times 6}$, $b=b^T\in L^\infty_\tau (\Gamma)^{3\times3}$ with $A_0,b\ge \eta I$, 
$u_0\in L^ 2(\RR^3_+)^6$, $f\in L^2(\Omega)^ 6$, and $g\in L^2(\Gamma)^3$ with $g\cdot \nu =0$. Then there is 
a unique solution $u\in L^2(\Omega)^6$ of \eqref{eq:maxwell-lin-i}. Moreover, $u$ belongs to 
$C(\ol{J},L^2(\RR^3_+)^6)$, $\tr_\tau u$ to  $L^2(\Gamma)^6$, and they satisfy equation \eqref{eq:energy} and the estimate
\begin{align}\label{est:apriori-0}
\e^ {-2\gamma T}\,\|u(T)&\|^2_{L^2(\RR_+^3)} + \gamma \, \|u\|_{L^2_\gamma(\Omega)}^ 2 
        + \|\tr_\tau u\|_{L^2_\gamma(\Gamma)}^ 2\\
&\le c\,\big(\|A_0(0)\|_{L^\infty(\RR^3_+)} \, \|u_0\|^ 2_{L^2(\RR_+^3)} + \tfrac1\gamma\, \|f\|_{L^2_\gamma(\Omega)}^2
      +\|g\|_{L^2_\gamma(\Gamma)}^ 2\big)\notag
\end{align}
for constants $c=c(\eta,\|b\|_\infty)\ge0$ and $\gamma_0(\eta,r)\ge1$ and
all $\gamma\ge \gamma_0(\eta,r)$ with $r:= \|D -\DDiv A/2\|_\infty$. 
 \end{prop}
This result is essentially known, see e.g.\ Proposition~2.1 of \cite{CE}, so that we only indicate the main steps 
of the proof. Since the precise form of the constants is crucial for us, we fully show the estimate 
\eqref{est:apriori-0} for a solution $u\in G^0$ such that also $\partial_j u$ for $j\in\{0,1,2\}$ and 
$A_3^\co\partial_3 v$ belong to $L^2(\Omega)$. Set $v= e_{-\gamma} u$. The equation $Lu=f$ yields
\[ \int_\Omega Lv \cdot v \dd x\dd t = \int_\Omega (e_{-\gamma} f\cdot v - \gamma A_0 v\cdot v) \dd x\dd t.\]
Using the symmetry of $A_j$, we next compute  
\[ \int_\Omega Lv \cdot v \dd x\dd t  =\frac12\sum_{j=0}^3\int_\Omega
    \partial_j (A_jv \cdot v)\dd x\dd t  + \int_\Omega(D-\tfrac12\DDiv A)v\cdot v \dd x\dd t.\]
The first term  on the right-hand side is equal to
\[\frac12\int_{\RR^3_+}(A_0(T)v(T)\cdot v(T)-A_0(0)u_0\cdot u_0)\dd x
              -\frac12\int_\Gamma A_3^\co v\cdot v \dd x\dd t.\]
Basic vector algebra and the boundary condition in \eqref{eq:maxwell-lin-i} then lead to
\begin{align*} 
-\frac12\int_\Gamma A_3^\co v\cdot v \dd x\dd t 
  &= \frac12\int_\Gamma e_{-2\gamma} \, (u^1\cdot \tr_t u^2  - u^2\cdot \tr_t u^1 )\dd x\dd t\\
  &= \int_\Gamma e_{-2\gamma} \, u^1\cdot \tr_t u^2 \dd x\dd t\\
   &=\int_\Gamma e_{-2\gamma} \,  u^1\cdot (g -(b\tr_t u^1)\times \nu) \dd x\dd t\\
   &= \int_\Gamma e_{-2\gamma} \,  (g\cdot \tr_\tau u^1  + (b\tr_t u^1)\cdot\tr_t u^1)  \dd x\dd t.
 \end{align*}
 The assumptions now imply the basic estimate  
 \begin{align*}
 &\e^{-2\gamma T} \tfrac{\eta}2\,\|u(T)\|^2_{L^2(\RR_+^3)} +\gamma\eta \, \|u\|_{L^2_\gamma(\Omega)}^ 2 
   + \eta \,  \|\tr_\tau u^1\|_{L^2_\gamma(\Gamma)}^ 2\\
 &\;\;\le \tfrac12  \|A_0(0)\|_{L^\infty(\RR^3_+)} \, \|u_0\|^ 2_{L^2(\RR_+^3)} + r \, \|u\|_{L^2_\gamma(\Omega)}^2 
     + \| f \! \cdot \! u  \|_{L^1_{2\gamma}(\Omega)} + \| g \! \cdot \! \tr_\tau  u^1  \|_{L^1_{2\gamma}(\Gamma)}. \notag
  \end{align*}  
By means of the Cauchy-Schwarz inequality and choosing $\gamma\ge 4r/\eta$, 
one easily deduces \eqref{est:apriori-0} with $\tr_\tau u^1$ instead of $\tr_\tau u$.
The remaining summand can be recovered from the boundary condition $\tr_t u^2= g - (b\tr_t u^1 )\times \nu$.
We note that for $\gamma=0$ we also obtain  the equality
\begin{align} \label{eq:energy}
\int_{\RR^3_+} &\tfrac12A_0(T)u(T)\cdot u(T)\dd x + \int_\Gamma (b\tr_t u^1)\cdot\tr_t u^1\dd (t,x)\\
  &= \int_{\RR^3_+}\tfrac12 A_0(0)u_0\cdot u_0\dd x + \int_\Omega\big((\tfrac12\DDiv A-D)u\cdot u + u\cdot f\big)\dd (t,x)\notag \\
  &\qquad - \int_\Gamma g\cdot \tr_\tau u^1 \dd (t,x).\notag
\end{align}

The  other steps follow a standard procedure going back (at least) to \cite{Ra}, see also \cite{BGS} or \cite{CP}.
Estimate \eqref{est:apriori-0} and a duality argument yield a solution $u$ of \eqref{eq:maxwell-lin-i} in $L^2(\Omega)$. 
One can also show a variant of \eqref{est:apriori-0} for the interval $J=\RR$  without the terms at times $t=T$ 
and $t=0$, assuming analogous regularity assumptions. One thus obtains a solution $u$ in  $L^2_\gamma(\Omega)$ 
of \eqref{eq:maxwell-lin-i} on $J=\RR$ without an initial condition. Mollifiers in $(t,x_1, x_2)$ yield
aproximate solutions $u_n$, where  $A_3^\co \partial_3 u_n$ belongs to $L_\gamma^2(\Omega)$ because of the 
evolution equation. Using the variant of \eqref{est:apriori-0} on  $J=\RR$ for $u_n$, we see that this estimate is 
also valid for solutions in $L^2_\gamma(\RR\times \RR^3_+)$ and that $\tr_\tau u$ is an element of 
$L^2_\gamma(\Gamma)$. 

As in Theorem~6.11 of \cite{CP} one next shows that the solution vanishes on $(-\infty,0)$ if the same is true for 
$f$ and $g$. For $u_0=0$,  again by mollification one can now construct a unique solution $u\in L^2_\gamma(\Omega)$ of \eqref{eq:maxwell-lin-i} 
 satisfying  $u\in C(\ol{J},L^2(\RR^3_+)^6)$, $\tr_\tau u\in L^2(\Gamma)^6$, \eqref{est:apriori-0}, and \eqref{eq:energy}. 
This fact also leads to the uniqueness statement in Proposition~\ref{prop:L2}. It thus remains to show that for $f=g=0$ the 
solution $u$ is contained in $C(\ol{J},L^2(\RR^3_+)^6)$ and fullfills $\tr_\tau u\in L^2(\Gamma)^6$, \eqref{est:apriori-0}, and \eqref{eq:energy}. 
In view of the estimate one only has to consider compactly supported $u_0$. In this case, the 
available full space result and the finite speed of propagation imply that $u$ is continuous in $L^2(\RR_+^3)$ for 
small times $t\ge0$. The result then follows by mollification.

\smallskip

We next establish higher order apriori estimates for solutions $u\in G^m$ of \eqref{est:apriori-0}, extending the 
approach of \cite{Sp1}. In the first step we treat the `tangential' derivatives $\partial^\alpha u$ with 
$\alpha\in\NN_0^4$ and $\alpha_3=0$. We use the space $\cH^m_\ta(\Omega)$ containing those functions $u\in 
L^2(\Omega)$ such that all such derivatives with $|\alpha|\le m$ belong to $L^2(\Omega)$, which is equipped with 
its natural norm. The space  $\cH^m_\ta(\E)$ is defined analogously.
The number $\gamma_0$ is taken from Proposition~\ref{prop:L2}, whereas $\delta_{m>2}$ is equal 
to 1 if $m\in\{3,4,\dots\}$ and zero if $m\in\{1,2\}$.

\begin{lem}\label{lem:tang-est}
Let $t_0=0$, $T\in (0,T')$, $\rho,\eta > 0$,  $r \geq r_0 > 0$, $m \in \NN$, and $\tilde{m}=\max\{m,3\}$. Take coefficients 
$A_0 \in F^{\tilde{m}}_{\eta}(\Omega)$, $A_1, A_2 \in F^{\tilde{m}}_{\cf}(\RR^3_+)$, 
$A_3 = A_3^{\co}$, $D \in F^{\tilde{m}}(\Omega)$, and 
$b\in F_{\cH,\eta}^{\tilde{m}}(\Gamma)$ satisfying
\begin{align*}
	&\Fnorm{\tilde{m}}{A_i} \leq r, \quad \Fnorm{\tilde{m}}{D} \leq r,\quad  \|b\|_{F^{\tilde{m}-1}(\Gamma)}\le r,
	\quad [b]_{\cH^{\tilde{m}}(\Gamma)}\le \rho, \\ 
&\max\{\Fvarnorm{\tilde{m}-1}{A_i(0)},\max_{1\leq j \leq m-1} \Hhn{\tilde{m}-1-j}{\partial_t^j A_0(0)}\} \leq r_0, \\
	&\max\{\Fvarnorm{\tilde{m}-1}{D(0)}, \max_{1 \leq j \leq m-1} \Hhn{\tilde{m}-1-j}{\partial_t^j D(0)}\} \leq r_0
\end{align*}
for all $i \in \{0, 1, 2\}$.  Choose data $u_0 \in \cH^m(\RR^3_+)$, $f \in \cH^m_\ta(\Omega)$, and 
$g \in \cH^m(\Gamma)$ with $g\cdot \nu =0$. Assume that the solution $u$ of~\eqref{eq:maxwell-lin-i} belongs 
to $G^{m}(\Omega)$. Then  it is also contained in  $G^{m}_\Sigma(\Omega)$ and we have
\begin{align} \label{est:tang}
  \sum_{\substack{|\alpha| \leq m \\ \alpha_3 = 0}} &\|\partial^\alpha u\|^2_{G^0_\gamma(\Omega)}  
    + \gamma\,\|u\|^2_{\cH^m_{\ta,\gamma}(\Omega)}  + \|\tr_\tau u\|_{\cH^m_\gamma(\Gamma)}^2\\
  &\leq  C_{m,0}^\ta \Big[\sum_{j = 0}^{m-1}\!\|\partial_t^j f(0)\|^2_{\cH^{m-1-j}(\RR_+^3)} 
  + \|u_0\|^2_{\cH^m( \RR^3_+)} +\delta_{m>2}\rho^2 \,\|B_1 u\|_{L^\infty_\gamma(\Gamma)}^2\Big]\notag \\
    &\qquad +  c\,\|g\|^2_{\cH^m_\gamma(\Gamma)} 
	+   \frac{C_m^\ta}{\gamma} \Big( \|f\|_{\cH^m_{\ta,\gamma}(\E)}^2 + \|u\|_{G^m_\gamma(\Omega)}^2 \Big)  \nonumber
\end{align}
for all $\gamma \geq \gamma_0$, where $C_m^\ta = C_m^\ta(\eta, r, T')$, $C_{m,0}^\ta = C_{m,0}^\ta(\eta, r_0,\|b\|_\infty)$, and $c=c(\eta,m,\|b\|_\infty)$.
If $b$ even belongs to $b\in F_{\eta}^{\tilde{m}}(\Gamma)^ {3\times3}$ with norm less or equal $r$, then one can 
set $\rho=0$ in the above inequality.
\end{lem}
\begin{proof}
Let $\alpha \in \NN_0^4$ with $|\alpha| \leq m$ and $\alpha_3 = 0$. We use the differential operator 
$L=L^i(A_0,A_1, A_2,A_3^\co,D)$ from \eqref{def:L}. Let $j\in\{0,1,2\}$. Exactly as in Lemma~3.2 
of \cite{Sp1} we derive the equations
\begin{align*} 
\begin{aligned}
L\partial^\alpha u &= f_\alpha, \qquad &&x\in\RR^3_+, \quad &t\in J,\\
\partial^\alpha u(0)&= u_{0,\alpha}, && x\in\RR^3_+,
\end{aligned}
\end{align*}
for the functions
\begin{align*}
  f_{\alpha} &= \partial^\alpha f - \sum_{j = 0}^2 \sum_{0 < \beta \leq \alpha} \binom{\alpha}{\beta} 
    \partial^\beta A_j \partial^{\alpha - \beta} \partial_j u
	- \sum_{0 < \beta \leq \alpha} \binom{\alpha}{\beta} \partial^\beta D \partial^{\alpha - \beta} u, \\
  u_{0,\alpha} &= \partial^\alpha u(0) = \partial^{(0,\alpha_1, \alpha_2, 0)} S_{m,\alpha_0,A_k,D}(0, u_0, f),
\end{align*}
 where we used $A_3 = A_3^{\co}$. Moreover, $f_\alpha$ is an element of $\cH^{m - |\alpha|}(\Omega)$ and  
 $u_{0,\alpha}$ of  $\cH^{m - |\alpha|}(\RR_+^3)$ satisfying
\begin{align} \label{est:f-u0-alpha}
 \|f_\alpha\|_{L^2_\gamma(\Omega)} &\leq 
                      \|f\|_{\cH^m_{\ta,_\gamma}(\Omega)} + c(m,r,T') \,\|u\|_{G^m(\Omega)},\\
  \|u_{0,\alpha}\|_{\cH^{m - |\alpha|}(\RR_+^3)} &\leq c(\eta,m,r_0)\Big(\sum_{k = 0}^{m-1} 
        \|\partial_t^k f(0)\|_{\cH^{m-1-k}(\E)} + \|u_0\|_{\cH^{m}(\RR_+^3)} \Big).\notag
\end{align}
In particular,  $\partial^{\alpha} u$ has a tangential trace in  $\cH^ {-1/2}(\Gamma)$. On the other hand,
  we can apply $\partial_j$ to $\tr_t \partial^{\alpha-e_j} u= (- \partial^{\alpha-e_j} u_2,  
  \partial^{\alpha-e_j} u_1, 0, - \partial^{\alpha-e_j} u_5,  \partial^{\alpha-e_j} u_4,0)$ in $\cH^{-1/2}(\Gamma)$ 
  and obtain  $\partial_j \tr_t \partial^{\alpha-e_j} u= \tr_t \partial^{\alpha} u$ for $j\in \{0,1,2\}$.
The boundary condition in \eqref{eq:maxwell-lin-i} thus leads to the equation
\[ B \partial^{\alpha} u =g_\alpha:= \partial^{\alpha} g 
 + \nu\times\sum_{0< \beta\leq\alpha} \binom{\alpha}{\beta} \partial^\beta b \tr_t\partial^{\alpha - \beta} u^1\]
on $\Gamma$. We bound the terms of this sum in $L^2(\Gamma)$, at first for  $b\in F_{\cH,\eta}^{\tilde{m}}(\Gamma)$.
We use the trace theorem in the form
\beq\label{est:trace-kappa}
\|\tr v\|_{H^{1/2}(\partial\E)}^2 \le \kappa\,\|v\|_{\cH^1_{\ta}(\E)}^2 + c\kappa^{-1} \, \|\partial_3 v\|_{L^2(\E)}^2
\eeq
for any $\kappa>0$, which can easily be derived using the Fourier transform.

Let $2\le |\beta|\le \tilde{m}-1$. Then  the norm of $\partial^\beta b$ in $L^\infty(J,L^2(\partial \RR^3_+))$ 
is less or equal $\|b\|_{F^{\tilde{m}-1}}$, whereas  $v:=e_{-\gamma}\tr_t\partial^{\alpha -\beta} u^1$ 
can be estimated in  $L^2(J, \cH^{3/2}(\partial \RR_+^3))\hra L^2(J,L^\infty(\partial \RR^3_+))$ and thus,  after taking squares,  by 
\[\ep \gamma\,\|u\|_{\cH^m_{\ta,\gamma}(\Omega)}^2 + c(\ep \gamma)^{-1}\,\|u\|_{\cH^m_{\gamma}(\Omega)}^2
   \le \ep \gamma\,\|u\|_{\cH^m_{\ta,\gamma}(\Omega)}^2 + cT'(\ep \gamma)^{-1}\,\|u\|_{G^m_{\gamma}(\Omega)}^2,\]
using \eqref{est:trace-kappa} with $\kappa=\ep \gamma$ and any $\ep>0$.  If $|\beta|=1$ we argue similarly, invoking the spaces 
$L^\infty(J,\cH^1(\partial\RR^3_+))\hra L^\infty(J,L^4(\partial\RR^3_+))$ for $b$ and $L^2(J,\cH^1(\E))$ for $v$.

The cases $m\in\{1,2\}$ are thus settled.  It remains to consider the
case $m\ge 3$ and $\alpha=\beta$. We can now only use the $L^2$-norm 
of  $\partial^\beta b$ and bound $v$ by its sup-norm which is  dominated by 
$\|u\|_{L^\infty_\gamma(J, \cH^2(\RR_+^3))}$. We conclude
\begin{align}\label{est:g-alpha}
&\|g_\alpha\|_{L^2_\gamma(\Gamma)}^2 \\
&\le \|g\|_{\cH^m_\gamma(\Gamma)}^2  \!   + \! c(m,r)\big[\ep \gamma \|u\|_{\cH^m_{\ta,\gamma}(\Omega)}^2
    \! + \! T'(\ep \gamma)^ {-1} \|u\|_{G^m_\gamma(\Omega)}^2 \big]
   \! + \!  c(m)\rho^2 \|\tr_t \! u^1\|_{L^\infty_\gamma(\Gamma)}^2. \notag
\end{align}
 The last term disappears if $m\in\{1,2\}$.  
 
 Since $g_\alpha \cdot \nu = 0$, Proposition~\ref{prop:L2} shows that $\tr_\tau \partial^\alpha u \in L^2(\Gamma)$ and that
\[ \|\partial^\alpha u\|_{G^0_\Sigma(\Omega)}^2+ \gamma \, \|\partial^\alpha u\|^2_{L^2(\Omega)} \le c \,\big(c(r_0)\,\|u_{0,\alpha}\|_{L^2(\E)}^2
    + \tfrac1\gamma\, \|f_\alpha\|_{L^2_\gamma(\Omega)}^2 + \|g_\alpha\|_{L^2_\gamma(\Gamma)}^2 \big)\]
for a constant $c=  c(\eta,\|b\|_\infty)$ and   all  $\gamma\ge\gamma_0(\eta,r)$. We now insert estimates \eqref{est:f-u0-alpha} and 
\eqref{est:g-alpha}. Fixing a small number $\ep=\ep(\eta,m,r)>0,$ one
 can absorb the second term in the right-hand side in \eqref{est:g-alpha} by the above left-hand side.
 The assertion follows.
 
 If $b$ belongs to  $F_{\eta,\cp}^{\tilde{m}}(\Gamma)$, then we can always estimate $\partial^\beta b$ in 
 $L^\infty(J,L^2(\partial \RR^3_+))$, so that \eqref{est:g-alpha} is true without the last term and we can proceed as above.
\end{proof}

The normal derivatives can be  treated by means of Proposition~3.3 of \cite{Sp1} which is independent of the boundary 
condition.  The full apriori estimate now follows by an induction argument.

\begin{thm}\label{thm:est}
Let $t_0=0$, $T\in (0,T')$, $\rho,\eta > 0$,  $r \geq r_0 > 0$, and $m \in \NN $. Take coefficients 
$A_0 \in F^{\tilde{m}}_{\eta}(\Omega)$, $A_1, A_2 \in F^{\tilde{m}}_{\cf}(\RR^3_+)$, 
$A_3 = A_3^{\co}$, $D \in F^{\tilde{m}}(\Omega)$, and $b\in F_{\cH,\eta}^{\tilde{m}}(\Gamma)$ satisfying
\begin{align*}
	&\Fnorm{\tilde{m}}{A_i} \leq r, \quad \Fnorm{\tilde{m}}{D} \leq r,\quad \|b\|_{F^{\tilde{m}-1}(\Gamma)}\le r,
	\quad [b]_{\cH^{\tilde{m}}(\Gamma)}\le \rho, \\ 
	&\max\{\Fvarnorm{\tilde{m}-1}{A_i(0)},\max_{1\leq j\leq m-1} \Hhn{\tilde{m}-1-j}{\partial_t^j A_0(0)}\} \leq r_0,\\
	&\max\{\Fvarnorm{\tilde{m}-1}{D(0)}, \max_{1 \leq j \leq m-1} \Hhn{\tilde{m}-1-j}{\partial_t^j D(0)}\} \leq r_0
\end{align*}
for all $i \in \{0, 1, 2\}$. Choose data $u_0 \in \cH^m(\RR^3_+)^6$,  $f \in \cH^m(\Omega)^6$,  and 
$g \in \cH^m(\Gamma)^3$ with $g\cdot \nu =0$. 
Assume that the solution $u$ of~\eqref{eq:maxwell-lin-i} belongs to $G^{m}(\Omega)$. Then it is also contained in 
$G^{m}_\Sigma(\Omega)$ and there is a number $\gamma_m= \gamma_m( \eta,r,T')\ge \gamma_0$ with
\begin{align} \label{est:main}
\|u&\|^2_{G^{m}_\gamma(\Omega)}+ \gamma\,\|u\|^2_{\cH^m_{\ta,\gamma}(\Omega)} + \|\tr_\tau u\|_{\cH^m(\Gamma)}^2 \\
  &\leq  (C_{m,0}+TC_m) \e^{mC_1T} \Big(\sum_{j = 0}^{m-1}\|\partial_t^j f(0)\|^2_{\cH^{m-1-j}(\RR_+^3)} 
  + \|u_0\|^2_{\cH^m( \RR^3_+)}+ \|g\|^2_{\cH^m_\gamma(\Gamma)} \notag \\
  &\qquad +  \delta_{m>2}\rho^2 \,\|B_1 u\|_{L^\infty_\gamma(\Gamma)}^2\Big) +
    \frac{C_m}{\gamma} \|f\|_{\cH^m_{\gamma}(\Omega)}^2 \nonumber
\end{align}
for all $\gamma \geq \gamma_m$, where $C_m = C_m(\eta, r, T')\ge 1$, $C_{m,0} = C_{m,0}(\eta, r_0,\|b\|_\infty)\ge1$, and 
$C_1= C_1(\eta, r, T')$ does not depend on $m$. If $b$ even belongs to 
$b\in F_{\eta}^{\tilde{m}}(\Gamma)^ {3\times3}$ with norm less or equal $r$, then one can set $\rho=0$ in the 
above inequality.
\end{thm}

\begin{proof}
For $m=1$ the result follows from Proposition~3.3 of \cite{Sp1}, Proposition~\ref{prop:L2} and 
Lemma~\ref{lem:tang-est}, after choosing $\gamma_1( \eta,r,T')\ge \gamma_0$ large enough.
We now assume that $m\ge 2$ and that the assertion has been shown for $m-1$, keeping the assumptions on the coefficients.

Following the proof of Theorem~3.4 of \cite{Sp1}, we apply the derivative $\partial_p$ with $p\in\{0,1,2\}$ 
to \eqref{eq:maxwell-lin-i}. The function $\partial_p u$ then satisfies
\begin{align}\label{eq:thm-est}
L(A_j,D)\partial_p u &= f_p :=  \partial_p f - \sum_{i = 0}^2 \partial_p A_i \partial_i u - \partial_p D u 
              \quad \text{on } \Omega,\notag\\
B \partial_p u & =g_p :=  \partial_p g + \nu\times (\partial_p b \tr_t u^1)\quad \text{on }\Gamma,\\
\partial_p u(0)&=u_{0,p}:=\partial_p u_0 \quad \text{on }\RR^3_+,\notag
\end{align}
where $\partial_0 u_0 = S_{m,1, A_j,D}(0,u_0,f)$.  The functions $f_p$ and $u_{0,p}$ belong to $\cH^{m-1}$ as 
shown in the proof of Theorem~3.4 of \cite{Sp1}. As in  \eqref{est:g-alpha} we infer
\begin{align}\label{est:gp}
&\|g_p\|_{\cH_\gamma^{m-1}(\Gamma)}^2 \\
 & \le  \|g\|_{\cH^m_\gamma(\Gamma)}^2 
  \! + \!  c(m,r)\big[\ep \gamma \|u\|_{\cH^m_{\ta,\gamma}(\Omega)}^2
   \! + \! T'(\ep \gamma)^{-1} \|u\|_{G^m_\gamma(\Omega)}^2 \big]
   \! + \!  c(m)\rho^2 \|\tr_t \! u^1\|_{L^\infty_\gamma(\Gamma)}^2,\notag
\end{align}
where the last term vanishes if $m=2$ or if $b\in F_{\eta}^{\tilde{m}}(\Gamma)$. We apply the induction hypothesis to \eqref{eq:thm-est} 
and insert~\eqref{est:gp} as well as Lemma~2.3 and estimates (3.36) and (3.37) of \cite{Sp1}. It follows
\begin{align} \label{est:main0}
\|&u\|^2_{G^{m-1}_\gamma(\Omega)}+ \gamma\,\|u\|^2_{\cH^m_{\ta,\gamma}(\Omega)} + \|\tr_\tau u\|_{\cH^{m}(\Gamma)}^2 
      + \sum_{p=0}^2 \|\partial_p u\|^2_{G^{m-1}_\gamma(\Omega)} \\
  &\leq  (C_{m,0}'\! +\!TC_m') \e^{(m-1)C_1T}\Big[\sum_{j = 0}^{m-1}\!\|\partial_t^j f(0)\|^2_{\cH^{m-1-j}(\RR_+^3)} 
  + \|u_0\|^2_{\cH^m( \RR^3_+)} + \|g\|^2_{\cH^m_\gamma(\Gamma)} \notag \\
  &\qquad +  \delta_{m>2}\rho^2\,\|B_1 u\|^2_{L^\infty_\gamma(\Gamma)} \Big]
  +  \frac{C_m'}{\gamma} \,(\|f\|_{\cH^m_{\gamma}(\Omega)}^2  + \|u\|_{G_\gamma^m(\Omega)}^2),   \nonumber
\end{align}
where we already have absorbed the term $c(m,r)\ep \gamma\,\|u\|_{\cH^m_{\ta,\gamma}(\Omega)}^2$ by the left-hand side fixing 
a small  $\ep=\ep(m,\eta,r)>0$.
Exactly as in (3.42) of \cite{Sp1}, we can bound the remaining derivative by
 \begin{align} \label{est:main1}
\|\partial_3^m u\|^2_{G^0_\gamma(\Omega)} 
  &\leq  (C_{m,0}'+TC_m') \e^{C_1T} \Big(\sum_{p = 0}^{2}  \|\partial_p u\|^2_{G^{m-1}_\gamma(\Omega)}
      + \|f(0)\|^2_{\cH^{m-1}(\RR_+^3)} \notag\\
 &\qquad   + \|u_0\|^2_{\cH^m( \RR^3_+)} \Big) 
   +   \frac{C_m'}{\gamma} \,(\|f\|_{\cH^m_{\gamma}(\Omega)}^2 + \|u\|_{G_\gamma^{m}(\Omega)}^2).
\end{align}
For $\gamma\ge \gamma_m$ and a sufficiently large $\gamma_m= \gamma_m( \eta,r,T')$, 
the inequalities \eqref{est:main0} and \eqref{est:main1} imply \eqref{est:main}. 
\end{proof}

Using the trace theorem and the estimate \eqref{est:main} for $m=2$, we can get rid of the extra term in this 
inequality.  

\begin{cor}\label{cor:est}
Let  the assumptions of Theorem~\ref{thm:est} be true with $m\ge3$ and $r\ge \rho$. Using the notation of this
theorem, we obtain
\begin{align} \label{est:main-cor} 
\|&u\|^2_{G^{m}_\gamma(\Omega)} + \gamma\,\|u\|^2_{\cH^m_{\ta,\gamma}(\Omega)}+ \|\tr_\tau u\|_{\cH^m(\Gamma)}^2\\
  &\leq  (C_{m,0}+TC_m) \e^{mC_1T} \Big(\sum_{j = 0}^{m-1}\|\partial_t^j f(0)\|^2_{\cH^{m-1-j}(\RR_+^3)} 
  + \|u_0\|^2_{\cH^m( \RR^3_+)}+ \|g\|^2_{\cH^m_\gamma(\Gamma)}\Big) \notag \\
  &\quad \ + C_m\e^{(m+2)C_1T}\big[ \|u_0\|^2_{\cH^2( \RR^3_+)} + \|f(0)\|^2_{\cH^{1}(\RR_+^3)} 
  + \|\partial_t f(0)\|^2_{L^2(\RR_+^3)}  + \|g\|^2_{\cH^2_\gamma(\Gamma)}\big] \notag\\ 
  &\quad \ +    \frac{C_m}{\gamma} \|f\|_{\cH^m_{\gamma}(\Omega)}^2.  \notag
\end{align}
\end{cor}
Unfortunately, the estimate \eqref{est:main-cor} does not fit to the nonlinear fixed point argument since
the constants in front of the data in $[\dots]$ depend on $r$ and thus on the size of functions inserted in the fixed point 
operator, see \eqref{est:local-inv2}.
 
\section{Linear regularity results}\label{sec:lin}
We still have to construct solutions of the linear problem \eqref{eq:maxwell-lin} in the class $G_\Sigma^m$.
In view of the localization procedure, we can focus on  the halfspace case  \eqref{eq:maxwell-lin-i}.
We start with the  $L^2$-solution from Proposition~\ref{prop:L2} and regularize it in normal, tangential, and time 
directions differently. The apriori estimates from the previous section then allow us to pass to the limit and derive
the required smoothness. Again we follow the procedure of the paper \cite{Sp1}. We  can directly use its 
results concerning normal regularity, namely Lemma~4.1 and Corollary~4.2, since they do not involve boundary conditions.

To regularize in spatial tangential variables, we make use of the norms
\begin{align*}
\|v\|^2_{\cH_\ta^k(\E)}&=\int_{\RR_+}\int_{\RR^2}(1 + |\xi|^2)^k |(\cF_2 v)(\xi,x_3)|^2 \dd\xi \dd x_3, \notag \\
\|v\|^2_{\cH_{\ta,\delta}^k(\E)} &= \int_{\RR_+} \int_{\RR^2} (1 + |\xi|^2)^{k+1} (1 + |\delta \xi|^2)^{-1} 
   |(\mathcal{F}_2 v)(\xi,x_3)|^2 \dd\xi \dd x_3
\end{align*}
for $k \in \ZZ$, $\delta > 0$, and functions $v \in \mathcal{S}'(\overline{\RR^3_+})$ whose Fourier 
transform $\mathcal{F}_2 v$ in $(x_1,x_2)$ belongs to $L^2_{\loc}(\RR^3_+)$.  
The space $\cH_\ta^k(\E)$  consists of those  $v$ with finite norm $\|v\|_{\cH_\ta^k(\E)}$. For $k\in \NN_0$ we obtain the standard 
tangential Sobolev spaces as defined before. (See Sections~1.7 and 2.4 in~\cite{Ho}.) The norm of $\cH_\ta^{k+1}(\E)$ dominates 
that of $\cH_{\ta,\delta}^k(\E)$. Conversely, if $\|v\|_{\cH_{\ta,\delta}^k(\E)}$ is bounded as  $\delta\to0$, 
then $v$ is contained in  $\cH_\ta^{k+1}(\E)$, see (2.4.4) in \cite{Ho}.

To construct mollifiers, we take a map $\chi \in C_c^\infty(\RR^2)$ such that $\cF_2 \chi(\xi) = O(|\xi|^{m+1})$ 
as $\xi \rightarrow 0$ and $\mathcal{F}_2 \chi(t \xi) = 0$ for all $t \in \RR$ implies $\xi = 0$.
Set $\chi_\ep(x) = \ep^{-2} \chi(\ep^{-1}x)$ for all $x \in \RR^2$ and $\ep > 0$. The convolution 
in spatial tangential variables $(x_1,x_2)$ by $\chi_{\ep}$ is called $J_\ep$.
We collect the properties of $J_\ep$ in the above norms which follow from
Theorems~2.4.5 and 2.4.6 in~\cite{Ho}. There it was assumed that the coefficient $A$ belongs to Schwartz' class.
An inspection of the proofs in \cite{Ho} shows that it suffices to require the regularity stated below.
\begin{lem} \label{lem:ho}
 Let $k \in \{0,\dots, m\}$,  $\delta\in(0,1)$, $v \in \cH_\ta^{k-1}(\E)$, and  $A \in F^m(\RR^3_+)$ with 
 $\partial^\alpha A \in L^2(\RR^3_+)$ for all $\alpha \in \NN_0^3$. Then there are constants $c,C>0$ not depending 
 on $\delta$ and $v$ such that
 \begin{align*}
  c \,\|v\|^2_{\cH_{\ta,\delta}^{k-1}(\E)} &\leq \|v\|^2_{\cH_{\ta}^{k-1}(\E)}
     + \int_0^1 \|J_\ep v\|_{L^2(\E)}^2 \,\ep^{-2k-1} \Big(1 + \frac{\delta^2}{\ep^2} \Big)^{-1} \dd\ep \\ 
  &\leq  C \,\|v\|^2_{\cH_{\ta,\delta}^{k-1}(\E)},\\
  \int_0^1 \| A J_\ep v - J_\ep&(A v)\|_{L^2(\E)}^2\, \ep^{-2k-1} \Big(1 + \frac{\delta^2}{\ep^2} \Big)^{-1} \dd\ep 
  \leq C\,\|v\|^2_{\cH_{\ta,\delta}^{k-2}(\E)}\,.
 \end{align*}
\end{lem}

We also use the analogous results on $\RR^2\cong\partial\E$ (dropping the subscript $\ta$) which are 
taken from Theorems~2.4.1 and 2.4.2 of \cite{Ho}. Because of the above lemma, for some time we have to work with smooth coefficients whose
derivatives of arbitrary order belong to $L^2$. An approximation argument  will bring us back to  limited regularity of 
the coefficients  later. The next result provides tangential regularity.

\begin{lem}\label{lem:reg-tan}
 Let $\eta > 0$ and $m \in \NN$. Take $A_0 \in F^{\tilde{m}}_{\eta}(\Omega)$, 
 $A_1, A_2 \in F^{\tilde{m}}_{\cf}(\RR^3_+)$,  $A_3 = A_3^{\co}$,  $D \in F^{\tilde{m}}(\Omega)$, and 
 $b\in F^{\tilde{m}}_{\cH,\eta}(\Gamma)$. We further assume that $\partial^\alpha A_i, \partial^\alpha D \in L^2(\Omega)$ 
 and $\partial^\beta b \in L^2(\Gamma)$ for all $\alpha \in \NN_0^4$, $\beta \in \NN_0^3$, and $i \in \{0,1,2\}$. Choose data $u_0 \in \cH^m(\RR^3_+)$, $f \in 
 \cH^m_\ta(\Omega)$,   and  $g \in \cH^m(\Gamma)$ with $g\cdot \nu =0$.  Let $u$ be the solution of 
 \eqref{eq:maxwell-lin-i}  from Proposition~\ref{prop:L2}. Suppose that $u$ belongs to 
$\bigcap_{j = 1}^m C^j(\ol{J}, \cH^{m-j}(\E))$ and $\tr_\tau u$ to $\bigcap_{j = 1}^m \cH^j(J, \cH^{m-j}(\partial\E))$.
 Then $u$ is an element of $C(\ol{J}, \cH^m_{\ta}(\E))$ and $\tr_\tau u$ of  $L^2(J, \cH^m(\partial\E))$.
\end{lem}
\begin{proof}
1) We first show that $u$ belongs to $L^\infty(J, \cH^m_{\ta}(\E))$ and $\tr_\tau u$ to  $L^2(J, \cH^m(\partial\E))$.
Let $\ep,\delta \in (0,1)$ and $\gamma > 0$. The generic constants below do not depend on $\delta$ or $\gamma$. 
We let $t_0=0$ for simplicity.

Applying the operators $L=L(A_0,A_1, A_2,A_3^\co,D)$ and $B=B_2^\co - \nu\times b B_1^\co$ to $J_\ep u$, we obtain
  \begin{equation}\label{eq:j-ep} \begin{split}
L J_\ep  u &= J_\ep f + \sum_{j = 0}^2 [A_j, J_\ep] \partial_j u + [D, J_\ep] u  \qquad \text{on } \Omega,\\
B J_\ep  u &= J_\ep g - \nu\times [b,J_\ep]B_1^\co u  \qquad \text{on } \Gamma,
 \end{split} \end{equation}
where $[A_j, J_\ep]= A_jJ_\ep -J_\ep A_j$ etc. Lemma~\ref{lem:ho} implies the commutator estimate
  \begin{align}\label{est:j-ep-comm}  
   \int_J e^{-2\gamma t} \int_0^1  \|[A_j, J_\ep] &\partial_j  u(t)\|^2_{L^2(\E)} \,\ep^{-2m-1} 
     \Big(1 + \frac{\delta^2}{\ep^2} \Big)^{-1} \dd\ep \dd t \\
  &\leq c\,\big(\|u\|_{L^2_\gamma(J,\cH_{\ta,\delta}^{m-1}(\E))}^2 
    + \|\partial_t u\|_{\cH_\gamma^{m-1}(\Omega)}^2\big) \notag
  \end{align}
for all $j \in \{0,1,2\}$. The other commutators are treated analogously. In particular, $L J_\ep  u$ is an element 
of $L^2(\Omega)$, $J_{\ep} u_0$ of $L^2(\E)$, and $B J_\ep  u$ of  $L^2(\Gamma)$. Hence, the 
apriori estimate from Proposition~\ref{prop:L2} can be applied to $J_\ep u$. We first use Lemma~\ref{lem:ho} to 
derive
  \begin{align}\label{est:j-ep1}
 S_\delta&:= \sup_{t \in J} \e^{-2 \gamma t} \|u(t)\|_{\cH_{\ta,\delta}^{m-1}(\E)}^2 
  + \int_0^T \e^{-2 \gamma t} \|\tr_\tau u(t)\|_{\cH_{\delta}^{m-1}(\RR^2)}^2 \dd t\\
 &\leq c \sup_{t \in J} e^{-2 \gamma t} \Big(\|u(t)\|_{\cH^{m-1}_{\ta}(\E)}^2
   + \int_0^1 \|J_\ep  u(t)\|^2_{L^2(\E)} \,\ep^{-2m-1} \Big(1 + \frac{\delta^2}{\ep^2}\Big)^{-1}\dd\ep\Big)\notag\\
   &\quad + c \int_0^T \e^{-2 \gamma t} \|\tr_\tau u(t)\|_{\cH^{m-1}(\RR^2)}^2\dd t\notag \\
   &\quad + c\int_0^1 \int_0^T \e^{-2 \gamma t} \|\tr_\tau J_\ep u(t)\|_{L^2(\RR^2)}^2\dd t
      \:\ep^{-2m-1} \Big(1 + \frac{\delta^2}{\ep^2} \Big)^{-1} \dd\ep \notag\\
&\le c\, \Big(\|u\|_{G^{m-1}_\gamma(\Omega)}^2 +    \|\tr_\tau u\|_{\cH_\gamma^{m-1}(\Gamma)}^2 
     +\int_0^1 [\|J_\ep u\|_{G^0_\gamma(\Omega)}^2 + \|J_\ep \tr_\tau u\|_{L^2_\gamma(\Gamma)}^2]\notag\\
    &\qquad \qquad \cdot\ep^{-2m-1} \Big(1 + \frac{\delta^2}{\ep^2} \Big)^{-1} \dd\ep\Big).  \nonumber
   \end{align}
  By Proposition~\ref{prop:L2} and \eqref{eq:j-ep}  there are constants $C_0,\gamma_0 > 0$ such that the term 
  in brackets $[\dots]$  is  bounded by $C_0$ times
\begin{align*}
 &\|J_\ep u_0\|_{L^2(\E)}^2  + \int_0^T\e^{-2\gamma t} \big(\|J_\ep g(t)\|_{L^2(\RR^2)}^2 
    + \| [b, J_\ep] B_1^\co u(t)\|_{L^2(\RR^2)}^2\big)\dd t\\
 &  + \frac1\gamma\!\int_0^T\!\!\! \e^{-2\gamma t}\Big[\|J_\ep f(t)\|_{L^2(\E)}^2 
   \!+\! \|[D,\! J_\ep] u(t)\|_{L^2(\E)}^2 
   \!+ \! \sum_{j=0}^2 \|[A_j,\! J_\ep ] \partial_ju(t)\|_{L^2(\E)}^2\Big]\D t 
   \end{align*}
  for all $\gamma \geq \gamma_0$.  We insert these quantities in \eqref{est:j-ep1} 
  and interchange $\dd t$ and $\dd \ep$. Combined with Lemma~\ref{lem:ho} and  \eqref{est:j-ep-comm}, it follows 
 \begin{align}\label{est:j-ep2}
  S_\delta &\le c\, \Big[\|u\|_{G^{m-1}_\gamma(\Omega)}^2 + \|\tr_\tau u\|_{\cH^{m-1}_\gamma(\Gamma)}^2
      +\|u_0\|_{\cH^{m}(\E)}^2 + \gamma^{-1}\, \big(\|f\|_{\cH^m_\gamma(\Omega)}^2 \\
&\ \ +  \|u\|_{L^2_\gamma(J, \cH_{\ta,\delta}^{m-1}(\E))}^2 \!\!+ \|\partial_t u\|_{\cH_\gamma^{m-1}(\Omega)}^2\big)
  \!+ \|g\|_{\cH^{m}_\gamma(\Gamma)}^2 
   \!+ \|\tr_\tau u\|_{L^2_\gamma(J, \cH_{\delta}^{m-2}(\RR^2))}^2\Big].\notag
\end{align}
The last summand is bounded by $\|\tr_\tau u\|_{\cH^{m-1}_\gamma(\Gamma)}^2$. We can absorb the term with $u$ 
in the regularized norm by the left-hand side choosing a sufficiently large $\gamma$, depending on $T'$. As a result, the 
quantity $S_\delta$ is bounded uniformly in $\delta\in(0,1)$. 
We conclude that $u$ belongs to  $L^\infty(J, \cH^m_{\ta}(\E))$. Fatou's lemma further yields
\begin{align*}
\|\tr_\tau u&\|_{L^2_\gamma(J, \cH^m(\RR^2))}^2 
  = \int_0^T \e^{-2\gamma t} \int_{\RR^2} (1+|\xi|^2)^m \,|\cF_2 (\tr_\tau u)(t,\xi)|^2 \dd \xi \dd t\\
 &\le \liminf_{\delta \to 0} \int_0^T \e^{-2\gamma t} \int_{\RR^2} (1+|\xi|^2)^m 
      (1+|\delta\xi|^2)^{-1}  \,|\cF_2(\tr_\tau u)(t,\xi)|^2 \dd \xi \dd t\\
 & =    \liminf_{\delta \to 0}    \|\tr_\tau u\|_{L^2_\gamma(J,\cH_{\delta}^{m-1}(\RR^2))}^2.
 \end{align*}
 The right-hand side is finite because of estimate \eqref{est:j-ep2}, and so $\tr_\tau u$ belongs to 
 $L^2_\gamma(J, \cH^m(\partial\E))$.

\smallskip
    
2) Step 1) and Corollary~4.2 of \cite{Sp1} imply that $u$ is an element of $\cH^m(\Omega)$.  To show 
$u\in C(\ol{J}, \cH^m_\ta(\E))$, we fix  $\alpha \in \NN_0^4$ with $|\alpha| = m$ and 
$ \alpha_0 = \alpha_3 = 0$.
 Using that $u\in\cH^m(\Omega)$ solves \eqref{eq:maxwell-lin-i} and that $\tr_\tau u \in\cH^m(\Gamma)$, we derive
 \begin{align*}
  L \partial^\alpha u  &= \partial^\alpha f  - \sum_{j = 0 }^2 \sum_{0 < \beta \leq \alpha} \binom{\alpha}{\beta} 
           \partial^\beta A_j \,\partial^{\alpha - \beta} \partial_j u 
  - \sum_{0 < \beta \leq \alpha} \binom{\alpha}{\beta} \partial^\beta D \partial^{\alpha - \beta} u
  =: f_\alpha,\\
 B \partial^\alpha u &= \partial^\alpha g + \nu \times \sum_{0 < \beta \leq \alpha}\binom{\alpha}{\beta} 
         \partial^\beta b \,\tr_t \partial^{\alpha - \beta} u^1=:g_\alpha,\\
  \partial^\alpha u(0)&= \partial^\alpha u_0.       
 \end{align*}
 The function $f_\alpha$ belongs to  $L^2(\Omega)$, $g_\alpha$ to $L^2(\Gamma)$, and $\partial^\alpha u_0$
 to $L^2(\E)$, cf.\ \eqref{est:f-u0-alpha} and~\eqref{est:g-alpha}. 
 Proposition~\ref{prop:L2} then shows that $\partial^\alpha u$ is contained in $G^0(\Omega)$, as required.
  \end{proof}
  
The next lemma allows us to gain one derivative in time. In the proof one constructs a solution $v$
to the initial boundary value problem which  $\partial_t u$ formally satisfies. 
One then checks that the time integral of $v$ coincides with $u$. 
Here and in the next proposition the compatibility conditions enter in a crucial way.

\begin{lem}\label{lem:reg-time}
  Let $\eta > 0$ and $J=(t_0,T)$. Take coefficients $A_0 \in F^{3}_{\eta}(\Omega)$,  $A_1, A_2 \in F^{3}_{\cf}(\RR^3_+)$,  
  $A_3 = A_3^{\co}$,  $D \in F^{3}(\Omega)$, and  $b\in F^{3}_{\cH,\eta}(\Gamma)$.  Choose data 
  $u_0 \in \cH^1(\RR^3_+)$, $f \in \cH^1(\Omega)$, and $g \in \cH^1(\Gamma)$ with  $g\cdot \nu=0$. Assume that 
  the tuple $(t_0,A_0,A_1,A_2,A_3,D,b,u_0,f,g)$ fulfills the compatibility conditions~\eqref{eq:cc-lin} of order $1$. 
 Let $u$ be the solution of \eqref{eq:maxwell-lin-i}  from Proposition~\ref{prop:L2}.  
 Assume that $u$ belongs to $G^1_\Sigma(J' \times \RR^3_+)$ whenever
 $u \in C^1(\overline{J'}, L^2(\E))$ and $\tr_\tau u\in \cH^1(J', L^2(\partial \E))$ for every open 
 interval $J' \subseteq J$. Then $u$ is contained in $G^1_\Sigma(\Omega)$.
\end{lem}

\begin{proof}
 1) Without loss of generality we assume $J = (0,T)$. Take $r > 0$ with
 \begin{align*}
	&\Fnorm{3}{A_i} \leq r, \quad \Fnorm{3}{D} \leq r,\quad \|b\|_{F^{3}_\cH(\Omega)}\le r,\\ 
	&\max\{\Fvarnorm{2}{A_i(t)},\max_{j\in\{1,2\}} \Hhn{2-j}{\partial_t^j A_0(t)}\} \leq r,\\
	&\max\{\Fvarnorm{2}{D(t)}, \max_{j\in\{1,2\}} \Hhn{2-j}{\partial_t^j D(t)}\} \leq r
\end{align*}
for all $i \in \{0, 1, 2\}$ and $t \in \ol{J}$. Let $\gamma = \gamma_1(\eta, r,T)\ge 1$ be defined by
Theorem~\ref{thm:est}. We further choose a number $C_0 = C_0(\eta,r,T)\ge1$ dominating
the constants in~\eqref{est:S-lin}, Proposition~\ref{prop:L2}, and Theorem~\ref{thm:est}.
We finally set 
 \begin{align*}
  R_1 = \e^{2 \gamma T} C_0(\|f\|_{G^0_\gamma(\Omega)}^2 +  \|f\|_{\cH^1_\gamma(\Omega)}^2
      + \|g\|_{\cH^1_\gamma(\Gamma)}^2 +  \|u_0\|_{\cH^1(\E)}^2).
 \end{align*}

 
2) Take an initial  time  $t_0 \in \ol{J}$ and assume that $u(t_0)$ belongs to $\cH^1(\E)$ with 
$\|u(t_0)\|_{\cH^1(\E)}^2 \le R_1$. Choose a time step 
\[0<T_s\le \min\{1, (6C_0r^2)^{-1}\}.\] 
Following  the proof of Lemma~4.5 of \cite{Sp1}, we want to construct a function $v$ in $C([t_0, T_s'], L^2(\E))$ 
satisfying
 \begin{align}\label{eq:DtD}
\tilde{L} v   &= \partial_{t} f - \partial_t D \Big(\int_{t_0}^t v(s) \dd s +  u(t_0) \Big) 
      \qquad \text{on }  \Omega',\notag\\
B v &= \partial_t g +\nu\times \partial_t b\Big(\int_{t_0}^t B_1^\co v(s) \dd s +  B_1^\co u(t_0)\Big) 
  \qquad \text{on }  \Gamma',\\
    v(t_0) &= S_{1,1, A_j, D}(t_0, f, u(t_0))=:v_0 \qquad \text{on } \RR^3_+.\notag
 \end{align}
Here we define $T_s' =\min\{t_0 + T_s,T\}$,  $J'=(t_0,T_s')$, $\Omega'=J'\times \E$, $\Gamma'=J'\times\partial\E$,
and $\tilde{L}  = L(A_0,A_1,A_2, A_3, \partial_t A_0 + D)$. 
To solve \eqref{eq:DtD}, we set 
\[R= (4C_0^2 (1+C_0) +6 + 6C_0 (1+c_{\tr}^2)r^2) R_1\ge R_1,\]
 where $c_{\tr}$ is the norm of the trace operator from $\cH^1(\E)$ to $L^2(\partial\E)$.
 Let $E$ be the closed ball in $G^0_{\Sigma,\gamma}(\Omega')$ with radius $R^{1/2}$ and center 0. Take $w\in E$. 
 We look at the problem
 \begin{align}\label{eq:DtD-fix}
\tilde{L} v   &= \partial_{t} f - \partial_t D \Big(\int_{t_0}^t w(s) \dd s +  u(t_0) \Big) =:f_w
   \qquad \text{on }  \Omega',\notag\\
 B v &= \partial_t g +\nu\times \partial_t b\Big(\int_{t_0}^t B_1^\co w(s) \dd s +  B_1^\co u(t_0)\Big) =:g_w
    \qquad \text{on }  \Gamma',\\
    v(t_0) &= v_0\qquad \text{on } \RR^3_+.\notag
 \end{align}
 Note that the data in the above problem fulfill  the assumptions of Proposition~\ref{prop:L2}. (Use  Lemma~2.3 of 
\cite{Sp1} for  the initial value.) This proposition thus provides a unique solution 
$\Phi(w)\in G^0_{\Sigma,\gamma}(\Omega')$ of \eqref{eq:DtD-fix}. As in step I in the proof of Lemma~4.5 of \cite{Sp1}, 
 Proposition~\ref{prop:L2} and~\eqref{est:S-lin} imply that
 \begin{align*}
  \|&\Phi(w)\|_{G^0_{\Sigma,\gamma}(\Omega')}^2 \leq C_0 \, \big(\|v_0\|^2_{L^2(\E)} + \|f_w\|_{L^2_\gamma(\Omega')}^2 
   +  \|g_w\|_{L^2_\gamma(\Gamma')}^2\big)\\
  &\le 2C_0^3\,(\|f(0)\|_{L^2(\E)}^2+ \|u(t_0)\|_{\cH^1(\E)}^2) + 3C_0\,(\|\partial_{t} f\|_{L^2_\gamma(\Omega')}^2 
         + \|\partial_{t}g\|_{L^2_\gamma(\Gamma')}^2)\notag\\
 &\quad + 3 C_0 r^2  \int_{t_0}^{T_s'} \e^{-2\gamma t} \Big[T_s\int_{t_0}^t \big(\|w(s)\|_{L^2(\E)}^2 
 + \|\tr_\tau w(s)\|_{L^2(\partial \E)}^2\big)\dd s   \\
    &\qquad \qquad + \|u(t_0)\|_{L^2(\E)}^2 + \|\tr_\tau u(t_0)\|_{L^2(\partial \E)}^2  \Big] \dd t\\
    &\le 2C_0^2(1+C_0) R_1 + 3R_1 + 3 C_0 (1+c_{\tr}^2)r^2R_1T_s + 3C_0r^2 T_s \, \|w\|_{G^0_{\Sigma,\gamma}(\Omega')}^2 \le R.\notag
 \end{align*}
  Hence, $\Phi$ maps $E$ into itself. In a similar way we estimate 
 \begin{align*}
 \|\Phi(w)-\Phi(\hat w)\|_{G^0_{\Sigma,\gamma}(\Omega')}^2 
 &\leq C_0 \, (\|f_w-f_{\hat w}\|_{L^2_\gamma(\Omega')}^2 +  \|g_w- g_{\hat w}\|_{L^2_\gamma(\Gamma')}^2) \\
 & \le r^2C_0 T_s \,\|w-\hat w\|_{G^0_{\Sigma,\gamma}(\Omega')}^2  \le \tfrac{1}{6}\, \|w-\hat w\|_{G^0_{\Sigma,\gamma}(\Omega')}^2 
  \end{align*}
 for all $w, \hat w \in E$. The contraction mapping principle thus gives a unique function $v \in E$ 
 solving \eqref{eq:DtD}.
 
\smallskip
 
 3) In this step we assume that $u(t_0)$ belongs to $\cH^1(\E)$  with 
 $\|u(t_0)\|_{\cH^1(\E)}^2  \le R_1$ and that the tuple $(t_0, A_0,A_1,A_2, A_3, D, b,f,g, u(t_0))$ 
 fulfills the compatibility conditions~\eqref{eq:cc-lin} of order one; i.e., $B(t_0) u(t_0) = g(t_0)$ 
 on $\partial \RR^3_+$. 
 We use the solution $v\in G^0_\Sigma(\Omega')$ of \eqref{eq:DtD} on the time interval $J'$ from part~2) and define
 \[  w(t) = u(t_0) + \int_{t_0}^t v(s) \dd s\]
 for $t\in\ol{J'}$.  Step II) of the proof of Lemma~4.5 of \cite{Sp1} shows that $L(A_j,D)w=f$ 
 on $\Omega'$ and $w(t_0)=u(t_0)$. Using also \eqref{eq:DtD},  we  compute
 \begin{align*}
 B(t)w(t) &= B(t_0) u(t_0)+ (B(t)-B(t_0))u(t_0) +\int_{t_0}^t B_2^\co v(s)\dd s\\
 &\qquad + (b(t)B_1^\co(w(t)-u(t_0))\times \nu\\
 &= g(t_0) + (b(t)- b(t_0))B_1^\co u(t_0) \times \nu + (b(t)B_1^\co(w(t)-u(t_0))\times \nu\\
  &\qquad +\int_{t_0}^t\partial_t g(s)\dd s
     -\int_{t_0}^t (\partial_t b(s)B_1^\co w(s) +b(s)  B_1^\co \partial_t w(s))\times \nu \dd s  \\
 &= g(t)
 \end{align*}
  for $t\in\ol{J'}$. The  uniqueness statement in Proposition~\ref{prop:L2} thus yields that $u=w$ on $J'$
  (where we use the obvious variant of this result with initial time $t_0$). We conclude that
  $u \in C^1(\overline{J'}, L^2(\E))$ and $\tr_\tau u\in \cH^1(J', L^2(\partial \E))$ as $\tr_\tau$ commutes 
  with integration in time, and hence
   $u$ belongs to $G^1_\Sigma(\Omega')$ by the assumption. 
   
The assertion now follows by an iteration argument as in Step III) of the proof in Lemma~4.5 of \cite{Sp1}.
(Here the exponential factor in the definition of $R_1$ comes into play.)
\end{proof}

For smooth coefficients we now obtain the desired regularity properties on~$\E$. 

\begin{prop}\label{prop:reg-smooth}
Let $\eta > 0$ and $m \in \NN $. Take $A_0 \in F^{\tilde{m}}_{\eta}(\Omega)$ with $\partial_t A_0 \in F^{\max\{m-1,3\}}_{\eta}(\Omega)$, 
$A_1,A_2\in F^{\tilde{m}}_{\cf}(\RR^3_+)$, $A_3=A_3^{\co}$, $D\in F^{\tilde{m}}(\Omega)$, 
and $b\in F_{\cH,\eta}^{\tilde{m}}(\Gamma)$. We assume that $\partial^\alpha A_i, \partial^\alpha D \in L^2(\Omega)$ 
and $\partial^\beta b \in L^2(\Gamma)$ for all $\alpha \in \NN_0^4$, $\beta \in \NN_0^3$, and $i \in \{0,1,2\}$. 
Choose data $u_0 \in \cH^m(\RR^3_+)^6$,  
$f \in \cH^m(\Omega)^6$,  and  $g \in \cH^m(\Gamma)^3$ with $g\cdot \nu =0$  such that the tuple 
$(t_0, A_j, D, b, u_0, f, g)$  fulfills the compatibility conditions~\eqref{eq:cc-lin} of order $m$. 
Then the solution $u$ of \eqref{eq:maxwell-lin-i} from Proposition~\ref{prop:L2} belongs to
$G_\Sigma^m(\Omega)$.
 \end{prop}
\begin{proof}
We proceed as in the proof of Proposition~4.7 of \cite{Sp1}, letting $t_0=0$ for simplicity. 
The result for $m=1$ is a consequence of the two 
previous lemmas and Lemma~4.1 of \cite{Sp1}. We assume that we have shown the assertion for some $m\in \NN$ and 
that the assumptions are satisfied for $m+1$. Hence the solution $u$ of  \eqref{eq:maxwell-lin-i}  belongs to 
$G^m_\Sigma(\Omega)$. With the notation of the proof of Lemma~\ref{lem:reg-time}, we see that the function
$v=\partial_t u$ 
fulfills 
\begin{align}\label{eq:DtD1}
\tilde{L} v   &= \partial_{t} f - \partial_t D u=:f_u \qquad \text{on }  \Omega',\notag\\
    B v &= \partial_t g +\nu\times \partial_t b \, B_1^\co u=:g_u  \qquad \text{on }  \Gamma',\\
    v(0) &= S_{m+1,1, A_j, D}(0, u_0,f)=:v_0 \qquad \text{on } \RR^3_+.\notag
 \end{align}
As in the proof of Proposition~4.7 of \cite{Sp1} one can check that the coefficients and data in \eqref{eq:DtD1}
satisfy the regularity  assumptions of the induction hypothesis. 
For the compatibility conditions,  we note that Lemma~4.6 of \cite{Sp1} yields 
 \[ S_{m,p,A_j,\tilde{D}}(0, v_0, f_u)=  S_{m+1,p+1,A_j,D}(0, u_0, f)\]
 for all $p\in\{0,1,\dots, m-1\}$ and $\tilde{D} :=\partial_t A_0 +D$. Equations \eqref{eq:S-lin} and 
 \eqref{eq:cc-lin} thus imply
\begin{align}\label{eq:cc-reg}
B(0) S_{m,p,A_j,\tilde{D}}(0, v_0, f_u) &= B(0)  S_{m+1,p+1,A_j,D}(0, u_0, f)\\
  &= \partial_t^{p+1} g(0)
     + \nu\times \sum_{k=1}^{p+1} \binom{p+1}{k} \partial^k_t b(0)  B_1^\co \partial_t^{p+1-k}u(0).\notag
\end{align}     
On the other hand, by means of  $v=\partial_t u$  we calculate
\begin{align*}
\partial_t^{p} g_u(0)
     + & \,\nu\times \sum_{k=1}^{p} \binom{p}{k} \partial^k_t b(0)  B_1^\co \partial_t^{p-k}v(0)\\
&=\partial_t^{p+1} g(0) + \nu\times \sum_{k=0}^p \binom{p}{k} \partial^{k+1}_t b(0)  B_1^\co \partial_t^{p-k} u(0)\\
 &\qquad   +  \nu\times \sum_{k=1}^{p} \binom{p}{k} \partial^k_t b(0)  B_1^\co \partial_t^{p-k+1} u(0)\\
  &=\partial_t^{p+1} g(0)
     + \nu\times \sum_{k=1}^{p+1} \binom{p+1}{k} \partial^k_t b(0)  B_1^\co \partial_t^{p+1-k}u(0).
\end{align*} 
Combined with \eqref{eq:cc-reg}, we have established the compatibility condition \eqref{eq:cc-lin} 
of order $m$ for \eqref{eq:DtD1}. 

The induction hypothesis now shows that $\partial_t u$ belongs to $G^m_\Sigma(\Omega)$. By Lemma~4.1 of \cite{Sp1}, 
Lemma~\ref{lem:reg-tan}, and the fact that $\tr_\tau$ commutes with differentiation in time, 
the map $u$ is thus contained in $G^{m+1}_\Sigma(\Omega)$.
\end{proof}

Above we have assumed extra smoothness of the coefficients. This assumption can be removed by an approximation argument. 
 Take $A_0 \in F^{\tilde{m}}_{\eta}(\Omega)$, $A_1,A_2\in F^{\tilde{m}}_{\cf}(\RR^3_+)$, $A_3=A_3^{\co}$, 
 $D\in F^{\tilde{m}}(\Omega)$, and  $b\in F^{\tilde{m}}_{\cH,\eta}(\Gamma)$.
 Using standard methods, one constructs functions
$A_{0,\ep} \in F^{\tilde{m}}_{\eta}(\Omega)$,  $A_{1,\ep},A_{2,\ep}\in F^{\tilde{m}}_{\cf}(\E)$, 
$D_{\ep}\in F^{\tilde{m}}(\Omega)$, and $b_\ep\in F_{\cH,\eta}^{\tilde{m}}(\Gamma)$ for $\ep>0$, 
which are uniformly bounded in the respective  $F$--space 
and tend uniformly to $A_0$, $A_1$, $A_2$, $D$ and $b$, respectively, as $\ep\to0$. Moreover, all their partial derivatives 
are contained in the respective $F$--space and thus in~$L^2$. The analogous results are true for $A_{0,\ep}(0)$ and $D_\ep(0)$ with $\tilde{m}$ replaced by 
 $\tilde{m}-1$. (Compare Lemma~2.2 of \cite{Sp1}.) 
 
 We again choose data $u_0 \in \cH^m(\RR^3_+)$,  
$f \in \cH^m(\Omega)$,  and  $g \in \cH^m(\Gamma)$ with $g\cdot \nu =0$  such that the tuple 
$(0, A_0,A_1, A_2,A_3, D, b,u_0, f, g)$  fulfills the compatibility conditions~\eqref{eq:cc-lin} of order $m$. 
To use the approximating coefficients, one has to modify the initial value in such a way that 
\eqref{eq:cc-lin} is still satisfied.

\begin{lem}\label{lem:cc-ep}
Under the above assumptions, there is a number $\ep_0>0$ and functions $u_{0,\ep}$ in $\cH^m(\E)$ for 
$\ep\in(0,\ep_0]$ such that the compatibility conditions~\eqref{eq:cc-lin} of order $m$
are satisfied by the  tuple $(0, A_{0,\ep},A_{1,\ep}, A_{2,\ep},A_3, D_\ep, b_\ep,   u_{0,\ep},f, g)$.
Moreover, the new initial values $u_{0,\ep}$ tend to $u_0$ in $\cH^m(\E)$ as $\ep\to0$.
\end{lem}
\begin{proof}
Slightly modifiying the notation in \eqref{eq:S-lin}, \eqref{def:S-lin} and \eqref{eq:cc-lin}, we  set
\begin{align*}
S_{m, p}^\lin(u_0) &= S_{m,p,A_j,D}(0,u_0,f) = \partial_t^p u(0) 
\end{align*}
for $p\in\{0,\dots,m-1\}$ as $t_0 = 0$ and $f$ remain fixed. We further define 
\[\cB=\begin{pmatrix} -I & B_0^\co b(0) \end{pmatrix} \qquad \text{and} \qquad
    \cB_\ep=\begin{pmatrix} -I & B_0^\co b_\ep(0) \end{pmatrix}. \]
 The  compatibility conditions~\eqref{eq:cc-lin} can be rewritten as
\begin{align}\label{eq:cc-lin1}
&\cB A_3^\co  S_{m, p}^\lin(u_0) = \partial_t^p g(0)
     + \nu\times \sum_{k=1}^p \binom{p}{k} \partial^k_t b(0)  B_1^\co S_{m, p-k}^\lin(u_0),\\
  &\cB  A_3^\co    \Big[S_{m, p}^\lin(u_0)  
     + \Big(0 , \sum_{k=1}^p \binom{p}{k} \partial^k_t b(0)  B_1^\co S_{m, p-k}^\lin(u_0)\Big)^T\Big]
       = \partial_t^p g(0) \notag
\end{align}
on $\partial\E$. Here and below, sums from 1 to 0 or from 0 to $-1$ are defined as zero. 
Here we understand $B_1^\co$ just as matrix and not as a trace operator. Since $\partial_t^k b(0) \in \cH^{\tilde{m}-k-1/2}(\RR^2)$ 
for all $k \in \{0, \ldots, \tilde{m}-1\}$, Theorem~2.5.7 in~\cite{Ho} yields a function $\tilde{b} \in \cH^{\tilde{m}}(\RR^3_+)$ 
such that $\tr_{\partial \E} \partial_3^k \tilde{b} = \partial_t^k b(0)$ for all $k \in \{0, \ldots, \tilde{m}-1\}$. 
In particular, we can extend $\partial_t^k b(0)$ by $\partial_3^k \tilde{b}$ to a function in $\cH^{\tilde{m}-k}(\RR^3_+)$. 
We write $\tilde{S}_{m, p}^\lin(u_0)$ for the term $[\dots]$ in \eqref{eq:cc-lin1}, where we extend it 
to a function on $\RR^3_+$ as described above.
Following (4.35) of \cite{Sp1}, this term is expanded as
\beq\label{eq:tilde-S} 
\tilde{S}_{m, p}^\lin(u_0) =\cA^p\partial_3^p u_0 + \sum_{j=0}^{p-1} \tilde{C}_{p,p-j} \partial_3^j u_0
                                + \tilde{B}_pf, 
\eeq
where $\cA:= -A_0(0)^{-1} A_3$, the tangential differential operators $\tilde{C}_{p,p-j}$ belong to
$\cB(\cH^{m-j}(\E),\cH^{m-p}(\E))$ and $\tilde{B}_p$ is an element of $\cB(\cH^m(\Omega), \cH^{m-p}(\E))$. 
These mapping properties can be shown using Lemma~2.1 and (4.35) of \cite{Sp1} and  the regularity 
of $\partial_3^k \tilde{b}$.

By $\cB_\ep$, $\tilde{S}_{m, p,\ep}^\lin$, $\cA_\ep$ and $\tilde{C}_{p,p-j,\ep}$ we denote the variants of the above operators 
for $A_{j,\ep}$, $D_\ep$ and $\tilde b_\ep$, where we obtain $\tilde b_\ep$ as $\tilde b$ above. 
As in the previous paragraph, one sees that the functions $\tilde{C}_{p,p-j,\ep}$ are  bounded in 
$\cB(\cH^{m-j}(\E),\cH^{m-p}(\E))$  uniformly in $\ep>0$ by the properties of the coefficients. 
We further note that $\partial_t^k b_\ep(0)$ converges 
to $\partial_t^k b(0)$ as $\ep \rightarrow 0$, since $b_\ep$ tends to $b$ in $\cH^{\tilde{m}}(\Gamma)$.
The construction of $\tilde{b}_\ep$ respectively $\tilde{b}$ in~\cite{Ho} thus yields that $\tilde{b}_\ep$ 
converges to $\tilde{b}$ in $\cH^{\tilde{m}}(\RR^3_+)$ as $\ep \rightarrow 0$.
One can thus show that in $\cH^{m-p}(\E)$ the maps $\tilde{S}_{m, p,\ep}^\lin(u_0)$ tend to  $\tilde{S}_{m, p}^\lin(u_0)$  
and $(\tilde{b}-\tilde{b}_\ep)B_1 \tilde{S}_{m, p}^\lin(u_0)$ to 0  as $\ep\to0$. 

  We are looking for functions $u_{0,\ep}\in  \cH^m(\E)$ satisfying
  \[\cB_\ep A_3^\co\tilde{S}_{m, p,\ep}^\lin(u_{0,\ep})
  = \partial_t^p g(0) = \cB A_3^\co  \tilde S_{m, p}^\lin(u_0)\quad \text{on } \partial \E. \]
Let $u_{0,\ep}=u_0 +h$ for some $h\in  \cH^m(\E)$. The $\ep$--variant of \eqref{eq:tilde-S} yields 
\[\tilde{S}_{m, p,\ep}^\lin(u_0+h) =\tilde{S}_{m, p,\ep}^\lin(u_0) + \cA_\ep^p \partial_3^p h 
  +\sum_{j=0}^{p-1} \tilde{C}_{p,p-j,\ep} \partial_3^j h.\]
  We thus have to find a map $h\in  \cH^m(\E)$ solving
\begin{align}\label{eq:cc-lin-ep}
\cB_\ep A_3^\co \cA_\ep^p \partial_3^p h 
 &= \cB_\ep A_3^\co  \Big[\tilde{S}_{m, p}^\lin(u_0)
  - \tilde{S}_{m, p,\ep}^\lin(u_0) - \sum_{j=0}^{p-1} \tilde{C}_{p,p-j,\ep} \partial_3^j h \Big]\notag\\
  &\qquad \qquad  + (\cB-\cB_\ep)A_3^\co \tilde{S}_{m, p}^\lin(u_0) \notag \\
  &= \cB_\ep A_3^\co  \Big[\tilde{S}_{m, p}^\lin(u_0)
  - \tilde{S}_{m, p,\ep}^\lin(u_0) - \sum_{j=0}^{p-1} \tilde{C}_{p,p-j,\ep} \partial_3^j h \\
&\qquad\qquad      + \big(0, (\tilde{b}-\tilde{b}_\ep)B_1^\co \tilde{S}_{m, p}^\lin(u_0)\big)^T\Big].\notag
 \end{align}
 
Similar as in the proof of Lemma~4.8 in \cite{Sp1}, we first construct functions $a_\ep^p\in \cH^{m-p}(\E)$
for $p\in \{0,\dots,m-1\}$, $\ep\in(0,\ep_0)$ and some $\ep_0>0$, which satisfy the variant of \eqref{eq:cc-lin-ep}
where we drop $\cB_\ep$ and replace $\partial_3^j h$ by $a_\ep^j\in\cH^{m-j}(\E)$. Moreover, the functions $a_\ep^p$
tend to 0 in $\cH^{m-p}(\E)$ as $\ep\to0$.  

This is done via induction over $p$. For $p=0$, set  $a_\ep^0= (0,(\tilde{b}-\tilde{b}_\ep)B_1^\co u_0)^T$ in $\cH^{m}(\E)$.
Let the functions $a_\ep^k$ be constructed  for $0\le k\le p-1<m-1$. The right-hand side of \eqref{eq:cc-lin-ep} without $\cB_\ep$
is equal to $A_3^\co v_\ep$ for maps $v_\ep$, which tend to 0 in  
$\cH^{m-p}(\E)$ as $\ep\to0$ by the above observations. Lemma~4.9 of \cite{Sp1} now yields  functions $a_\ep^{p}$ 
for $\ep\in(0,\ep_0]$ and some $\ep_0>0$ such that $a_\ep^{p}\to0$ in $\cH^{m-p}(\E)$ as $\ep\to0$ and $A_3^\co \cA_\ep^p a_\ep^{p}= A_3^\co v_\ep$.
 So the maps $a_\ep^p$ exist.

Again by Theorem~2.5.7 in~\cite{Ho}, we can find functions $h_\ep\in\cH^m(\E)$ satisfying $\tr_{\partial \E} \partial_3^p h_\ep = \tr_{\partial\E} a_\ep^p$
for $p\in\{0,\dots, m-1\}$. Moreover, $h_\ep$ converges to 0 in $\cH^m(\E)$ as $\ep\to0$.
The maps $h_\ep$ thus satisfy \eqref{eq:cc-lin-ep} and $u_{0,\ep}= u_0+h_\ep$ fulfills the variant of 
\eqref{eq:cc-lin1} for the approximating coefficients. 
\end{proof}

Based on the results established so far, we can now derive the desired regularity result.
Recall that we allow for $G=\E$. 

\begin{thm}\label{thm:reg}
Let $T\in (0,T')$, $\rho,\eta > 0$,  $r \geq r_0 > 0$, and $m \in \NN $. Take coefficients 
$A_0 \in F^{\tilde{m}}_{\eta}(J\times G)$,  $D \in F^{\tilde{m}}(J\times G)$, $A_3 = A_3^{\co}$, and 
$b\in F_{\cH,\eta}^{\tilde{m}}(J\times \Sigma)$. If $G = \E$, pick $A_1, A_2 \in F^{\tilde{m}}_\cf(\RR^3_+)$. 
Otherwise, let $A_1 = A_1^\co$ and $A_2 = A_2^\co$. Assume that the coefficients satisfy
\begin{align*}
	&\|A_i\|_{F^{\tilde{m}}(J\times G)} \leq r, \quad \|D\|_{F^{\tilde{m}}(J\times G)} \leq r,\quad 
	\|b\|_{F^{\tilde{m}-1}(J\times \Sigma)}\le r,
	\quad [b]_{\cH^{\tilde{m}}(J\times \Sigma)}\le \rho, \\ 
&\max\{\|A_i(0)\|_{F^{\tilde{m}-1,0}(J\times G)},\max_{j\in\{1,\dots, m-1\}}
     \|\partial_t^j A_0(0)\|_{ \cH^{\tilde{m}-1-j}(G)}\} \leq r_0,\\
	&\max\{\|D(0)\|_{F^{\tilde{m}-1,0}(J\times G)},\max_{j\in\{1,\dots, m-1\}} 
	   \|\partial_t^j D(0)\|_{ \cH^{\tilde{m}-1-j}(G)}\} \leq r_0,
\end{align*}
for $i \in \{0,1,2\}$.
Choose data $u_0 \in \cH^m(G)^6$,  $f \in \cH^m(J\times G)^6$,  and $g \in \cH^m(J\times \Sigma)^3$ with 
$g\cdot \nu =0$ such that the tuple $(0, A_0,A_1, A_2,A_3,D,b,u_0,f,g)$ 
 fulfills the compatibility conditions~\eqref{eq:cc-lin} of order $m$.

 Then there is a unique solution $u\in G^{m}_\Sigma(J\times G)$ of~\eqref{eq:maxwell-lin} and there is a number 
 $\gamma_m= \gamma_m( \eta,r,T')\ge 1$ such that
\begin{align} \label{est:mainG}
\|u&\|^2_{G^{m}_\gamma(J\times G)} + \|\tr_\tau u\|_{\cH^m(J\times \Sigma)}^2\\
  &\leq  (C_{m,0}+TC_m) \e^{mC_1T} \Big(\sum_{j = 0}^{m-1}\|\partial_t^j f(0)\|^2_{\cH^{m-1-j}(G)} 
  + \|u_0\|^2_{\cH^m(G)}+ \|g\|^2_{\cH^m_\gamma(J\times \Sigma)} \notag \\
  &\qquad +  \delta_{m>2}\rho^2 \,\|B_1 u\|_{L^\infty_\gamma(J\times \Sigma)}^2\Big) +
    \frac{C_m}{\gamma} \|f\|_{\cH^m_{\gamma}(J\times G)}^2 \nonumber
\end{align}
for all $\gamma \geq \gamma_m$ and with constants $C_i = C_i(\eta, r, T')\ge 1$ for $i\in \{1,m\}$ 
and  $C_{m,0} = C_{m,0}(\eta, r_0,\|b\|_\infty)\ge1$. If $b$ even belongs to 
$F_{\eta,\tau}^{\tilde{m}}(J\times \Sigma)^ {3\times3}$ with norm less or equal $r$, then one can set $\rho=0$ in the 
above inequality. For $m\ge 3$ and $r\ge \rho$ we further have the estimate
\begin{align} \label{est:mainG-cor} 
\|&u\|^2_{G^{m}_\gamma(J\times G)} + \|\tr_\tau u\|_{\cH^m(J\times \Sigma)}^2\\
  &\leq  (C_{m,0}+TC_m) \e^{mC_1T} \Big(\sum_{j = 0}^{m-1}\|\partial_t^j f(0)\|^2_{\cH^{m-1-j}(G)} 
  + \|u_0\|^2_{\cH^m(G)}+ \|g\|^2_{\cH^m_\gamma(J\times \Sigma)}\Big) \notag \\
  &\qquad + C_m\e^{(m+2)C_1T}\big( \|u_0\|^2_{\cH^2(G)} + \|f(0)\|_{\cH^{1}(G)}^2 + \|\partial_t f(0)\|^2_{L^2(G)} 
  + \|g\|^2_{\cH^2_\gamma(J\times \Sigma)}\big) \notag\\ 
  &\qquad +    \frac{C_m}{\gamma} \|f\|_{\cH^m_{\gamma}(J\times G)}^2.  \notag
\end{align}
\end{thm}

\begin{proof}
We only sketch the proof  since it is very similar to those of Theorems~4.10 and 1.1
of \cite{Sp1}. We first treat the localized problem \eqref{eq:maxwell-lin-i} on $\E$.
We take approximating data as in Lemma~\ref{lem:cc-ep} for $\ep=\frac1n$. Proposition~\ref{prop:reg-smooth}
then provides solutions $u_n\in G^m_\Sigma(\Omega)$ which are uniformly bounded in this 
space due to \eqref{est:main} and~\eqref{est:main-cor}. From  Banach-Alaoglu we thus obtain  a weak$^*$ accumulation
point $u$ of $(u_n)$ which belongs to $\tilde{G}^m_\Sigma(\Omega)$.  We apply  \eqref{est:apriori-0} with the given coefficients 
to the difference $u_n-u$. By means of  the uniform convergence of 
the coefficients,  it follows that the maps $u_n$ tend to $u$ in  $G^0_\Sigma(\Omega)$. Using this fact, one sees
that $u$ satisfies \eqref{eq:maxwell-lin-i}. 

To show that $u\in G^m_\Sigma(\Omega)$, one first applies  $\partial^{m-1}_t$ to the system \eqref{eq:maxwell-lin-i}.
The resulting data satisfy the  compatibility conditions~\eqref{eq:cc-lin} of order 1 as
the given data fulfill them up to order $m$. 
 Since $\partial^{m-1}_t u\in\tilde{G}^1_\Sigma(\Omega)$,  as in step~2) 
of the proof of Lemma~\ref{lem:reg-tan} we can  deduce that $\partial^{m-1}_t u$ belongs to $C(\ol{J}, \cH^1_\ta(\E))$. 
By Lemma~4.1 of \cite{Sp1}, the function $\partial^{m-1}_t u$ is an element of $G^1(J'\times \E)$
provided that $\partial_t^{m-1} u\in C^1(\ol{J'}, L^2(\E))$ for any open interval $J'\sub J$.
 Our Lemma~\ref{lem:reg-time} then shows that $\partial^{m-1}_tu$ is contained in  $G^1_\Sigma(\Omega)$.
As in  the proof of Theorem~4.10 in \cite{Sp1}, one now inductively infers that $u$ belongs to $G^m_\Sigma(\Omega)$.

Finally, one passes to the domain $G\neq \RR_+^3$ by a localization argument. See steps IV--VI of the proof of Theorem~5.6 
in \cite{Sp0} or of Theorem~3.1 in \cite{SS}.
\end{proof}

If $G$ is unbounded, the above  result imposes decay of the derivatives of $A_0$ and $D$ as $x\to\infty$.
Actually, if these derivatives are bounded one obtains the same results much easier. As in \cite{Sp1}, we have thus 
focused on the case treated in the theorem and describe the easy extension in the next result.

\begin{rem}\label{rem:bdd}
Let $G$ be unbounded. As in  Remark~1.2 of \cite{Sp1} we can weaken the regularity assumptions  in 
Theorem~\ref{thm:reg} to  $A_0, D\in W^{1,\infty}(J\times G)$ and
\beq\label{def:f-var}
\forall \;\alpha\in\NN_0^4 \ \text{ with } 1\le |\alpha|\le m:  \
  \partial^\alpha A_0,\partial^\alpha D\in L^\infty(J,L^2(G))+L^\infty(J\times G).
\eeq  
One further has to assume that the corresponding norms of $A_0$ and $D$ are less or equal $r$, and 
$\partial^\alpha A_0(0)$ and $\partial^\alpha D(0)$ are bounded in $L^2(G)+L^\infty(G)$ by $r_0$ 
for all $|\alpha|\le m-1$. Here one can also replace $G$ by $\E$. The conditions on $b$ remain unchanged.
Throughout, in new terms involving bounded parts of $\partial^\alpha A_0$ and 
$\partial^\alpha D$ these derivatives can easily be  estimated by their sup-norms.\hfill $\lozenge$
\end{rem}
 
\section{Local existence and uniqueness}\label{sec:local}

 The apriori estimate of Theorem~\ref{thm:reg}  does not allow us to treat the nonlinear absorbing boundary
 conditions as described in \eqref{eq:maxwell} and \eqref{ass:main} in full generality. The problem arises
in the terms with highest derivatives of $b=\zeta(B_1 \hat{u})$ for a function
$\hat{u}\in \tilde{G}_{\Sigma}^m(J\times G)$ with range in $\mathcal{U}$. For simplicity we first look at 
the problem on $\Gamma$, the  case $J\times \Sigma$ then follows by the localization procedure described in 
Section~\ref{sec:aux} and in \cite{Sp0}. Lemma~2.1 in \cite{Sp2} yields the formula 
\begin{align}\label{eq:zeta(v)}
\partial^\alpha \zeta(B_1 \hat{u}) = \sum_{\substack{\beta, \gamma \in \NN_0^3, \beta_0 = 0 \\  \beta + \gamma = \alpha}} \,
\sum_{1\leq j\leq |\gamma|}\,&\sum_{\substack{\gamma_1,\ldots,\gamma_j\in\NN_0^3\setminus\{0\}\\ 
   \sum\gamma_i= \gamma}}	 \, \sum_{l_1,\ldots,l_j = 1}^2 C(\alpha, \gamma_1, \ldots, \gamma_j) \\
	  &\hspace{2em}\cdot (\partial_{y_{l_j}} \cdots \partial_{y_{l_1}} \partial_x^{(\beta_1,\beta_2)} \zeta)(B_1 \hat{u})
	      \prod_{i=1}^j \partial^{\gamma_i} (B_1 \hat{u})_{l_i},\notag
	\end{align}
for $\alpha \in \NN_0^3$ with $|\alpha|\le m$. We take $m\ge3$. 
This expression can be written as a sum $S_{\textrm{main}}$ of 
the terms with  $\alpha =\gamma$ and $j=1$ plus the sum $S_{\textrm{rem}}$ of the other terms. 
The summands in   $S_{\textrm{rem}}$  can be estimated  by the norm in $\cH^1(\E)$ of 
the product $\prod_{i=1}^j \partial^{\gamma_i} (B_1 \hat{u})_{l_i}$ using the trace theorem. 
The product rules in Lemma~2.1 in \cite{Sp1} (and localization) lead to  the inequality
\beq \label{est:rho0}
\|S_{\textrm{rem}}\|_{L^2(J \times\Sigma)}^2 \le T\,\|S_{\textrm{rem}}\|_{L^\infty(J,L^2(\Sigma))}^2
\le cT\,(1+\|\hat u\|_{G^m(J\times G)}^{2m})=\tilde{C}(R)T,
\eeq
where $\|\hat{u}\|_{G^m_\Sigma(J\times G)}\le R$.
If $|\alpha|<m$, one can estimate the full function $\partial^\alpha \zeta(B_1 \hat{u})$ in this way.

To treat the term $S_{\textrm{main}} = \partial_\xi \zeta(B_1 \hat{u}) \partial^\alpha (B_1 \hat{u})$ in the case $|\alpha|=m$,  we define the quantities
\beq\label{def:z}
z_0(\ol{\kappa})= \max_{x\in\Sigma, |\xi|\le\ol{\kappa}} |\partial_\xi\zeta(x,\xi)|, \qquad
z(\ol{\kappa})= z_0(\ol{\kappa})\,\ol{\kappa}
\eeq
for $\ol{\kappa}\ge0$. 
Let $\|B_1 \hat{u}(0)\|_{L^\infty(J\times \Sigma)}\le \ol{\kappa}$. Assuming also that $\|\hat{u}\|_{G^m_\Sigma(J\times G)}\le R$, 
we estimate
\begin{align*}
 \|\partial_\xi \zeta(B_1 \hat{u})\|_{L^\infty(J \times \Sigma)} &\leq \|\partial_\xi \zeta(B_1 \hat{u}(0))\|_{L^\infty(\Sigma)} + T \|\partial_\xi^2 \zeta(B_1 \hat{u}) B_1 \partial_t \hat{u} \|_{L^\infty(J \times \Sigma)} \\
    &\leq z_0(\ol{\kappa}) + \tilde{C}(R) T \|\hat{u}\|_{G^3(J \times G)} \leq z_0(\ol{\kappa}) + \tilde{C}(R) T.
\end{align*}
We derive
\[ \|S_{\textrm{main}} \|_{L^2(J\times \Sigma)}^2
  \le c(z_0(\ol{\kappa})^2 + \tilde{C}(R) T^2) \,\|B_1 \hat{u}\|_{\cH^m(J\times \Sigma)}^2.  \]
It thus follows
\beq\label{est:rho}
\sum_{|\alpha|=m}\|\partial^\alpha \zeta(B_1 \hat{u})\|_{L^2(J\times \Sigma)}^2  
 \le \ol{C} (T + z_0(\ol{\kappa})^2). 
\eeq
We further take functions $\hat{v}$ and $v$ with ranges in $\cU$,
$\|B_1 \hat{v}(0)\|_{L^\infty(J\times \Sigma)}\le \ol{\kappa}$, 
$\|\hat{v}\|_{G^m_\Sigma(J\times G)}\le R$, and analogously for $v$.
In a similar way we estimate 
\beq\label{est:rho1}
\|(\zeta(B_1 \hat{u}) -\zeta(B_1\hat {v}))B_1v\|_{\cH^{m-1}_\gamma (J\times \Sigma)}^2  
 \le \ol{C}(T + z(\ol{\kappa})^2) \,\|\hat{u}-\hat{v}\|_{G^{m-1}_{\Sigma,\gamma}}^2\,.
\eeq
The constant $\ol{C}$ depends on $R$, $\zeta$, and a time $T' > T$. In the fixed point argument, one part of the resulting 
right-hand sides can be made small choosing a small time interval $(0,T)$ depending 
on the radius $R$. For the other one we will have to assume that $z(\ol{\kappa})$ is small, which either means 
that we are close to a linear boundary condition or that we deal with electric fields having uniformly small 
tangential traces initially. 
In the linear case, where $\zeta$ does not depend on the state 
$u$, the number $z(\ol{\kappa})$ is even 0. Here we actually allow for time depending coefficients $b=\zeta$ 
in $F^{\tilde{m}}_{\eta,\tau}(J\times \Sigma)$, see \eqref{ass:main1}. Inequalities like \eqref{est:rho} and \eqref{est:rho1} 
are used several times below.

Exactly as in Remark~2.5 of \cite{Sp2},  for unbounded $G$ in our proofs we will make a simplifying assumption 
on the coefficients $\chi$ and $\sigma$ in order to avoid certain easier terms in the calculations.

\begin{rem}\label{rem:bdd1}
Let $G$ be unbounded, $m\ge 3$, $\hat u\in \tilde{G}^m(J\times G)$, and $\chi$ and $\sigma$ be given 
by \eqref{ass:main} or \eqref{ass:main1}. As noted in Remark~\ref{rem:bdd}, for our linear 
results we can admit coefficients $ A_0=\chi(\hat u)$ and  $D= \sigma(\hat u)$ belonging to the space
described in \eqref{def:f-var}. The additional bounded terms can easily be estimated 
in each computation. Without loss of generality, in the proofs we will therefore exclude such terms 
by imposing extra decay on	the space derivatives of $\chi$ and $\sigma$ as $|x| \rightarrow \infty$. 
 More precisely, for all multiindices $\alpha \in \NN_0^9$ with $\alpha_4 = \ldots = \alpha_9 = 0$ and  
 $1 \leq |\alpha| \leq m$, $R > 0$, $V \Subset \mathcal{U}$,	and $v \in L^\infty(J, L^2(G))$ with range in $V$ 
 and	$\|v\|_{L^\infty(J, L^2(G))} \leq R$ we require that
	\begin{align}\label{ass:f-var}
&(\partial^\alpha \chi)(v), (\partial^\alpha \sigma)(v) \in L^\infty(J, L^2(G)), \nonumber\\
&\|(\partial^\alpha \chi)(v)\|_{L^\infty(J, L^2(G))} + \|(\partial^\alpha \sigma)(v)\|_{L^\infty(J,L^2)(G))} \leq C,
	\end{align}
	where $C = C(\chi,\sigma,m,R,V)$.	With this assumption, Lemma~2.1 of \cite{Sp2} yields
	that $\chi(\hat{u})$ and $\sigma(\hat{u})$ are contained in  $F^m(J\times G)$. \hfill $\lozenge$
\end{rem}

We start with the uniqueness of solutions to \eqref{eq:maxwell}.

\begin{lem} \label{lem:unique}
 Let $t_1> t_0$ in $\RR$ and  $J = (t_0, t_1)$.
 Assume that either \eqref{ass:main} or  \eqref{ass:main1} is valid. 
 Let $u_1,u_2 \in G^3_\Sigma(J\times G)$ solve \eqref{eq:maxwell} with inhomogeneity $f$, 
 boundary value $g$, and initial value $u_0$ at initial time $t_0$. If assumption \eqref{ass:main}
 is satisfied, we require that $z_0(\ol{\kappa}_1)^2 \ol{\kappa}_2^2 \le (2C_0)^ {-1}$ 
 where $\ol{\kappa}_1 \geq \max_{j\in\{1,2\}}\|B_1 u_j\|_{L^\infty(J \times \Sigma)}$, 
 $\ol{\kappa}_2 \geq \min_{j\in\{1,2\}}\|B_1 u_j\|_{L^\infty(J \times \Sigma)}$,
 and $C_0$ is taken from \eqref{est:unique} and depends on the norm of $u_1$ and $u_2$ in $W^{1,\infty}(J\times G)$
 and on the lower bound $\eta$ of $\chi$ and $\zeta$. Then $u_1 = u_2$.
\end{lem}
\begin{proof}
We focus on the assumption \eqref{ass:main} of a nonlinear boundary condition, since the linear one
in \eqref{ass:main1} is easily treated as in Lemma~3.2 of \cite{Sp2}.
Let $T_0\in \ol{J}$ be the supremum of all $t\in \ol{J}$ such that $u_1 = u_2$ on  $[t_0, t]$.
The two functions coincide on $[t_0, T_0]$ by their continuity.

We suppose that $T_0 < t_1$. We take a time $T' \in (T_0, t_1)$ and set $J'=(T_0, T')$.  We fix a compact set $V\subset \cU$ 
containing the ranges of $u_1$ and $u_2$ on $J'$. The maps $u_1$ and $u_2$
in $G^3_\Sigma(J'\times G)$ both solve \eqref{eq:maxwell} on $\ol{J'}$ with the same initial value $u_1(T_0)$, 
inhomogeneities $f$ and $g$, and the operators $L_j = L(\chi(u_j),A_1^{\co}, A_2^{\co}, A_3^{\co}, \sigma(u_j))$ 
and $B^j=B(u_j)$ for $j=1$, respectively $j=2$. Without loss of generality we assume that 
$\|B_1 u_2\|_{L^\infty(J \times \Sigma)} = \min_{j \in \{1,2\}} \|B_1 u_j\|_{L^\infty(J \times \Sigma)}$.
The difference $u_1-u_2$ fulfills
\begin{align*}
L_1 (u_1-u_2) &= (\chi(u_2)- \chi(u_1)) \partial_t u_2 + (\sigma(u_2)- \sigma(u_1))u_2=:f_1\qquad \text{on } J'\times G,\\
B_1(u_1-u_2) &= ((\zeta(B_1u_2)-\zeta(B_1u_1))B_1 u_2)\times \nu=:g_1 \qquad \text{on } J'\times \Sigma,\\
(u_1-u_2)(T_0)&=0 \qquad \text{on } G.
\end{align*}
Lemma~2.1 of \cite{Sp2} and Sobolev's embedding theorem
 yield that $\chi(u_j)$ and  $\sigma(u_j)$ are elements of $F^3(J\times G)$,  and $\zeta(u_j)$  of  
 $F^3_\cH(J\times \Sigma)$.  Moreover, $\chi(u_1)$ and  $\zeta(u_1)$ are symmetric and bounded from below by $\eta>0$.
 Let $r$ (resp.\ $R$) be the maximum of the norms of $u_1$ and $u_2$ in $W^{1,\infty}(J\times G)$ 
 (resp.\ in $G^3_\Sigma(J\times G)$). Then  the Lipschitz norms of $\chi(u_1)$ and  $\sigma(u_1)$ and the sup-norm of $\zeta(u_1)$
 are bounded by a constant depending on $r$.
 Proposition~\ref{prop:L2} now provides constants $C_0= C_0(\eta,r)$ and $\gamma= \gamma(\eta,r)$  such that
 \beq \label{est:unique}
 \|u_1-u_2\|_{G^0_{\Sigma,\gamma}(J'\times G)}^2  \le C_0\,( \|f_1\|_{L^2_\gamma(J'\times G)}^2 + \|g_1\|_{L^2_\gamma(J'\times \Sigma)}^2).
 \eeq
 Exactly as in the proof of Lemma~3.2 in \cite{Sp2}, we can control
 \beq \label{est:unique-f}
  \|f_1\|_{L^2_\gamma(J'\times G)}^2 \le c(\eta,R) (T'-T_0)\,\|u_1-u_2\|_{G^0_\gamma(J'\times G)}^2\,.
\eeq
Recalling the definition of $\ol{\kappa}_j$ in the statement and of $z$ in \eqref{def:z}, we next derive
\beq \label{est:unique-g}
  \|g_1\|_{L^2_\gamma(J'\times \Sigma)} ^2\le z_0(\ol{\kappa}_1)^2 \ol{\kappa}_2^2 \, \|\tr_\tau(u_1-u_2)\|_{L^2_\gamma(J'\times G)}^2\,.
 \eeq
By the assumption on $\ol{\kappa}_j$, we can choose $T'>T_0$ so  small that \eqref{est:unique},  \eqref{est:unique-f}, and \eqref{est:unique-g}
imply that  $u_1 = u_2$ on $[T_0, T']$ and thus on $[t_0, T']$. This result contradicts the definition 
 of $T_0$, and hence $u_1 = u_2$ on $J$.
\end{proof}

We next construct local in time solutions of \eqref{eq:maxwell}  using Banach's fixed point theorem
and our linear result Theorem~\ref{thm:reg}.  Special care in the treatment of the constants is required
to close the argument, and we need the structure of the estimate \eqref{est:main} here. For the data we define the quantity 
\beq \label{def:d}
	d_k(J):=  \|u_0\|_{\cH^k(G)}^2 + \sum_{j = 0}^{k-1} \|\partial_t^j f(t_0)\|_{\cH^{k-1-j}(G)}^2 +\|f\|_{\cH^k(J\times G)}^2 
 + \|g\|_{\cH^k(J\times \Sigma)}^2.
 \eeq
 Moreover, $C_S$ is the norm of the Sobolev embedding $\cH^2(G)\hra C_b(\ol{G})$. We note that below the number $C_0$ 
 only depends on a radius $r_3\ge d_3(J)^{1/2}$ instead of $r$, as an inspection of the proof shows. 
  
 \begin{thm}\label{thm:local}
 Let $t_0 \in \RR$, $T > 0$,  $J = (t_0,t_0 + T)$, and $m \in \NN$ with $m \geq 3$.
 Assume that either \eqref{ass:main} or  \eqref{ass:main1} is valid. 
 Choose data $u_0 \in \cH^m(G)^6$,  $f \in \cH^m(J\times G)^6$,  and $g \in \cH^m(J\times \Sigma)^3$ with 
$g\cdot \nu =0$ such that the tuple $(t_0, \chi, \sigma,\zeta,u_0,f,g)$ 
 fulfills the compatibility conditions~\eqref{eq:cc-nl} of order $m$. Pick a radius $r > 0$
 satisfying
 \[ 
  d_m(J)\le r^2.
 \]
 Take a number $\kappa >0 $ with
 \begin{equation*}
  \dist(\{u_0(x) \,|\, x \in G\}, \partial \mathcal{U}) > \kappa.
 \end{equation*}
 If \eqref{ass:main} is valid, we take $\wt{\kappa} > 0$ with
 \beq \label{ass:z}  
 z(\wt{\kappa})\le \min\big\{ \tfrac18 (C_{m,0}\ol{C})^{-1/2},\,(2C_0)^{-1/2} \big\}.
\eeq 
and assume that $\|B_1 u_0\|_{L^\infty(\Sigma)} < \wt{\kappa}$.
The constants $C_0$,  $C_{m,0}$, 
and $\ol{C}$ depending on $\chi$, $\sigma$, $\zeta$, $m$, $r$, $\kappa$, and $T$
are given by Lemma~\ref{lem:unique}, \eqref{def:Cm0-local-existence}, \eqref{est:rho}, and  \eqref{est:rho1}.

 Then there is a time
 $\tau = \tau(\chi,\sigma,\zeta,m,T,r,\kappa,\wt{\kappa}) > 0$ such that the nonlinear initial boundary value 
 problem~\eqref{eq:maxwell} with data $f$, $g$, and $u_0$ 
 has a solution $u$ on $[t_0, t_0 + \tau]$ which belongs to $G^m_\Sigma((t_0, t_0 + \tau) \times G)$. 
 It is unique among those solutions with $\|B_1 u\|_{L^\infty((t_0, t_0 + \tau) \times \Sigma)} < \wt{\kappa}$.
 \end{thm}

\begin{proof}
1) We focus on the assumption \eqref{ass:main} of a nonlinear boundary condition, since the linear one
in \eqref{ass:main1} can be treated as in Theorem~3.3 of \cite{Sp2}.
Without loss of generality we assume $t_0 = 0$ and, if $G$ is unbounded, that 
$\chi$ and $\sigma$ satisfy \eqref{ass:f-var}, cf.\ Remark~\ref{rem:bdd1}. Moreover, the quantities $d_k(J)$ can be chosen to 
be positive since $u=0$ is the unique solution of~\eqref{eq:maxwell} for $(u_0, f,g)=0$ by Lemma~\ref{lem:unique}. 

Let $\tau \in (0,T]$ and $R>0$. We introduce $J_\tau = (0, \tau)$ and
\[ V_{\kappa} = \{y \in \mathcal{U}\,|\,  \dist(y, \partial \mathcal{U}) \geq \kappa \} 
       \cap \overline{B}(0,C_S r).\]
Note that $\ran(u_0)$ is contained in the compact set $V_\kappa$. 
Our fixed point space is
 \begin{align*}
  E(R,\tau) &= \{v \in \tilde{G}^m_\Sigma(J_\tau\times G)\,|\, \|v\|_{G^m_\Sigma(J_\tau\times G)} \leq R, 
    \; \|v - u_0\|_{L^\infty(J_\tau \times \ol{G})} \leq \kappa/2, \\
  &\qquad\qquad\qquad\qquad\qquad\partial_t^j v(0) = S_{m,j,\chi,\sigma}(0,u_0,f) 
   \text{ for } 0\le j\le m-1\}
 \end{align*}
 endowed  with the metric induced by the norm of  $\tilde{G}^{m-1}_\Sigma(J_\tau\times G)$.  We have
 \beq \label{est:ERtau}
 \ran(v)\subseteq \tilde{V}_\kappa := V_\kappa + \ol{B}(0,\kappa/2)\sub \cU  
 \eeq  
for $v\in E(R,\tau)$. As in Lemma~2.6 in \cite{Sp2}, starting from Lemma~\ref{lem:nonempty}
below one can construct a function  $w\in\tilde{G}^m_\Sigma(J_\tau\times G)$ satisfying the initial conditions in $E(R,\tau)$. 
Using Lemma~\ref{lem:nonempty} and the estimates on $S_{m,j,\chi,\sigma}$ from Lemma~2.4 in \cite{Sp2}, one obtains a constant 
$C_{\ref{lem:nonempty}}=C_{\ref{lem:nonempty}}( \chi, \sigma, m, T,r, \kappa)$ such that 
$\|w\|_{\tilde{G}^m_\Sigma}\le  C_{\ref{lem:nonempty}} r$. Take $R > C_{\ref{lem:nonempty}} r$.  Since 
\[w(t)-u_0=\int_0^t \partial_s w(s)\dd s,\]
we can bound $\|w - u_0\|_{L^\infty(J_\tau \times \ol {G})}\le C_SR\tau$.
As a result, $E(R,\tau)$ is non-empty, if we choose $R > C_{\ref{lem:nonempty}} r$ and $\tau\in(0, \kappa /(2C_SR)]$.
It is straightforward to show the completeness of $E(R,\tau)$ 
for its metric by means of the Banach-Alaoglu theorem, cf.\ the proof of Theorem~3.3 in \cite{Sp2}.

\smallskip
 
2)  Let $\hat{u} \in E(R,\tau)$. Take $\eta > 0$ from \eqref{ass:main}. Then $A_0:=\chi(\hat{u})$ is contained in 
$F^{m}_{\eta} (J_\tau \times G)$, $b:=\zeta(B_1\hat{u})$ in $F^{m}_{\cH,\eta} (J_\tau \times \Sigma)$,
and $D:=\sigma(\hat{u})$ in $F^{m}(J_\tau \times G)$  by Lemma~2.1 of \cite{Sp2}, Remark~\ref{rem:bdd1},
 Sobolev's embedding, and the remarks before \eqref{est:rho}. The tuple 
 $(0,\chi(\hat{u}), A_1^{\co}, A_2^{\co}, A_3^{\co},\sigma(\hat{u}),\zeta(\hat{u}),u_0,f,g)$ 
satisfies the linear compatibility conditions~\eqref{eq:cc-lin} due to Lemma~2.6 of \cite{Sp2},
the initial conditions in $E(R,\tau)$, and formula \eqref{eq:zeta(v)}.
Theorem~\ref{thm:reg} yields a solution $u \in G^m_\Sigma(J_\tau\times G)$ of the system~\eqref{eq:maxwell-lin}
with the coefficients $A_0$, $D$, $b$ and the data $u_0$, $f$, $g$. In this way one defines
a mapping $\Phi \colon \hat{u} \mapsto u$ from $E(R,\tau)$ to  $G^m_\Sigma(J_\tau\times G)$. We want to prove 
that $\Phi$ is a strict contraction on $E(R,\tau)$ for a suitable radius $R$ and a sufficiently small time step $\tau$.
 
To this aim, take numbers $\tau \in (0,T]$ with $\tau\le \kappa/(2C_SR)$ and $R > C_{\ref{lem:nonempty}} r$ 
which will be fixed below. Let $\hat{u} \in E(R,\tau)$. Because of \eqref{est:ERtau},
 the map  $\zeta(\hat{u})$ is bounded by a constant $c(\tilde{V}_\kappa)$.
As in step~II) of the proof of Theorem~3.3 in \cite{Sp2}, one finds radii
$r_0 = r_0(\chi,\sigma,\zeta,m,r,\kappa)$ and $R_1 = R_1(\chi, \sigma,\zeta, m, R, \kappa,T)$  such that
\begin{align}\label{est:local-coeff-init}
 &\max\{\|\chi(\hat{u})(0)\|_{F^{m-1,0}(G)}, \ \max_{1 \leq l \leq m-1} \|\partial_t^l \chi(\hat{u})(0)\|_{\cH^{m-l-1}(G)}\} \leq r_0, \nonumber \\
 &\max\{\|\sigma(\hat{u})(0)\|_{F^{m-1,0}(G)},\ \max_{1 \leq l \leq m-1} \|\partial_t^l \sigma(\hat{u})(0)\|_{\cH^{m-l-1}(G)}\} \leq r_0,\\
 \label{est:local-coeff1}
 &\|\chi(\hat{u})\|_{F^m(J\times G)},\  \|\sigma(\hat{u})\|_{F^m(J\times G)},\  \|\zeta(\hat{u})\|_{F^{m-1}(J\times G)}  \leq R_1.
\end{align}
Moreover, the relations \eqref{est:rho}  and \eqref{est:ERtau} imply the bound
\beq\label{est:local-coeff2}
\sum_{|\alpha|=m}\|\partial^\alpha \zeta(B_1 \hat{u})\|_{L^2(J\times \Sigma)}^2  \le \ol{C}(\zeta,R,T,\kappa)(\tau + z_0(\wt{\kappa})^2).
\eeq

Let the constant 
\begin{equation}
\label{def:Cm0-local-existence}
  C_{m,0} = C_{m,0}(\chi,\sigma,\zeta,r,\kappa) = C_{m,0}(\eta(\chi,\zeta), r_0(\chi,\sigma,\zeta,m,r,\kappa),c(\tilde{V}_\kappa))
\end{equation}
be given by Theorem~\ref{thm:reg}. The radius $R=R(\chi,\sigma,\zeta, m,r,\kappa,T)$  for $E(R,\tau)$ is now defined as
\begin{align} \label{def:local-R}
 R= \max\Big\{\sqrt{32 C_{m,0}}\, r, \, C_{\ref{lem:nonempty}} r + 1\Big\}.
\end{align}
Let $\gamma_m = \gamma_m(\chi, \sigma,\zeta, T,r,\kappa)$ and $C_m = C_m(\chi,\sigma,\zeta,T,r,\kappa)$ be the
constants from Theorem~\ref{thm:reg} with $\eta(\chi,\zeta)$ and $R_1(\chi, \sigma, \zeta, m, R, \kappa,T)$.
Lemma~2.1 in \cite{Sp1} yields product rules and Corollary~2.2 in \cite{Sp2} Lipschitz bounds of composition 
operators. We write $C_{2.1,\text{\cite{Sp1}}}$ for the maximum of the constants in Lemma~2.1 in \cite{Sp1} and
$C_{2.2, \text{\cite{Sp2}}}$ for that of Corollary~2.2 in \cite{Sp2}  applied to our material laws and with
the numbers $m$ and $R$ and the set  $\tilde{V}_\kappa$. 
We finally introduce the parameter $\gamma = \gamma(\chi, \sigma, \zeta,m, T, r,\kappa)$ and the time step 
$\tau = \tau(\chi, \sigma, \zeta,m, T, r,\kappa,\wt{\kappa})$  by
\begin{align}
 \gamma &= \max\Big\{\gamma_m, \,C_{m,0}^{-1} C_m \Big\} ,  \label{def:local-gamma} \\
  \tau &= \min \Big\{T, \frac{\kappa}{2C_SR}\,, \frac{\ln 2}{2 \gamma + m C_1}\,, \frac{C_{m,0}}{C_m}, 
    [16 C_{m,0}\ol{C}(4 \wt{\kappa}^2 + C_S^2 T (T + z_0(\wt{\kappa})^2))]^{-1}\,, \notag\\
       & \hspace*{2cm} (16 C_{m,0} \ol{C})^{-1}, [32  R^2 C_{m,0} C_{2.1,\text{\cite{Sp1}}}^2 C_{2.2, \text{\cite{Sp2}}}^2 ]^{-1}\Big\}.  \label{def:local-tau}
       \end{align}
3) With the definitions and  notations of step~2), Theorem~\ref{thm:reg}, \eqref{est:local-coeff2}, and 
Sobolev's embedding
yield
\begin{align} \label{est:local-inv2}
&\|\Phi(\hat u)\|^2_{G^{m}_\Sigma(J_\tau \times G)} \le  \e^{2\gamma\tau}\,\|\Phi(\hat u)\|^2_{G^{m}_{\Sigma,\gamma}(J_\tau \times G)}\notag\\
  &\leq  (C_{m,0}+\tau C_m) \e^{(mC_1+2\gamma)\tau} \Big(\sum_{j = 0}^{m-1}\|\partial_t^j f(0)\|^2_{\cH^{m-1-j}(G)} 
  + \|u_0\|^2_{\cH^m(G)}\notag \\
  &\quad + \|g\|^2_{\cH^m_\gamma(J_\tau \times \Sigma)} + \ol{C}(\tau + z_0(\wt{\kappa}^2)) \|B_1 \Phi(\hat{u})\|_{L^\infty_\gamma(J_\tau \times \Sigma)}^2 \Big) +\frac{C_m}{\gamma} \e^{2\gamma\tau} \,\|f\|_{\cH^m_{\gamma}(J_\tau \times G)}^2\notag \\
    &\le 2C_{m,0} \cdot 2 \Big(r^2 + \ol{C}(\tau + z_0(\wt{\kappa})^2)(2\|B_1 u_0\|_{L^\infty(\Sigma)}^2 + 2 \tau^2 \|\partial_t (B_1 \Phi(\hat{u}))\|_{L^\infty(J_\tau \times \Sigma)}^2 )\Big) \nonumber \\
    &\le 8 C_{m,0}\Big(r^2 + \ol{C}(\tau + z_0(\wt{\kappa})^2)(\wt{\kappa}^2 + \tau^2 C_S^2 \|\Phi(\hat{u})\|^2_{G^{m}_\Sigma(J_\tau \times G)})\Big).
\end{align}
Employing~\eqref{ass:z} and $R \geq 1$, we thus obtain
\begin{align*}
 \|\Phi(\hat u)\|^2_{G^{m}_\Sigma(J_\tau \times G)} 
 &\le 16 C_{m,0}\Big(r^2 + \ol{C}(\tau + z_0(\wt{\kappa})^2)\wt{\kappa}^2\Big) \nonumber \\
 &\le \frac{R^2}{2} + 16 C_{m,0} \ol{C} \, \wt{\kappa}^2 \tau + 16 C_{m,0} \ol{C} z(\wt{\kappa})^2 
 \le R^2.
\end{align*}
Step~III) of the proof of Theorem~3.3 in \cite{Sp2} shows that the map $\Phi(\hat u)$ satisfies the 
initial and sup-norm conditions in $E(R,\tau)$. So we have shown that $\Phi$ maps  $E(R,\tau)$ into itself.

Take $\hat{u},\hat{v}\in  E(R,\tau)$. Set $u=\Phi(\hat{u})$ and $v=\Phi(\hat{v})$. As above, we look at the linear
system~\eqref{eq:maxwell-lin} with coefficients $A_0=\chi(\hat{u})$, $D=\sigma(\hat{u})$, and $b=\zeta(B_1\hat{u})$. The difference $v-u$
solves this system with inhomogeneities 
\begin{align*}
\tilde{f} &= (\chi(\hat{u}) - \chi(\hat{v})) \partial_t v + (\sigma(\hat{u}) - \sigma(\hat{v})) v,\qquad
\tilde{g} = B_0  (\zeta(B_1\hat{u}) - \zeta(B_1\hat{v})) B_1 v,
\end{align*}
cf.\ in  step~IV) of the proof of Theorem~3.3 in \cite{Sp2}. Proceeding as in this step and in \eqref{est:local-inv2},  from Theorem~\ref{thm:reg} we deduce  the estimate  
\begin{align}\label{est:local-contr}
\|\Phi(\hat v) &-\Phi(\hat u)\|^2_{G^{m-1}_\Sigma(J_\tau\times G)} 
\le  \e^{2\gamma\tau}\,\|\Phi(\hat v)- \Phi(\hat u)\|^2_{G^{m-1}_{\Sigma,\gamma}(J_\tau\times G)}\notag \\
 & \le \frac14\, \|\hat v -\hat u\|^2_{G^{m-1}_\Sigma(J_\tau \times G)} 
       + 4C_{m,0}\ol{C}(\tau + z(\wt{\kappa})^2)\|\hat{u}-\hat{v}\|_{G^{m-1}_\Sigma(J_\tau \times G)}^2\notag\\
 &\le \frac34\, \|\hat v -\hat u\|^2_{G^{m-1}_\Sigma(J_\tau \times G)}\,,
\end{align}
employing also \eqref{est:rho1}. (The last part of \eqref{def:local-tau} enters when  using the arguments of \cite{Sp2}.)
Together with Lemma~\ref{lem:unique}, the result is proven.
\end{proof}

We add a lemma used in the proof of  Theorem~\ref{thm:local}.

\begin{lem}\label{lem:nonempty}
  Let $m \in \NN$ and $k \in \{0, \ldots, m-1\}$. Take maps $h_k$ in $\cH^{m-k}(\RR^3_+)$.
  Then there is a function $u \in G^m(\RR \times \RR^3_+)$ such that $\partial_t^k u(0) = h_k$ for all $k$,
the trace of $u$ on $\partial \RR^3_+$ belongs to $\cH^m(\RR \times \partial \RR^3_+)$, and 
  \begin{align}
  \label{EquationEstimatesForu}
   \|u&\|_{G^m(\RR \times \RR^3_+)} + \|\partial_t^j u\|_{L^2(\RR, \cH^{m+1/2 - j}(\RR^3_+))}
   +\|\operatorname{tr}_{\partial \RR^3_+} u\|_{\cH^m(\RR \times \partial \RR^3_+)}\notag\\
   &\leq c \sum_{k = 0}^{m-1} \|h_k\|_{\cH^{m-k}(\RR^3_+)}, 
  \end{align}
  for all $j \in \{0, \ldots, m\}$ and a constant $c=c(m)$.  
 \end{lem}

 \begin{proof}
  Let $k \in \{0, \ldots, m-1\}$. Take $g_k \in \mathcal{S}(\RR^3)$. Fix a map $\psi \in C_c^\infty(\RR)$ which
 equals $1$ in $(-\frac12,\frac12)$ and vanishes on $\RR\setminus (-2,2)$. We define the function $v$ by
 \begin{equation*} 
  v(t,x) = \mathcal{F}^{-1}\Big(\psi((1+ |\cdot|^2)^{1/2} t) \sum_{k = 0}^{m-1} \hat{g}_k \frac{t^k}{k!} \Big)(x), \qquad (t,x) \in \RR^4,
 \end{equation*}
  where $\mathcal{F}$ and the hat denote the spatial Fourier transform. 
 Observe that we apply $\cF^{-1}$ to a function in $\mathcal{S}(\RR^4)$.
 The dominated convergence theorem  yields
 \begin{equation*} 
  \partial_t^k v(0) = g_k
 \end{equation*}
 for all $k$. To show~\eqref{EquationEstimatesForu} for $v$ and $g_k$, 
  we take $j \in \{0, \ldots, m\}$ and compute
 \begin{align*} 
  &\|\partial_t^j v\|_{L^\infty(\RR, \cH^{m-j}(\RR^3))}^2 = \sup_{t \in \RR} \int_{\RR^3} (1 + |\xi|^2)^{m-j} \,
             |\mathcal{F}(\partial_t^j v)(t,\xi)|^2 \dd\xi \nonumber\\
  &\leq C \sup_{t \in \RR} \int_{\RR^3} (1 + |\xi|^2)^{m-j} \sum_{k = 0}^{m-1} 
       \Big|\partial_t^j \Big(\psi((1 + |\xi|^2)^{1/2} t)\hat{g}_k(\xi) \frac{t^k}{k!}\Big)\Big|^2 \dd\xi \nonumber\\
  &\leq C \sum_{k = 0}^{m-1} \sup_{t \in \RR} \int_{\RR^3} (1+ |\xi|^2)^{m-j-k} |\hat{g}_k(\xi)|^2 \,
  \big|\partial_t^j[\psi((1 + |\xi|^2)^{\frac12} t) ((1 + |\xi|^2)^{\frac12} t)^k)]\big|^2 \dd\xi \nonumber\\
  &\leq C \sum_{k = 0}^{m-1}  \int_{\RR^3} (1+ |\xi|^2)^{m-k} |\hat{g}_k(\xi)|^2 \sup_{s \in \RR}|\partial_t^j(\psi( s) s^k)|^2 \dd\xi \nonumber\\
  &= C \sum_{k = 0}^{m-1}  \int_{\RR^3} (1+ |\xi|^2)^{m-k} |\hat{g}_k(\xi)|^2  \dd\xi = C \sum_{k = 0}^{m-1} \|g_k\|_{\cH^{m-k}(\RR^3)}^2.
 \end{align*}
 So the first estimate in~\eqref{EquationEstimatesForu} has been shown.
 
 For the second one we proceed similarly, now abbreviating $\psi_{l,k}(s) := \partial_t^l (\psi(s) s^k)$
 for $s \in \RR$ and each  $l \in \{0, \ldots, m\}$. We then derive
 \begin{align*} 
  &\|\partial_t^j v\|_{L^2(\RR, \cH^{m +1/2 -j}(\RR^3))}^2 = \int_{\RR} \int_{\RR^3} (1 + |\xi|^2)^{m+1/2-j}\,
    |\mathcal{F}(\partial_t^j v)(t,\xi)|^2 \dd\xi\dd t \\
  &\leq C \sum_{k = 0}^{m-1}  \int_{\RR^3} (1+ |\xi|^2)^{m-k} \,|\hat{g}_k(\xi)|^2 
  \int_{\RR}|\psi_{j,k}((1+|\xi|^2)^{1/2} t)|^2\, (1 + |\xi|^2)^{1/2} \dd t \dd\xi \nonumber\\
  &= C \sum_{k = 0}^{m-1} \int_{\RR} |\psi_{j,k}(s)|^2 \dd s \int_{\RR^3} (1+ |\xi|^2)^{m-k} \,|\hat{g}_k(\xi)|^2 \dd\xi 
  \leq C \sum_{k = 0}^{m-1} \|g_k\|_{\cH^{m-k}(\RR^3)}^2.\nonumber
 \end{align*}
 
 Denoting $(x_1, x_2)$ by $x'$ and $(\xi_1, \xi_2)$ by $\xi'$, we finally compute
 \begin{align*}
  &\partial_t^m v(t,x',0) = \partial_t^m \mathcal{F}^{-1}\Big(\psi((1+ |\cdot|^2)^{1/2} t) \sum_{k = 0}^{m-1} \hat{g}_k \frac{t^k}{k!} \Big)(x',0) \\
  &= \frac{1}{2\pi}\int_{\RR^2} e^{\ii x'\cdot \xi'} \frac{1}{\sqrt{2\pi}}\int_{\RR} \partial_t^m \Big(\psi((1+ |\xi|^2)^{1/2} t) 
   \sum_{k = 0}^{m-1} \hat{g}_k(\xi) \frac{t^k}{k!} \Big) \dd\xi_3 \dd\xi'.
 \end{align*}
 The spatial Fourier transform on $\RR^2$ of $\partial_t^m v(t,x',0)$ is thus given by
 \begin{align*}
  \mathcal{F}( \partial_t^m v(t,\cdot,0))(\xi') = \frac{1}{\sqrt{2\pi}}\int_{\RR} \partial_t^m \Big(\psi((1+ |\xi|^2)^{1/2} t)
    \sum_{k = 0}^{m-1} \hat{g}_k(\xi) \frac{t^k}{k!} \Big) \dd\xi_3.
 \end{align*}
 We next fix a time $t \in \RR \setminus \{0\}$. Since $\psi$ vanishes on the complement of $(-2,2)$, the 
 integrand above vanishes if $|\xi_3| > 2 / |t|$. This fact yields the estimate
 \begin{align*}
  &|\mathcal{F}( \partial_t^m v(t,\cdot,0))(\xi')|^2 \leq C \sum_{k=0}^{m-1} \Big[\int_{\RR} (1+ |\xi|^2)^{\frac{m-k}{2}}\, |\hat{g}_k(\xi)|\,
    \psi_{m,k}((1+ |\xi|^2)^{1/2} t) \dd \xi_3 \Big]^2 \\
  &\leq C \sum_{k = 0}^{m-1} \int_{-2/|t|}^{2/|t|} 1 \,\dd\xi_3 \int_{\RR} (1+ |\xi|^2)^{m-k}\, |\hat{g}_k(\xi)|^2\, |\psi_{m,k}((1+ |\xi|^2)^{1/2} t)|^2 \dd \xi_3 \\
  &= C \sum_{k = 0}^{m-1} \int_{\RR} (1+ |\xi|^2)^{m-k} \,|\hat{g}_k(\xi)|^2\, 
    \frac{|\psi_{m,k}((1+ |\xi|^2)^{1/2} t)|^2}{|(1+ |\xi|^2)^{1/2} t|} (1+ |\xi|^2)^{1/2}  \dd \xi_3.
 \end{align*}
 Since $k < m$, at least one derivative falls onto $\psi$ in  $\psi_{m,k}(t) = \partial_t^m (\psi(t) t^k)$. As $\psi$ is constant on $(-1/2,1/2)$, 
 the function $\psi_{m,k}$ vanishes on this interval, and hence the map $s \mapsto \frac{|\psi_{m,k}(s)|^2}{|s|}$ 
 belongs to $C_c^\infty(\RR)$.  We  infer
 \begin{align*}
  &\|\partial_t^m v(\cdot, 0)\|_{L^2(\RR \times \RR^2)}^2 = \int_{\RR} \int_{\RR^2} |\mathcal{F}( \partial_t^m v(t,\cdot,0))(\xi')|^2 \dd\xi' \dd t \nonumber\\
  &\le C \sum_{k = 0}^{m-1} \int_{\RR} \frac{|\psi_{m,k}(s)|^2}{|s|} \dd s \int_{\RR^3} (1+ |\xi|^2)^{m-k} |\hat{g}_k(\xi)|^2 \dd \xi 
      \leq C \sum_{k = 0}^{m-1} \|g_k\|_{\cH^{m-k}(\RR^3)}^2.
 \end{align*}
 Also employing the trace theorem, we obtain~\eqref{EquationEstimatesForu} 
 for the functions $v$ and $g_k$. The assertion now follows 
by approximation.
 \end{proof}

 We assume that the conditions of Theorem~\ref{thm:local} concerning the data are valid and that 
 the inhomogeneities
$f$ and $g$ belong to the spaces $\mathcal{H}^m((t_0,T) \times G)$ respectively 
$\mathcal{H}^m((t_0,T) \times \Sigma)$, for all $T > 0$. 
For the assumption~\eqref{ass:z} we take the quantity $d_m((t_0,t_0+1))$ unless something else is specified.
We then define the \emph{maximal existence time} by 
\begin{align}\label{def:max-time}
 T_+ &=T_+(m, t_0, u_0,f,g) \nonumber \\
 &= \sup \{\tau \geq t_0 \,|\, \exists \,\text{unique } \mathcal{G}^m_\Sigma  \text{-solution } u \text{ of } \eqref{eq:maxwell} 
      \text{ on } [t_0,  \tau] \}.
\end{align}
 The interval $(t_0, T_+) =: J_{\mathrm{max}}$ is called 
 the \emph{maximal interval of existence}. These notions are modified in a straightforward way if the inhomogeneities are given only on 
 a bounded interval $(t_0,T)$.  By standard methods we can extend the solution 
from Theorem~\ref{thm:local} to a  \emph{maximal solution} $u$  of~\eqref{eq:maxwell} 
 on $J_{\mathrm{max}}$ which belongs to $G^m_\Sigma((t_0,T)\times G)$ for all $T< T_+$ and cannot 
 be extended beyond this interval by a positive time span. 
 More precisely, we obtain the following basic blow-up criterion, cf.\ Lemma~4.1 of \cite{Sp2}.

 \begin{prop}\label{prop:max-sol}
Let $t_0 \in \RR$ and $m \in \NN$ with $m \geq 3$.
 Assume that either \eqref{ass:main} or  \eqref{ass:main1} is valid. 
 Choose data $(u_0, f,g)$ such that $u_0 \in \cH^m(G)^6$,  $f \in \cH^m((t_0,T)\times G)^6$,  $g \in \cH^m((t_0,T)\times \Sigma)^3$ 
 for all $T>t_0$, $g\cdot \nu =0$, and   the tuple $(t_0, \chi, \sigma,\zeta,u_0,f,g)$ 
 fulfills the compatibility conditions~\eqref{eq:cc-nl} of order $m$. 
 If assumption \eqref{ass:main} is valid, we require condition  \eqref{ass:z}.
Let $u$ be the maximal solution of~\eqref{eq:maxwell} on $J_{\mathrm{max}}$   introduced above. If 
   $T_+  < \infty$, then  one of the following blow-up properties
   \begin{enumerate}
    \item 
    $\liminf_{t \nearrow T_+} \dist(\{u(t,x) \,|\, x \in G\}, \partial \mathcal{U}) = 0$, 
    \item 
    $\lim_{t \nearrow T_+} \|u(t)\|_{\cH^m(G)} = \infty$,
	\item $\limsup_{t \nearrow T_+} \|B_1 u(t)\|_{L^\infty(\Sigma)} \geq \wt{\kappa}$ for any $\wt{\kappa}$ satisfying~\eqref{ass:z},
   \end{enumerate}
 occurs, where the last item is removed  if  \eqref{ass:main1} is satisfied.
In (c), we assume that (a) and (b) do not occur and define the constants in  \eqref{ass:z} 
for the quantities $\kappa:= \dist(\{u(t,x) \,|\, x \in G, t\in (t_0, T_+)\}, \partial \cU) >0$ and  $r^2= d_m(T_+-\delta,T_+ +\delta)$ for some $\delta  \in (0,T_+-t_0)$ 
and with $\|u_0\|_{\cH^m(G)}$ replaced by $\liminf_{t \nearrow T_+} \|u(t)\|_{\cH^m(G)}$.
  \end{prop}

\section{Local wellposedness}\label{sec:lwp}
In this section we improve the blow-up criterion of  Proposition~\ref{prop:max-sol} and show the 
continuous dependence on the data. For various quasilinear hyperbolic systems, one has established such criteria
in terms of  Lipschitz norms. (See Section~4 of \cite{Sp2} for references.) These results rely on Moser-type estimates
as stated in Lemma~4.2 of \cite{Sp2}. They will imply in partcular that the maximal existence time is independent of $m\ge 3$
in the case of linear boundary conditions.
The next proposition  is the key step in  this direction, where we recall \eqref{def:d}.

\begin{prop}\label{prop:lip}
	Let $m \in \NN$ with $m \geq 3$ and $t_0 \in \RR$.  Assume that either \eqref{ass:main} or  \eqref{ass:main1} is valid. 
	Choose data $u_0 \in \cH^m(G)^6$,  $f \in \cH^m((t_0,T)\times G))^6$,  and $g \in \cH^m((t_0,T)\times \Sigma)^3$ with 
$g\cdot \nu =0$ for $T>t_0$ such that the tuple $(t_0, \chi, \sigma,\zeta, u_0,f,g)$
fulfills the compatibility conditions~\eqref{eq:cc-nl} of order $m$.
   Let $u$ be the maximal solution of~\eqref{eq:maxwell}
    provided by Proposition~\ref{prop:max-sol} on $J_{\mathrm{max}}=(t_0,T_+)$. We introduce the quantity 
    \begin{align*}
     \omega(T) = \sup_{t \in (t_0,T)} \| u(t) \|_{W^{1,\infty}(G)}
    \end{align*}
    for every $T \in (t_0,T_+)$. We further take $r > 0$ with $d_m(J_{\mathrm{max}})\le r^2$.
   We set $T^* = T_+$ if $T_+ < \infty$ and pick any $T^* > t_0$ if $T_+ = \infty$. Take $\omega_0 > 0$ and a compact subset
    $\mathcal{U}_1$ of $\mathcal{U}$ such that  $\omega(T) \leq \omega_0$ and 
    $\ran u(t) \subseteq \mathcal{U}_1$ for all $t \in [t_0,T]$ and some $T \in (t_0,T^*)$.
    If \eqref{ass:main} is true, we also assume that~\eqref{ass:z} is valid for $\ol{\kappa}$ and that
    \beq \label{ass:z1}
   z (\ol{\kappa})^2\le 1/(2\tilde{C}_m),
    \eeq
    where $\ol{\kappa}= \|B_1 u\|_{L^\infty((t_0,T)\times \Sigma)}$
    and $\tilde{C}_m=\tilde{C}_m (\chi,\sigma,\zeta, m,r,\omega_0,\mathcal{U}_1,T^* - t_0)$ is defined as 
    $\max_{1 \leq k \leq m} \sum_{|\alpha| = k} C_{k,\alpha}$ with $C_{k,\alpha}$ appearing in~\eqref{est:ind-lip1}.

    Then there exists a constant $C = C(\chi,\sigma,\zeta, m,r,\omega_0,\mathcal{U}_1,T^* - t_0)$ such that
    \begin{align*}
     \|u\|_{G^m_\Sigma((t_0,T)\times G)}^2 &\leq C\, d_m((t_0,T)).
     \end{align*}
\end{prop}
\begin{proof}
1) We focus on the assumption \eqref{ass:main} of a nonlinear boundary condition, since the linear one
in \eqref{ass:main1} is easily treated as in Proposition~4.4 of \cite{Sp2}.
Without loss of generality we assume $t_0 = 0$ and that, if $G$ is unbounded, the nonlinearities 
$\chi$ and $\sigma$ satisfy \eqref{ass:f-var}, cf.\ Remark~\ref{rem:bdd1}.  We fix a number $T'\in (0,T^*)$
such  that $\omega(T') \leq \omega_0$ and $\ran u(t) \subseteq \mathcal{U}_1$ for all $t\in[0,T']$. Let 
$T\in(0,T']$ and set $J=(0,T)$. As in the proof of Proposition~4.4 of \cite{Sp2}, we have to work with the localized nonlinear problem
on $G=\RR^3_+$ and coefficents $A_1,A_2\in F^{\tilde{m}}_{\cf}(\E)$ and $A_3 = A_3^{\co}$. The full space case has already been treated in Proposition~7.20 in \cite{Sp0}.
We do not repeat the localization procedure itself, cf.\ Section~\ref{sec:aux}. As in (4.3) of \cite{Sp2} we obtain a constant
$c=c(\chi,\sigma,r,\omega_0,\mathcal{U}_1,T^*)$ such that
\[ \|u\|_{W^{1,\infty}(\Omega)}\le \|\partial_t u\|_{L^\infty(\Omega)} +\omega(T) \le c.\]

We put $L(u)=  L(\chi(u),A_1, A_2, A_3, \sigma(u))$. Let  $\alpha \in \NN_0^4$ with $|\alpha| \leq m$. In view of differentiated versions 
of \eqref{eq:maxwell}, we define 
\begin{align*}
      f_\alpha &= \partial^\alpha f - \sum_{0 < \beta \leq \alpha} \binom{\alpha}{\beta} \partial^\beta \chi(u)  \partial_t \partial^{\alpha - \beta} u 
      - \sum_{j = 1}^2 \sum_{0 < \beta \leq \alpha} \binom{\alpha}{\beta} \partial^\beta A_j \partial_j \partial^{\alpha - \beta} u  \nonumber \\
     &\quad - \sum_{0 <\beta \leq \alpha} \binom{\alpha}{\beta} \partial^\beta \sigma(u) \partial^{\alpha - \beta} u, \\
     g_\alpha &= \partial^\alpha g + \nu\times \sum_{0 < \beta \leq \alpha} \binom{\alpha}{\beta} \partial^\beta \zeta(B_1^\co u)  
        \partial^{\alpha - \beta} B_1^\co u 
\end{align*}
As $u$ solves~\eqref{eq:maxwell}, the function $v=\partial^\alpha  u$ satisfies the system
    \begin{align}
       \label{eq:ivp-u-alpha}
       \begin{aligned}
       L(u) v& = f_{\alpha}, \qquad &&x \in \RR^3_+, \quad &t \in (0,T), \\
	 v(0) &= \partial^{(0,\alpha_1,\alpha_2,\alpha_3)} S_{ m, \alpha_0,\chi,\sigma}(0,u_0,f), &&x \in \RR^3_+. 
	\end{aligned}
    \end{align}
If additionally $\alpha_3 = 0$, it is a solution of the boundary value problem 
    \begin{align}\label{eq:ibvp-u-alpha}
    \begin{aligned}
	L(u)v &= f_{\alpha}, \qquad  &&x \in \RR^3_+, \quad &t \in (0,T), \\
  B(u)v&= g_\alpha,  &&x \in \partial \RR^3_+, &t \in (0,T), \\
	v(0) &= \partial^{(0,\alpha_1,\alpha_2,\alpha_3)} S_{ m, \alpha_0,\chi,\sigma}(0,u_0,f), &&x\in \RR^3_+.
      \end{aligned}
      \end{align} 
Here we used that $\partial_t^j u(0) = S_{m, j,\chi,\sigma }(0,u_0,f)$
    for all $j \in \{0, \ldots, m\}$ by~\eqref{eq:S-nl}.
    
Let $|\alpha'|\le m-1$. Step~I) of the proof of Proposition~4.4 of \cite{Sp2} shows that
\begin{align*}
\|f_\alpha\|_{L^2(\Omega)}&\le \|f\|_{\cH^{|\alpha|}(\Omega)}  + c\, \|u\|_{\cH^{|\alpha|}(\Omega)},\\
\|f_{\alpha'}\|_{\cH^1(\Omega)}+ \|f_{\alpha'}(0)\|_{L^2(\E)}&\le cd_{|\alpha'|+1}(J) + c \|u\|_{\cH^{|\alpha'|+1}(\Omega)}
\end{align*}
with a constant $c=c(\chi,\sigma,m,r,\omega_0,\mathcal{U}_1)$. The above results rely on Lemma~4.2 of \cite{Sp2}
which is actually true for $u\in \cH^m(\Omega)$, cf.\ Lemma~7.19 of \cite{Sp0}. 

As in \eqref{est:rho}, we reduce most terms in $g_\alpha$ to those appearing in $f_\alpha$ by means of the trace theorem. 
The main ones then lead to a summand involving $z(\ol{\kappa})$. So we arrive at
\begin{align*}
\|g_\alpha\|_{L^2(\Gamma)}&\le \|g\|_{\cH^{|\alpha|}(\Gamma)}  + c\, (\|u\|_{\cH^{|\alpha|}(\Omega)} 
    + z(\ol{\kappa})\,\|B_1^\co \partial^\alpha u\|_{L^2(\Gamma)} \big),\\
\|g_{\alpha'}\|_{\cH^1(\Gamma)}&\le \|g\|_{\cH^{|\alpha'|+1}(\Gamma)}  + c\,( \|u\|_{\cH^{|\alpha|}(\Omega)} 
  + z(\ol{\kappa})\,\sum_{|\beta|= |\alpha'| + 1}\|B_1^\co \partial^\beta u\|_{L^2(\Gamma)}\big),
\end{align*}
where $c = c(\chi, \sigma, \zeta, m, r, \omega_0, \cU_1)$.

 2) We next show  that there are constants 
    $C_{k} \!=\! C_k(\chi,\sigma,\zeta, m,r,\omega_0,\mathcal{U}_1,\!T^*\!)$ with
    \begin{align}\label{est:ind-lip}
     \|\partial^\alpha u\|_{G^0_\Sigma(\Omega) }^2 \leq C_k d_k(J)
    \end{align}
for all $\alpha \in \NN_0^4$ with $|\alpha| = k$ and $k \in \{0,\ldots,m\}$. Proposition~\ref{prop:L2}
yields the case  $k=0$ as in the proof 
of Proposition~4.4 of \cite{Sp2}. So let \eqref{est:ind-lip} be true for all $j\in \{0,\dots, k-1 \}$ and some $k \in \{1,\dots,m\}$. 
Take $\alpha \in \NN_0^4$ with $|\alpha| =k$. 
We first show that there is a constant $C_{k,\alpha}\ = C_{k,\alpha}(\chi,\sigma,\zeta, m,r,\omega_0,\cU_1,T^*)$ with
  \begin{align}\label{est:ind-lip1}
     &\|\partial^\alpha u\|_{G^0_\Sigma(\Omega) }^2 \\
     &\leq\! C_{k,\alpha} \Big[d_k(J) 
       + \!\!\sum_{|\beta| = k} \!\int_0^T \!\!\big[\|\partial^\beta u(s)\|_{L^2(\E)}^2\! + \!
       z(\ol{\kappa})^2\|B_1^\co\partial^\beta u(s)\|_{L^2(\partial \E)}^2\big]\D s \Big] \nonumber
    \end{align}
 for each $\alpha \in \NN_0^4$ with $|\alpha| =k$.   This claim is shown via induction over $\alpha_3$.
 
 So let $\alpha_3=0$. Since $\partial^\alpha u$ solves \eqref{eq:ibvp-u-alpha},
Proposition~\ref{prop:L2}, the bounds on $f_\alpha$ and $g_\alpha$, and estimate~\eqref{est:S-nl} yield a constant 
$c=c(\chi,\sigma,\zeta, k,r,\omega_0,\mathcal{U}_1,T^*)$ such that
\begin{align*}
\|\partial^\alpha u\|_{G^0_\Sigma(\Omega) }^2 \leq c\,\big(d_k(J) +\|u\|_{\cH^{|\alpha|}(\Omega)}^2
    + z(\ol{\kappa})^2\,\sum_{|\beta| = k} \|B_1^\co \partial^\beta u\|_{L^2(\Gamma)}^2\big).
\end{align*}  
The derivatives of $u$ of order up to $k-1$ can be bounded by  the induction hypothesis \eqref{est:ind-lip}. 
So we have shown \eqref{est:ind-lip1} for $k$ and $\alpha_3=0$. The other induction steps then only involve the
initial value problem \eqref{eq:ivp-u-alpha} without a boundary condition so that we can argue exactly as in
Proposition~4.4 of \cite{Sp2} to derive \eqref{est:ind-lip1} for all $\alpha_3\le k$.

We now sum in \eqref{est:ind-lip1} over all $\alpha\in \NN^4_0$ with $|\alpha| = k$. Assumption~\eqref{ass:z1} 
then allows to absorb the boundary terms in the left-hand side. Afterwards, we use Gronwall's inequality 
to control $\sum_{|\alpha| = k} \|\partial^\alpha u\|_{G^0(\Omega)}$ as in~(4.14) of~\cite{Sp2}.
Combining these two estimates, we finally obtain a constant $C_{k} \!=\! C_k(\chi,\sigma,\zeta, m,r,\omega_0,\mathcal{U}_1,\!T^*\!)$
such that
\begin{align}\label{est:tildeC}
\sum_{|\alpha|=k} \|\partial^\alpha u\|_{G^0_\Sigma(\Omega) }^2 \leq  C_{k} d_k(J).
\end{align} 
We have thus  shown \eqref{est:ind-lip}. The assertion now follows by induction.
\end{proof}

The blow-up criterion for \eqref{eq:maxwell} will be established in the local wellposedness Theorem~\ref{thm:lwp} below.
Before, we provide auxiliary results needed to show the continuous dependence on data, starting with an approximation
lemma in lowest order. Its proof is omitted since it is a minor modification of that of Lemma~5.1 in \cite{Sp2}.

\begin{lem}\label{lem:cont-dep1}
   Let $J \subset \RR$ be an open interval and $t_0 \in \ol{J}$.
   Take coefficients $A_{0,n}, A_0 \in F^3_{\eta}(\Omega)$, $A_1,A_2\in F^3_ {\cf}(\RR^3_+)$, $A_3 = A_3^{\co}$,
   $D_n, D \in F^3(\Omega)$, and $b_n,b\in  F^3_{\cH,\eta}(\Gamma)$ for all $n \in \NN$ 
   such that $(A_{0,n})_n$, $(D_n)_n$ respectively $(b_n)_n$ are bounded in $W^{1,\infty}(\Omega)$ respectively 
    $W^{1,\infty}(\Gamma)$ and converge to $A_0$,  $D$ respectively $b$ uniformly. 
   Let $B_j = B^{\operatorname{co}}_j$ for $j\in\{1,2\}$  and $G=\E$.
   Choose $u_0 \in L^2(\E)$,  $f  \in L^2(\Omega)$, and $g \!\in\! L^2(\Gamma)$ with 
$g\cdot \nu =0$. Let $u_n,u \in  G_\Sigma^0(\Omega)$ solve  the linear Maxwell system~\eqref{eq:maxwell-lin}
   with the above coefficients and data. Then  $(u_n)_n$ tends to $u$ in $G_\Sigma^0(\Omega)$.
\end{lem}

The next result is the core of the proof of continuous dependence. It improves the norm in which solutions converge by
one regularity level, provided one has appropriate apriori information.

\begin{lem}\label{lem:cont-dep2}
Let $J \subseteq \RR$ be an open  bounded interval, $t_0 \in \overline{J}$, and $m \in \NN$ with $m \geq 3$.
	Assume that either \eqref{ass:main} or  \eqref{ass:main1} is valid.
	Choose data $u_0, u_{0,n}\in \cH^m(G)$,  $f,f_n \in \cH^m(J\times G)$,  and $g,g_n \in \cH^m(J\times \Sigma)$ with 
$g\cdot \nu =0$ and $g_n\cdot \nu =0$ for all $n\in\NN$ such that
 \begin{align*}
    \|u_{0,n} - u_0\|_{\cH^m(G)} \longrightarrow 0, \quad \|f_n - f\|_{\cH^m(J\times G)} \longrightarrow 0, 
    \quad   \|g_n - g\|_{\cH^m(J\times \partial G)}\longrightarrow 0,
   \end{align*}
   as $n \rightarrow \infty$.  We further assume that~\eqref{eq:maxwell} 
   with data $(t_0, u_{0,n}, f_n, g_n)$ and $(t_0, u_0, f, g)$ have $G^m_\Sigma(J\times G)$-solutions
   $u_n$ and $u$ for all $n \in \NN$, that there is a compact subset $\mathcal{U}_1$ of $\mathcal{U}$ 
   with $\ran u(t) \subseteq \mathcal{U}_1$ for all $t \in J$, that $(u_n)_n$ is bounded in 
   $G^m_\Sigma(J\times G)$, and that $(u_n)_n$ converges to $u$ in $G^{m-1}_\Sigma(J \times G)$. If \eqref{ass:main} is valid, we require that 
   \beq\label{ass:z2}
    z(\ol{\kappa})^2\le 1/(2\hat{C}_m)
   \eeq 
    for a fixed number $\ol{\kappa}>\|B_1 u\|_{L^\infty(\Gamma)} $, where    $\hat{C}_m=\hat{C}_m(\chi,\sigma,\zeta,r,\mathcal{U}_1,T')$ appears in \eqref{est:un-u},
    $|J|\le T'$, and $r$ only depends on $d_m(J)$, $\|u\|_{G^m_\Sigma(J\times G)}$ and $\Omega$.
   Then the solutions $u_n$ converge to $u$ in $G^m_\Sigma(J\times G)$.
\end{lem}
\begin{proof}
1) We focus on the assumption \eqref{ass:main} of a nonlinear boundary condition, since the linear one
in \eqref{ass:main1} is easily treated as in Lemma~5.2 of \cite{Sp2}.
Without loss of generality we take $t_0 = 0$, $J=(0,T)$ and that, if $G$ is unbounded, the nonlinearities 
$\chi$ and $\sigma$ satisfy \eqref{ass:f-var}, cf.\ Remark~\ref{rem:bdd1}. Moreover, $T$ is less or equal  
than a fixed time $T'<\infty$. As in Proposition~\ref{prop:lip} we have to work with the localized 
nonlinear problem on $G=\RR^3_+$ and coefficients $A_1,A_2\in F^{\tilde{m}}_{\cf}(\E)$ and $A_3 = A_3^{\co}$. 
We do not repeat the localization procedure itself, cf.\ Section~\ref{sec:aux}.

Throughout, we let $\alpha\in \NN^4_0$ with $|\alpha| \le m$ and $n\in\NN\cup\{\infty\}$, where we put $u_\infty=u$ etc.
Due to our assumptions, we can fix a number $r>0$ that only depends on $d_m(J)$, $\|u\|_{G^m_\Sigma(J\times G)}$, and $\Omega$
and  that dominates the quantities $d_m^n(J)$ for the data $(u_{0,n}, f_n, g_n)$
and the norms of $u_n$ in $G^{m}_\Sigma(\Omega)$ and $L^\infty(\Omega)$ as well as  $A_1$ and $A_2$ in $F^m(\Omega)$.
Here and in the next statement we may omit $n\le n_0$ for some $n_0\in\NN$.
Let $\kappa=\dist (\cU_1,\partial \cU)>0$ and $\cU_1'= \cU_1 +\ol{B}(0,\kappa/2).$  Then we obtain
 $\ol{\kappa}>\|B_1 u_n\|_{L^\infty(\Gamma)} $ and $\ran u_n(t)\sub \cU_1'$ for all $t\in\ol{J}$ and $n\in\NN$.
There is another radius $R=R(\chi, \sigma,\zeta,m, r,\mathcal{U}_1)$ dominating the functions $\chi(u_n)$ and $\sigma(u_n)$ in $F^m(\Omega)$
and $\zeta(B_1^\co u_n)$ in $F^m_{\cH}(\Gamma)$.

Let $L_n$ and $B_n$ be the differential and boundary operator from \eqref{eq:maxwell-lin} with coefficients $A_0=\chi(u_n)$ and $D=\sigma(u_n)$ 
respectively $b=\zeta(B_1^\co u_n)$. We use the modified inhomogeneities 
\begin{align*}
      f_{\alpha,n} &= \partial^\alpha f - \sum_{0 < \beta \leq \alpha} \binom{\alpha}{\beta} \partial^\beta \chi(u_n)  \partial_t \partial^{\alpha - \beta} u_n 
      - \sum_{j = 1}^2 \sum_{0 < \beta \leq \alpha} \binom{\alpha}{\beta} \partial^\beta A_j \partial_j \partial^{\alpha - \beta} u_n  \nonumber \\
     &\quad - \sum_{0 <\beta \leq \alpha} \binom{\alpha}{\beta} \partial^\beta \sigma(u_n) \partial^{\alpha - \beta} u_n, \\
     g_{\alpha,n} &= \partial^\alpha g + \nu\times \sum_{0 < \beta \leq \alpha} \binom{\alpha}{\beta} \partial^\beta \zeta(B_1^\co u_n)  
     \partial^{\alpha - \beta} B_1^\co u_n,
\end{align*}
where we assume that $\alpha_3=0$ when considering $g_{\alpha,n}$ here and below.
Exploiting that $A_3$ and $B_j^\co$  are constant, we see that the function $v=\partial^\alpha u_n$  
solves the linear initial boundary value problem
   \begin{align}
   \label{eq:ibvp-un-alpha}
   \begin{aligned}
L_n v &= f_{\alpha,n}, \qquad &&x \in \RR^3_+, \quad &t \in J, \\
B_nv &=  g_{\alpha,n}, &&x \in \partial\RR^3_+, &t \in J,\\
   			v(0) &= \partial^{(0,\alpha_1,\alpha_2,\alpha_3)} S_{m,\alpha_0,\chi,\sigma}(0,u_{0,n},f_{n}), &&x \in \RR^3_+,
   \end{aligned}
   \end{align}
 if $\alpha_3 = 0$. We further introduce the auxiliary map
   \begin{align*}
    h_n(t) &= \sum_{i=1}^3 \sum_{0 \leq j \leq m} \, \sum_{\substack{0 \leq \gamma \leq \alpha, \gamma_0 = 0 \\ |\gamma| = m-j}} \,
\sum_{l_1,\ldots,l_{j} = 1}^6 \|(\partial_{y_{l_{j}}} \ldots \partial_{y_{l_1}} \partial_x^{(\gamma_1,\gamma_2,\gamma_3)}\theta_i)(u_n(t)) \notag\\
  &\hspace{14em} - (\partial_{y_{l_{j}}} \ldots \partial_{y_{l_1}} \partial_x^{(\gamma_1,\gamma_2,\gamma_3)}\theta_i)(u(t))\|_{L^\infty(\RR^3_+)}\\
  &\quad+ \sum_{0 \leq j \leq m} \, \sum_{\substack{0 \leq \gamma \leq \alpha', \gamma_0 = 0 \\ |\gamma| = m-j}} \,
\sum_{l_1, \ldots, l_{j} = 1}^3 \|(\partial_{y_{l_{j}}} \ldots \partial_{y_{l_1}} \partial_x^{(\gamma_1,\gamma_2)}\zeta)(B^\co_1 u_n(t)) \nonumber \\
  &\hspace{12em} - (\partial_{y_{l_{j}}} \ldots \partial_{y_{l_1}} \partial_x^{(\gamma_1,\gamma_2)}\zeta)(B^\co_1u(t))\|_{L^\infty(\partial\RR^3_+)},
   \end{align*}
where $t \in \overline{J}$, $n \in \NN$, $\theta_1 = \chi$, $\theta_2 = \sigma$, $\theta_3=\chi^{-1}$, and $\alpha'\in\NN^3_0$ with $|\alpha'|=m$.
Observe that the functions $h_n$ tend to 0 uniformly as $n\to\infty$.

Using the calculus results Lemma~2.1 of \cite{Sp1} and Corollary~2.2 of \cite{Sp2}, one can show that all maps $f_{\alpha,n}$ 
and $g_{\alpha,n}$ are bounded in  $L^2(\Omega)$ respectively $L^2(\Gamma)$ by a constant $c=c(\chi,\sigma,\zeta,m,r,\mathcal{U}_1,T')$.
If $|\alpha|\le m-1$, then we have analogous bounds in $\cH^1(\Omega)$, $G^0(\Omega)$ respectively,  $\cH^1(\Gamma)$. We further derive the inequalities
\begin{align}\label{est:fg-n-infty}
\|f_{\alpha,n} \! -  f_{\alpha,\infty}\|^2_{\cH^k(\Omega)} 
   &\leq c\Big[\|f_n \!- f\|_{\cH^m(\Omega)}^2 + \|u_n\! - u\|_{G^{m-1}(\Omega)}^2 + \delta_{|\alpha| (m-k)} \|h_n\|_\infty^2 \nonumber\\
&\hspace{5em} +\int_0^T\sum_{\tilde{\alpha}\in\NN_0^4,|\tilde{\alpha}| = m} \|\partial^{\tilde{\alpha}} (u_n(s)-u(s))\|_{L^2(\E)}^2\dd s\Big],\notag\\
 \|f_{\alpha,n} \! -  f_{\alpha,\infty}\|^2_{G^0(\Omega)} 
   &\leq c\,\big(\|f_n - f\|_{G^{m-1}(\Omega)}^2 + \|u_n - u\|_{G^{m-1}(\Omega)}^2   \big),\\
   \|g_{\alpha,n} \! -  g_{\alpha,\infty}\|^2_{\cH^k(\Omega)} 
   &\leq c\Big[\|g_n \!- g\|_{\cH^m(\Gamma)}^2 + \|u_n\! - u\|_{G^{m-1}(\Omega)}^2 + \delta_{|\alpha| (m-k)} \|h_n\|_\infty^2 \nonumber\\
   &\quad +\int_0^T\sum_{\tilde{\alpha}\in\NN_0^4,|\tilde{\alpha}| = m} \|\partial^{\tilde{\alpha}} (u_n(s)-u(s))\|_{L^2(\E)}^2\dd s \nonumber \\
   &\quad \ +z(\ol{\kappa})^2\int_0^T\sum_{\tilde{\alpha}\in\NN_0^3,|\tilde{\alpha}|=m}
          \|B_1^\co \partial^{\tilde{\alpha}}(u_n(s) - u(s))\|_{L^2(\partial\E)}^2\dd s\Big]\notag
 \end{align}
 for $k\in\{0,1\}$ and $|\alpha| \leq m-1$, using also \eqref{est:rho} and \eqref{est:rho1}.
 Here the first and the last estimate are also true 
 for $|\alpha| = m$ in the case $k = 0$.
 
 \smallskip
 
2)  We first treat tangential derivatives with $\alpha_3=0$. We set $w_{0,n}=\partial^{(0,\alpha_1,\alpha_2,0)} S_{m,\alpha_0,\chi,\sigma}(0,u_{0,n},f_{n})$.
To decompose $\partial^\alpha u_n= w_n+ z_n$, we use the solution $w_n\in G^0_\Sigma (\Omega)$ of the linear system 
 \begin{align}
 \label{eq:ibvp-wn-alpha}
 \begin{aligned}
L_n v &= f_{\alpha,\infty}, \qquad &&x \in \RR^3_+, \quad &t \in J,\\
B_nv &=  g_{\alpha,\infty}, &&x \in \partial\RR^3_+, &t \in J,\\
v(0) &= w_{0,\infty}, &&x \in \RR^3_+,
  \end{aligned}
  \end{align}
with fixed data, and $z_n\in G^0_\Sigma (\Omega)$ of the linear problem
\begin{align}
  \label{eq:ibvp-zn-alpha}
  \begin{aligned}
L_n v &= f_{\alpha,n}-f_{\alpha,\infty}, \qquad &&x \in \RR^3_+, \quad &t \in J,\\
B_nv &=  g_{\alpha,n}- g_{\alpha,\infty}, &&x \in \partial\RR^3_+, \quad &t \in J,\\
v(0) &=w_{0,n}-w_{0,\infty}, &&x \in \RR^3_+,
\end{aligned}
   \end{align} 		
with data tending to 0 as we show below. These solutions exist due to   Proposition~\ref{prop:L2}, and we have $w_\infty= \partial^\alpha u$ by uniqueness
and \eqref{eq:ibvp-un-alpha}. By our assumptions, the coefficients $\chi(u_n)$, $\sigma(u_n)$  and $\zeta(B_1^\co u_n)$ converge uniformly
to $\chi(u)$, $\sigma(u)$  respectively $\zeta(B_1^\co u)$. In view of the estimates in step~1), Lemma~\ref{lem:cont-dep1} shows 
\beq\label{est:wn-winfty}
\|w_n- \partial^\alpha u\|_{G^0_\Sigma(\Omega)} = \|w_n- w_\infty\|_{G^0_\Sigma(\Omega)}\ \longrightarrow \ 0, \qquad n\to\infty.
\eeq
Let  $\gamma= \gamma(\chi, \sigma,\zeta,m, r,\mathcal{U}_1, T') \geq 1$ be the parameter $\gamma_0(\eta, R)$ from Proposition~\ref{prop:L2}.
We now apply this result to \eqref{eq:ibvp-zn-alpha} and argue as in (5.22) of \cite{Sp2}. By means of \eqref{est:fg-n-infty}, we thus obtain
\begin{align}\label{est:zn}
&\|z_n\|_{G^0_\Sigma(\Omega)}^2 \le c\Big[\tilde{d}_m^n(J)+\|u_n - u\|_{G^{m-1}(\Omega)}^2+ \|h_n\|_\infty^2 \\
  & +\! \sum_{\tilde{\alpha},\alpha'}\int_0^T \!\!\!\big(\|\partial^{\tilde{\alpha}} (u_n(s) - u(s))\|_{L^2(\E)}^2 
   +z(\ol{\kappa})^2\|\partial^{\alpha'}B_1^\co(u_n(s) - u(s))\|_{L^2(\partial\E)}^2\big)\dd s\Big] \notag
\end{align}
for a constant $c=c(\chi, \sigma,\zeta,m, r,\mathcal{U}_1, T')$, where we sum over all multi-indices $\tilde\alpha\in \NN^4_0$ and
$\alpha'\in \NN^3_0$ with $|\tilde{\alpha}|,|\alpha'| = m$ and the quantity
$\tilde{d}_m^n(J)$ is defined as in \eqref{def:d} for $u_{0,n}-u_0$, $f_n-f$ and $g_n-g$. We write $I_n(T)$ for the above sum of integrals.
Since $\partial^\alpha(u_n-u)= w_n-\partial^\alpha u + z_n$ by uniqueness again, estimates 
\eqref{est:wn-winfty} and \eqref{est:zn} imply  the bound
\begin{align*}
\|\partial^\alpha(u_n-u)\|_{G^0_\Sigma(\Omega)} &\le a_{\alpha,n} + c I_n(T)
\end{align*}
for numbers $a_{\alpha,n}$ tending  to 0 as $n\to\infty$. As in step~III) of Lemma~5.2 of \cite{Sp2} an induction 
extends the above estimate to the case of all $\alpha_3\le m$. These arguments do not involve  the boundary conditions, so that
there is no need to repeat them here. We obtain as in  \cite{Sp2} the inequality
\begin{align}\label{est:un-u}
&\sum_{|\tilde{\alpha}|,|\alpha'| = m}\big(\|\partial^{\tilde{\alpha}} (u_n(t)-u(t))\|_{L^2(\E)}^2 + \|\partial^{\alpha'} \tr_\tau (u_n-u)\|_{L^2(\Gamma)}^2 \big)\\
 &\le a_n + \hat{C}_m \int_0^t \|D^{\tilde{\alpha}} (u_n(s) - u(s))\|_{L^2(\E)}^2 \dd s
   +z(\ol{\kappa})^2 \hat{C}_m  \|D^{\alpha'}B_1^\co(u_n- u)\|_{L^2(\Gamma)}^2\notag
\end{align}
for $t\in \ol{J}$, a null sequence $(a_n)$ and a constant $\hat{C}_m=\hat{C}_m(\chi,\sigma,\zeta,r,\mathcal{U}_1,T')$. The notation $D^{\alpha}$ also includes
the summation over $|\alpha| = m$, where $\tilde\alpha\in \NN^4_0$ and $\alpha'\in \NN^3_0$.
Finally, we first use the smallness assumption on $z(\ol{\kappa})$ and then Gronwall's inequality to conclude the assertion.
\end{proof}

We finally establish the full local wellposedness theorem. 
 For times $t_0 < T$ we introduce the data space 
  \begin{align*}
     &M_{\chi,\sigma,\zeta,m}(t_0,T) = \{( \tilde{u}_0,\tilde{f},\tilde{g}) \in \cH^m(G)\times\cH^m((t_0,T) \times G) \times \cH^m((t_0,T) \times \Sigma)  \,|\, \\
	&\hspace{3.5cm} (\chi, \sigma, \zeta, t_0, \tilde{f}, \tilde{g}, \tilde{u}_0) \text{ is compatible of order } m, \; \tilde{g}\cdot \nu=0\}
    \end{align*}
  and endow it  with its natural norm.
  
\begin{thm}\label{thm:lwp}
Let $m \!\in\! \NN$ with $m \!\geq\! 3$ and $t_0 \!\in \!\RR$. 
Assume that either \eqref{ass:main} or  \eqref{ass:main1} is valid.
	Choose data $u_0 \in \cH^m(G)$,  $f\in \cH^m((t_0,T)\times G))$,  and $g\in \cH^m(((t_0,T)\times \Sigma)$ with 
$g\cdot \nu =0$ for all $T>t_0$ such that $\overline{\ran(u_0)} \!\subseteq \!\mathcal{U}$ and 
   the tuple $(\chi,\sigma,\zeta, t_0,u_0,f,g)$ fulfills the compatibility 
   conditions~\eqref{eq:cc-nl} of order $m$. If assumption \eqref{ass:main} is true, we 
   pick $\wt{\kappa} > 0$ satisfiying~\eqref{ass:z} and we require $\|B_1 u_0\|_{L^\infty(\Sigma)} < \wt{\kappa}/4$.
   
   Then the maximal existence times $T_+(k, t_0, u_0,f,g)$ from~\eqref{def:max-time}
   do not depend on $k\in\{3,\dots,m\}$ if  \eqref{ass:main1} is true. Moreover, the following assertions hold.
   \smallskip
  
 {\rm (1)}    There exists a unique maximal solution $u$ of~\eqref{eq:maxwell}
    which belongs to the function space $G^m_\Sigma((t_0,T)\times G)$ for all $T<T_+$.
    
   \smallskip 
   
 {\rm   (2)} If $T_+ < \infty$, then
	\begin{enumerate}
	 \item the solution $u$ leaves every compact subset of $\mathcal{U}$, or
	 \item  $\limsup_{t \nearrow T_+} \|\nabla u(t)\|_{L^{\infty}(G)} = \infty$, or
	 \item  condition (c) from  Proposition~\ref{prop:max-sol} occurs  or \eqref{ass:z1} fails as $T\to T_+$.
	\end{enumerate}
	If \eqref{ass:main1} is valid, the last condition can be dropped.
	\smallskip
	
{\rm (3)} Let $T \in (t_0, T_+)$. Fix $T'\in(T,T_+)$. If assumption \eqref{ass:main} is true, 
let \eqref{ass:z2} hold on $(t_0, T')$ and assume that $\|B_1 u\|_{L^\infty((t_0,T') \times \Sigma)} < \wt{\kappa}/4$ 
and that~\eqref{est:hat-kappa} is valid for $\wt{\kappa}$. (The constants in these conditions depend on 
$r$ from \eqref{est:lwp-data}, $(T'-t_0)$, and $\kappa=\tfrac12 \dist(\mathcal{U}_1, \partial \mathcal{U})$ for
a compact subset $\mathcal{U}_1 \sub \mathcal{U}$  with $\ran (u(t)) \subseteq \mathcal{U}_1$ 
for all $t \in [t_0, T']$.) 
Then there is a number 
    $\delta > 0$ such that for all data  $(\tilde{u}_0, \tilde{f}, \tilde{g})\in M_{\chi,\sigma,\zeta,m}(t_0,T)$  fulfilling
    \begin{equation*}
   \|\tilde{u}_0 - u_0\|_{\cH^m(G)}  < \delta, \quad  \|\tilde{f} - f\|_{\cH^m((t_0,T')\times G)} < \delta, \quad
    \|\tilde{g} - g\|_{\cH^m((t_0, T') \times \Sigma)} < \delta
    \end{equation*}
    the maximal existence time satisfies $T_+(m, t_0, \tilde{f},\tilde{g}, \tilde{u}_0) > T$. 
    Let $u(\cdot; \tilde{u}_0, \tilde{f}, \tilde{g})$ be the corresponding maximal 
    solution of~\eqref{eq:maxwell}. The flow map
    \begin{align*}
     \Psi\colon B_{M_{\chi,\sigma,\zeta,m}(t_0, T')}((u_0,f,g), \delta)  &\rightarrow G^m_\Sigma((t_0,T) \times G) , \quad
     ( \tilde{u}_0,\tilde{f},\tilde{g}) \mapsto u(\cdot;  \tilde{u}_0,\tilde{f}, \tilde{g}), 
    \end{align*}
    is continuous. Moreover, there is a constant 
    $C = C(\chi, \sigma,\zeta, m,r,T' - t_0, \kappa)$ with
    \begin{align}\label{est:lwp-lip}
     \|\Psi(\tilde{u}_{0,1},&\tilde{f}_1, \tilde{g}_1) - \Psi( \tilde{u}_{0,2},\tilde{f}_2, \tilde{g}_2)\|_{G^{m-1}_\Sigma((t_0,T) \times G)} \\
&\leq C \|\tilde{u}_{0,1} - \tilde{u}_{0,2}\|_{\cH^{m}(G)} + C \sum_{j = 0}^{m-1}\|\partial_t^j \tilde{f}_1(t_0) - \partial_t^j \tilde{f}_2(t_0)\|_ {\cH^{m-j-1}(G)} \notag\\
       &\qquad   +C \,\| \tilde{f}_1 - \tilde{f}_2 \|_{H^{m-1}((t_0,T) \times G)}
       + C \,\|\tilde{g}_1 - \tilde{g}_2\|_{\cH^{m-1}((t_0,T) \times \Sigma)} \nonumber
    \end{align}
    for all $(\tilde{u}_{0,j},\tilde{f}_j, \tilde{g}_j) \in B_{M_{\chi,\sigma,\zeta,m}(t_0, T')}((u_0,f,g), \delta)$.
\end{thm}
\begin{proof}
Except for (3) the assertions follow in a standard way from Theorem~\ref{thm:local} and Propositions~\ref{prop:max-sol} and \ref{prop:lip},
cf.\ the proof of Theorem~~5.3 in \cite{Sp2}. 

1) To show (3), let $t_0< T< T'<T_+$ be as in the statement and $J'=(t_0,T')$. Again, we focus on assumptions
\eqref{ass:main} concerning nonlinear boundary conditions since the linear ones from \eqref{ass:main1} can be treated 
very similar to the proof of  Theorem~5.3 in \cite{Sp2}. Let $C_S'$ be the norm of the embedding of $\cH^m(J' \times G)$  
in $G^{m-1}(J'\times G)$ and $C_S$ of $\cH^2(G)$ into $C_b(\ol{G})$. We take radii $r>0$ such that
\begin{align}\label{est:lwp-data}
&\|u_0\|_{\cH^m(G)} +  \|f\|_{G^{m-1}(J' \times G)} + \|f\|_{\cH^m(J' \times G)} 
  +\|g\|_{\cH^m(J' \times\Sigma)}  < r/(mC_S'), \nonumber\\
    & \|u\|_{G^m_\Sigma(J' \times G)} < r, 
\end{align} 
Let  $\mathcal{U}_1\sub \mathcal{U}$ and $\wt{\kappa}>4\|B_1 u\|_{L^\infty((t_0,T') \times \Sigma)}$ be given as in the statement.
As in the proof of  Theorem~~5.3 in \cite{Sp2} one finds a radius
$\tilde{r} = \tilde{r}(\chi,\sigma,m,r, \mathcal{U}_1)$  larger than the norms of $\theta(u)$ 
in $F^m((t_0, T') \times G)$, of $\zeta(B_1u)$ in $F^m_{\cH}((t_0, T') \times \Sigma)$, of $\theta(u(t_0))$ in $F^{m-1,0}(G)$
and that of $\partial_t^j\theta(u)(t_0)$ in $\cH^{m-1-j}(G)$ for $j\in\{1,\dots, m-1\}$ and $\theta\in\{\chi,\sigma\}$.
We fix a number $\kappa< \tfrac12\dist(\mathcal{U}_1, \partial \mathcal{U})$  and set
\[ V_{\kappa} = \{y \in \mathcal{U}\,|\,  \dist(y, \partial \mathcal{U}) \geq \kappa \}  \cap \overline{B}(0,2C_S r) 
\quad \text{ and } \quad \tilde{V}_\kappa = V_\kappa +\ol{B}(0,\kappa/2)\sub \cU.  \]
We take 
$R=R(\chi,\sigma,\zeta,m,4r,\kappa,T') > 4r$ of \eqref{def:local-R} in the proof of  Theorem~\ref{thm:local}.

Choose a number $\hat{T}\in(t_0,T')$ and data  $( \hat{u}_0,\hat{f},\hat{g}) \in M_{\chi,\sigma,\zeta,m}(J')$
such that $\hat{u}_0$ maps into $V_\kappa$ and the data satisfy the bounds \eqref{est:lwp-data} with $2r$ instead of $r$.
Let $\hat{J}=(t_0, \hat{T})$.
We assume that a solution $\hat{u}\in G^m_\Sigma(\hat{J}\times G)$ of \eqref{eq:maxwell} exists for these data
with norm less or equal $R$ in this space and taking values in  $\tilde{V}_\kappa$.
 Let $\hat{\kappa}\ge \ol{\kappa}$ bound the supnorm of $B_1 \hat{u}$
on $(t_0, \hat{T})\times \Sigma$. There then exists a constant  
$\hat{C}= \hat{C}(\chi,\sigma,\zeta, m,2r,R,\tilde{V}_\kappa, T')$ and a time step 
$\hat{\tau}= \hat{\tau}(\chi,\sigma,\zeta, m,2r,R,\tilde{V}_\kappa, \hat{\kappa}, T')$ such that the difference
of $u$ and $\hat{u}$ is controlled by
\begin{align}\label{est:lwp-dep0}
\|u-\hat{u}\|_{G^{m-1}_\Sigma(\hat{J}\times G)}^2 &\leq \hat{C}\Big( \|u_0 - \hat{u}_{0}\|_{\cH^{m-1}(G)} ^2
  + \sum_{j = 0}^{m-1}\|\partial_t^j f(t_0) - \partial_t^j \hat{f}(t_0)\|_ {\cH^{m-j-1}(G)}^2 \notag \\
       &\qquad \quad  +\| f - \hat{f} \|_{H^{m-1}(\hat{J} \times G)}^2
       +\|\hat{g}_1 - \hat{g}_2\|_{\cH^{m-1}(\hat{J} \times \Sigma)}^2 \Big),
\end{align}
where   $\hat{J}=(t_0,t_0+\hat{\tau})$ and we assume that 
\beq \label{est:hat-kappa}
\hat{\tau}, z(\hat{\kappa})^2\le (4\hat{C}\ol{C}(R))^{-1}
\eeq
with $\ol{C}(R)$ from \eqref{est:rho1}.
Using Proposition~\ref{prop:L2} and \eqref{est:rho1}, this fact can be shown as (5.31) in \cite{Sp2} with modifications
 analogous  to \eqref{est:local-contr}.  
 
 \smallskip

2) We take as a time step $\tau$ the minimum of  $\hat{\tau}$ in step~1), of $\wt{\kappa}/(2C_S R)$, and of
$ \tau(\chi, \sigma, \zeta,m, T', 2r,\kappa, \wt{\kappa})$ from \eqref{def:local-tau}.  There is an  index $N \in \NN$ 
   with $t_0 + (N-1) \tau < T \leq t_0 + N \tau.$
We set $t_k = t_0 + k \tau$ for $k \in \{1, \ldots, N-1\}$. If
    $t_0 + N \tau < T'$, we put $t_N = t_0 + N \tau$; else we take any $t_N$ from $(T, T')$.
Next, we choose a radius $\delta_0>0$ which is less than $r/(4mC_S')$, $\wt{\kappa}/(4C_S)$, and $\kappa/C_S$.
Let $( \tilde{u}_0,\tilde{f},\tilde{g}) \in B_{M_{\chi,\sigma,\zeta,m}(t_0, T')}((u_0,f,g), \delta_0)=:B_M(\delta_0)$. 
As in (5.37) and (5.38) of \cite{Sp2} one sees that these data satisfy the bounds \eqref{est:lwp-data} with $2r$ 
instead of $r$, that $\|B_1 \tilde{u}_0\|_{L^\infty(\Sigma)} < \wt{\kappa}/2$, and that the range of $\tilde{u}_0$
is contained in $V_\kappa$.

As a result, Theorem~\ref{thm:local} yields a solution $\tilde{u}\in G^m_\Sigma((t_0,t_1)\times G)$ of \eqref{eq:maxwell}
with data $( \tilde{u}_0,\tilde{f},\tilde{g})$ instead of $(u_0,f,g)$. The proof of this theorem also shows that $\tilde{u}$
is bounded by $R$ in $G^m_\Sigma((t_0,t_1)\times G)$ and thus $\Psi$ maps $B_M(\delta_0)$ into the ball in 
$G^m_\Sigma((t_0,t_1)\times G)$ with  center 0 and radius $R$.  
Moreover, both $\|B_1 u\|_{L^\infty((t_0,t_1) \times \Sigma)}$ and $\|B_1 \tilde{u}\|_{L^\infty((t_0,t_1) \times \Sigma)}$ are smaller 
than $\wt{\kappa}$ by the choice of $\tau$. It follows that  estimate \eqref{est:lwp-dep0} is true for $\tilde{u}$ instead of $\hat u$, 
with time step $\tau$ and constant $\hat{C}$, because $\wt{\kappa}$ satisfies~\eqref{est:hat-kappa}. 

Next take a sequence $(u_{0,n}, f_n, g_n)_n$ in $B_{M_{\chi,\sigma,\zeta,m}(t_0,T')}((u_0,f,g),\delta_0)$ 
which converges to $(u_0,f,g)$ in this space.    Since 
\[\sum\nolimits_{j = 0}^{m-1} \|\partial_t^j f_n(t_0) - \partial_t^j f(t_0)\|_{\cH^{m-j-1}(G)}^2 
     \leq m C'_S \|f_n - f\|_{\cH^m((t_0, T') \times G)}^2 \ \longrightarrow \ 0     \]
    as $n \rightarrow \infty$, estimate~\eqref{est:lwp-dep0}    yields the limit
    \begin{align*}
     \|\Psi(u_{0,n},f_n,g_n) - \Psi(u_0,f,g)\|_{G^{m-1}_\Sigma((t_0,t_1)\times G)} \ \longrightarrow \ 0.
    \end{align*}
Lemma~\ref{lem:cont-dep2}   thus shows that $(\Psi(u_{0,n}, f_n, g_n))_n$ converges to
$\Psi(f,g,u_0)$ in $G^{m}_\Sigma((t_0,t_1) \times G)$.   
We conclude that $\Psi$ is continuous in $(u_0,f,g)$.
Using \eqref{est:lwp-data} and the choice of $\kappa$,
we then find a number $\delta_1 \in (0, \delta_0]$ such that 
    for all data $(\tilde{u}_0,\tilde{f}, \tilde{g} ) \in B_{M_{\chi,\sigma,\zeta,m}(t_0,T')}((u_0,f,g),\delta_1)$
    the solution $\Psi(\tilde{u}_0,\tilde{f}, \tilde{g})$ exists on $[t_0,t_1]$ and 
    satisfies~\eqref{est:lwp-dep0} on $(t_0,t_1)$ and
    \begin{align*}
     &\| \Psi( \tilde{u}_0, \tilde{f}, \tilde{g})\|_{G^m_\Sigma((t_0,t_1) \times G)} \\
     &\quad \leq \| \Psi( \tilde{u}_0, \tilde{f},\tilde{g},) - \Psi(u_0, f,g)\|_{G^m_\Sigma((t_0,t_1) \times G)} 
        + \|\Psi(u_0,f,g)\|_{G^m_\Sigma((t_0,t_1) \times G)} < 2r, \\
     &\|B_1 \tilde{u}(t_1)\|_{L^\infty(\Sigma)} < \wt{\kappa}/2, \\
     &\dist(\ran \Psi(\tilde{u}_0,\tilde{f}, \tilde{g})(t),\partial \mathcal{U}) >  \kappa,
    \end{align*}
    for all $t \in [t_0,t_1]$. In particular, $\Psi(\tilde{u}_0, \tilde{f}, \tilde{g})(t_1)$ 
    satisfies the assumptions of Theorem~\ref{thm:local} with the same parameters as used before.

    \smallskip
   
3) As in the proof of Theorem~5.3 in \cite{Sp2} we can iterate the above argument up to time $t_N\ge T$, arriving at a final 
radius $\delta:=\delta_N$ for the data. In particular,  the final existence time $T_+(m, t_0, \tilde{u}_0,\tilde{f},\tilde{g})$ 
is larger than $T$ if $(\tilde{u}_{0}, \tilde{f}, \tilde{g})$ belongs to $B_M(\delta)$.
Next fix  two tuples $( \tilde{u}_{0,j},\tilde{f}_j, \tilde{g}_j)$ from this ball.
Replacing $u$ by $\Psi(\tilde{u}_{0,2}, \tilde{f}_2, \tilde{g}_2)$ in step I), we deduce from~\eqref{est:lwp-dep0} that
\begin{align}\label{est:lwp-dep}
  &\|\Psi(\tilde{u}_{0,1}, \tilde{f}_1, \tilde{g}_1) - \Psi(\tilde{u}_{0,2},\tilde{f}_2, \tilde{g}_2)\|_{G^{m-1}_\Sigma((t_0,T) \times G)}^2 \\
     &\leq \hat{C}\Big(  \|\tilde{u}_{0,1} - \tilde{u}_{0,2}\|_{\cH^m(G)}^2 
       + \| \tilde{f}_1 - \tilde{f}_2 \|_{\cH^{m-1}((t_0,T) \times G)}^2 
       + \|\tilde{g}_1 - \tilde{g}_2 \|_{\cH^{m-1}((t_0, T) \times \Sigma)}^2 \notag\\
     &\qquad + \sum\nolimits_{j = 0}^{m-1} \|\partial_t^j \tilde{f}_1(t_0) - \partial_t^j \tilde{f}_2(t_0)\|_{\cH^{m-j-1}(G)}^2  \Big), \notag
    \end{align}
    where $\hat{C}= \hat{C}(\chi,\sigma,\zeta, m,2r,R,\tilde{V}_\kappa,T')$.
    This estimate implies~\eqref{est:lwp-lip}. 
    
Finally, take a sequence $(\tilde{u}_{0,n},\tilde{f}_n, \tilde{g}_n)_n$ in $B_M(\delta)$ with limit 
$(\tilde{u}_{0,1}, \tilde{f}_1, \tilde{g}_1)$ in this ball. Inequality \eqref{est:lwp-dep} shows that
 the solutions $\Psi( \tilde{u}_{0,n}, \tilde{f}_n, \tilde{g}_n)$ tend to $\Psi(\tilde{u}_{0,1},\tilde{f}_1, \tilde{g}_1)$ 
    in  $G^{m-1}_\Sigma((t_0,T)\times G)$ as $n \rightarrow \infty$. Lemma~\ref{lem:cont-dep2} thus shows that
    this convergence takes place in $G^{m}_\Sigma((t_0,T)\times G)$. So also part (3) is established.
 \end{proof}

\begin{rem}
Reversing time  and adapting coefficients, data and smallness assumptions accordingly, we can transfer
  the results of Theorem~\ref{thm:lwp} to the negative time direction, cf.\ Remark~3.3 in~\cite{Sp0}.
 \end{rem}

\end{document}